\documentclass[11pt, a4paper]{amsart}
\pdfoutput=1

\usepackage{amscd,latexsym,amsthm,amsfonts,amssymb,amsmath,amsxtra}
\usepackage[mathscr]{eucal}
\usepackage{hyperref}
\usepackage{stmaryrd}
\usepackage{MnSymbol} 
\usepackage{mathtools}
\usepackage{multicol}
\usepackage{enumitem}

\usepackage{tabu}
\usepackage{float}
\newcommand{\e}[1]{{\mathbf{e}\left( #1 \right)}}

\newcommand{\ad}{\operatorname{ad}}
\newcommand{\qtextq}[1]{{\quad \text{#1} \quad}}
\newcommand{\qtext}[1]{{\quad \text{#1}}}
\newcommand{\dash}{\ensuremath{\text{---}}}
\usepackage[dvipsnames]{xcolor}

\renewcommand\theequation{\thesection.\arabic{equation}}

\mathtoolsset{showonlyrefs}

\ifpdf
\usepackage[protrusion=true,expansion=true]{microtype}
\fi

\usepackage{xparse}
\usepackage{mdframed}

\usepackage{datetime2}

\hypersetup{
    pdftitle={Fourier coefficients attached to small automorphic representations of SLn(A)},         %
    pdfauthor={Olof Ahlen, Henrik P. A. Gustafsson, Axel Kleinschmidt, Baiying Liu, and Daniel Persson},      %
    linktocpage=true,           
    colorlinks=true,            
    linkcolor=blue!75!black,    
    citecolor=green!50!black,   
    filecolor=magenta,          
    urlcolor=blue!75!black      
}

\newlist{itemize-small}{itemize}{1}
\setlist[itemize-small]{label=\textbullet,leftmargin=0cm,itemindent=\parindent,labelwidth=\itemindent,labelsep=-1mm,align=left}

\NewDocumentCommand{\intl}{g}{
    \IfNoValueTF{#1}{\int\limits}{\int\limits_{\mathclap{#1}}}
}
\NewDocumentCommand{\suml}{g}{
    \IfNoValueTF{#1}{\sum\limits}{\sum\limits_{\mathclap{#1}}}
}
\NewDocumentCommand{\prodl}{g}{
    \IfNoValueTF{#1}{\prod\limits}{\prod\limits_{\mathclap{#1}}}
}

\newcommand{\iso}{\cong}
\DeclareMathAlphabet{\mathbbold}{U}{bbold}{m}{n}

\DeclareMathOperator{\Mat}{Mat}
\newcommand{\defeq}{\coloneqq}

\DeclarePairedDelimiter\floor{\lfloor}{\rfloor}

\newcommand{\BA}{{\mathbb {A}}}

\newcommand{\BC}{{\mathbb {C}}}

\newcommand{\BN}{{\mathbb {N}}}

\newcommand{\BQ}{{\mathbb {Q}}}
\newcommand{\BR}{{\mathbb {R}}}

\newcommand{\BZ}{{\mathbb {Z}}}

\newcommand{\CF}{{\mathcal {F}}}

\newcommand{\RG}{{\mathrm {G}}}

\newcommand{\GL}{{\mathrm{GL}}}

\renewcommand{\Im}{{\mathrm{Im}}}

\newcommand{\I}{{\mathrm{I}}}

\renewcommand{\Re}{{\mathrm{Re}}}

\newcommand{\SL}{{\mathrm{SL}}}

\newcommand{\SO}{{\mathrm{SO}}}
\newcommand{\Spin}{{\mathrm{Spin}}}

\newcommand{\SU}{{\mathrm{SU}}}

\newcommand{\tr}{{\mathrm{tr}}}

\newcommand{\ol}{\overline}
\newcommand{\ul}{\underline}

\newcommand{\bs}{\backslash}

\newcommand{\lie}[1]{\mathfrak{#1}}
\newcommand{\Oh}{\mathcal{O}}

\newcommand{\diag}{\operatorname{diag}}
\newcommand{\trdiag}{\operatorname{trdiag}}

\newtheorem{mainthm}{Theorem}

\newtheorem{thm}{Theorem}[section]
\newtheorem{cor}[thm]{Corollary}
\newtheorem{lem}[thm]{Lemma}
\newtheorem{prop}[thm]{Proposition}

\newtheorem {ques/conj}[thm]{Question/Conjecture}

\theoremstyle{definition}

\newtheorem{examp}[thm]{Example}
\newtheorem{rmk}[thm]{Remark}

\newcommand{\ads}{\mathbb{A}}
\newcommand{\ints}{\mathbb{Z}}

\makeatletter

\newcommand{\Rmnum}[1]{\expandafter\@slowromancap\romannumeral #1@}
\makeatother

\begin{document}
\renewcommand{\theequation}{\arabic{equation}}
\numberwithin{equation}{section}

\title[Fourier coefficients of small automorphic representations]{Fourier coefficients attached to small automorphic representations of $\SL_n(\mathbb{A})$}

\author[O. Ahl\'en]{Olof Ahl\'en}
\address{Max Planck Institute for Gravitational Physics (Albert Einstein Institute)\\
Am M\"uhlenberg 1\\
14476 Potsdam, Germany}
\email{olof.ahlen@aei.mpg.de}

\author[H. Gustafsson]{Henrik P. A. Gustafsson}
\address{Department of Physics\\
Chalmers University of Technology\\
 SE-412 96 Gothenburg, Sweden}
\email{henrik.gustafsson@chalmers.se}

\author[A. Kleinschmidt]{Axel Kleinschmidt}
\address{Max Planck Institute for Gravitational Physics (Albert Einstein Institute)\\
Am M\"uhlenberg 1\\
14476 Potsdam, Germany
\hfill \newline
\indent Solvay Institutes \\
ULB-Campus Plaine CP231, BE-1050 Brussels, Belgium}
\email{axel.kleinschmidt@aei.mpg.de}

\author[B. Liu]{Baiying Liu}
\address{Department of Mathematics\\
Purdue University\\
150 N. University St\\
West Lafayette, IN, 47907}
\email{liu2053@purdue.edu}

\author[D. Persson]{Daniel Persson}
\address{Department of Mathematical Sciences \\
Chalmers University of Technology\\
 SE-412 96 Gothenburg, Sweden}
\email{daniel.persson@chalmers.se}


\date{\today}


\begin{abstract}
We show that Fourier coefficients of automorphic forms attached to minimal or next-to-minimal automorphic representations of $\SL_n(\mathbb{A})$ are completely determined by certain highly degenerate Whittaker coefficients. We give an explicit formula for the Fourier expansion, analogously to the Piatetski-Shapiro--Shalika formula. In addition, we derive expressions for Fourier coefficients associated to all maximal parabolic subgroups. These results have potential applications for scattering amplitudes in string theory. 
\end{abstract}

\maketitle

\tableofcontents

\newpage

\section{Introduction}

\subsection{Background and motivation}
\label{sec:background}
Let $F$ be a number field and $\ads$ be the associated ring of adeles. 
Let $\RG$ be a reductive algebraic group defined over $F$ and $\pi$ an (irreducible) automorphic representation of $\RG(\ads)$ as defined in \cite{BJ79, FGKP18}.

Fix a Borel subgroup $B$ and let $P \subset \RG$ be a standard parabolic subgroup with Levi decomposition $P=LU$, and let $\psi : U(F)\backslash U(\mathbb{A})\to \mathbb{C}^\times$ be a global unitary character. 
Given any automorphic form $\varphi\in \pi$ one can consider the following function on $\RG(\ads)$:
\begin{equation}
\mathcal{F}_{U}(\varphi, \psi; g)=\intl_{U(F)\backslash U(\mathbb{A})} \varphi(ug)\psi^{-1}(u) \, du\,.
\end{equation}
This can be viewed as a Fourier coefficient of the automorphic form $\varphi$ with respect to the unipotent subgroup $U$. Fourier coefficients of automorphic forms carry a wealth of arithmetic and representation-theoretic information. For example, in the case of classical modular forms on the upper half-plane, Fourier coefficients are well-known to encode information about the count of rational points on elliptic curves. On the other hand, for higher rank Lie groups  their arithmetic content is not always transparent, but they always encode important representation-theoretic information. Langlands showed that the constant terms in the Fourier expansion of Eisenstein series provide a   source for automorphic $L$-functions \cite{L67}, and Shahidi extended this method (now called the Langlands-Shahidi method) to include also the non-constant Fourier coefficients \cite{Sha78,Sha81}.

Theta correspondences provide realizations of Langlands functorial transfer between automorphic representations $\pi$ and $\pi'$ of two different groups $\RG$ and $\RG'$. In this context automorphic forms attached to {\it minimal} automorphic representations play a key role \cite{G06}. The wave front set of a minimal representation $\pi_\text{min}$ of a group $\RG$ is the closure of the smallest non-trivial nilpotent coadjoint orbit $\mathcal{O}_\text{min}$ of $\RG$ \cite{J76,KS90}. The automorphic realizations of minimal representations are characterized by having very few non-vanishing Fourier coefficients \cite{GRS97}. Conversely, the method of descent \cite{GRS11} can be viewed as an inverse to the functorial lifting, in which an automorphic representation of a general linear group $\GL_n$ is transferred to a representation of a smaller classical group $\RG$. Also in this case do Fourier coefficients of small representations enter in a crucial way. 

In general it is a difficult problem to obtain explicit formulas for Fourier coefficients for higher rank groups, let alone settle the question of whether an automorphic form $\varphi$ can be reconstructed from only a subset of its Fourier coefficients.
For cusp forms on $\GL_n$ this is possible due to the Piatetski-Shapiro--Shalika formula \cite{S74,PS79} that allows to reconstruct $\varphi$ from its Whittaker coefficients; i.e. the Fourier coefficients with respect to the unipotent radical $N$ of the Borel subgroup $B \subset \RG$. These coefficients are sums of Eulerian Whittaker coefficients on subgroups of $\RG$, and their non-archimedean parts can be obtained from the Casselman--Shalika formula \cite{Sh76,CS80} as described in \cite{FGKP18}. However, even if this gives us complete control of the Fourier expansion with respect to $N$ it does not automatically give us a way of calculating an arbitrary Fourier coefficient $\mathcal{F}_{U}(\varphi, \psi; g)$ with respect to some other unipotent subgroup $U$. Such coefficients play an important role in the construction of $L$-functions, and also carry information about non-perturbative effects in string theory as described in section~\ref{sec:string}.

Expanding upon the classic results of \cite{GRS97}, Miller and Sahi proved in \cite{MS12} that for automorphic forms $\varphi$ attached to a minimal representation $\pi_\text{min}$ of $E_6$ and $E_7$, any Fourier coefficient $\mathcal{F}_{U}(\varphi, \psi; g)$ is completely determined by maximally degenerate Whittaker coefficients of the form
\begin{align}
\label{maxdeg}
\intl_{N(F)\backslash N(\mathbb{A})} \varphi(ng)\psi_\alpha (n)^{-1} \,dn\, ,
\end{align}
where $\psi_\alpha$ is non-trivial only on the one-parameter subgroup of $N$ corresponding to the simple root $\alpha$. This result maybe viewed as a global version of the classic results of Moeglin-Waldspurger in the non-archimedean setting \cite{MW87}, and Matumoto in the archimedean setting \cite{Mat87}.

For the special cases of $\SL_3$ and $\SL_4$ the Miller--Sahi results were generalized in \cite{GKP16} (following related results in \cite{FKP14}) to automorphic forms attached to a next-to-minimal automorphic representation $\pi_\text{ntm}$. It was shown that any Fourier coefficient is completely determined by \eqref{maxdeg} and coefficients of the following form
\begin{equation} 
\int\limits_{N(F)\backslash N(\mathbb{A})} \varphi(ng)\psi_{\alpha, \beta} (n)^{-1} \,dn\,, 
\label{ntmdeg}
\end{equation}
where $\psi_{\alpha, \beta}$ is only supported on strongly orthogonal pairs of simple roots $(\alpha, \beta)$ which here reduces to that $[E_\alpha, E_\beta]=0$~\cite{Knapp}. The main goal of the present paper is to use the techniques of \cite{JL13,JLS16,GGS15}, in particular the notion of \emph{Whittaker pair}, to extend the above results to all of $\SL_n$. 

\subsection{Summary of results}

We now summarize our main results. In the rest of this paper we will consider $\SL_n$ for $n \geq 5$ where we have fixed a Borel subgroup with the unipotent radical $N$. 
Let also $T$ be the diagonal elements of $\SL_n(F)$ and, for a character $\psi_0$ on $N$, let $T_{\psi_0}$ be the stabilizer of $\psi_0$ under the action $[h.\psi_0](n) = \psi_0(h n h^{-1})$ for $h \in T$. \label{tpsi0}

Define
\begin{equation}
    \label{eq:Gamma-i}
    \Gamma_i(\psi_0) \defeq
    \begin{cases}
        (\SL_{n-i}(F))_{\hat Y} \bs \SL_{n-i}(F) & 1 \leq i \leq n-2 \\
        (T_{\psi_0} \cap T_{\psi_{\alpha_{n-1}}}) \bs T_{\psi_0} & i = n-1\, , 
    \end{cases}
\end{equation}
where $(\SL_{n-i}(F))_{\hat Y}$ is the stabilizer of $\hat Y = {}^t(1, 0, 0, \ldots, 0) \in \Mat_{(n-i)\times 1}(F)$ and consists of elements 
$\begin{psmallmatrix}
    1 & \xi \\
    0 & h
\end{psmallmatrix}$, with $h \in \SL_{n-i-1}(F)$ and $\xi \in \Mat_{1 \times (n-i-1)}(F)$.
When $\psi_0 = 1$ we write $\Gamma_i(1)$ as $\Gamma_i$.

Similarly, let $(\SL_j(F))_{\hat X}$ be the stabilizer of $\hat X = (0, \ldots, 0, 1) \in \Mat_{1\times j}(F)$ with respect to multiplication on the right, $\psi_0$ a character on $N$, and define
\begin{equation}
    \label{eq:Lambda-j}
    \Lambda_j(\psi_0) \defeq
    \begin{cases}
        (\SL_j(F))_{\hat X} \bs \SL_j(F) & 2 < j \leq n \\
        (T_{\psi_0} \cap T_{\psi_{\alpha_1}}) \bs T_{\psi_0} & j = 2 \, ,
    \end{cases}
\end{equation}
where, again, we denote $\Lambda_j(1) = \Lambda_j$.

Define also the embeddings $\iota, \hat \iota : \SL_{n-i} \to \SL_n$ for any $0 \leq i \leq n-1$ as
\begin{equation}
   \label{eq:iota}
    \iota(\gamma) =
    \begin{pmatrix}
        I_{i} & 0 \\
        0 & \gamma
    \end{pmatrix} \qquad
    \hat\iota(\gamma) =
    \begin{pmatrix}
        \gamma & 0 \\
        0 & I_{i}
    \end{pmatrix}
    \, ,
\end{equation}
where we for brevity suppress their dependence on $i$. Note that for $i = 0$, they are just the identity maps for $\SL_n$.

The following theorem expands an automorphic form $\varphi$ attached to a small automorphic representation of $\SL_n$ in terms of highly degenerate Whittaker coefficients similar to how cusp forms on $\GL_n$ can be expanded in terms of Whittaker coefficients with the Piatetski-Shapiro--Shalika formula \cite{S74, PS79}. Expansion of non-cuspidal automorphic forms on $\GL_n$ in terms of Whittaker coefficients were discussed in~\cite{Yukie,JL13}.

\begin{mainthm}
    \label{thm:varphi}
    Let $\pi$ be a minimal or next-to-minimal irreducible automorphic representation of $\SL_n(\BA)$, and let $\varphi \in \pi$.
    \begin{enumerate}[label=\textnormal{(\roman*)}, leftmargin=0cm,itemindent=1.75\parindent,labelwidth=\itemindent,labelsep=0mm,align=left]
        \item \label{itm:varphi-min} If $\pi=\pi_{min}$, then $\varphi$ has the expansion
            \begin{equation}
                \varphi(g) = \intl_{N(F)\backslash N(\mathbb{A})} \varphi(ng) \, dn + \sum_{i=1}^{n-1} \sum_{\gamma \in \Gamma_{i}} \, \intl_{N(F)\backslash N(\mathbb{A})} \varphi(n \iota(\gamma) g) \psi^{-1}_{\alpha_i}(n) \, dn \, .
            \end{equation}
        \item \label{itm:varphi-ntm} If $\pi=\pi_{ntm}$,  then $\varphi$ has the expansion
            \begin{multline}
                \varphi(g) = \intl_{N(F) \bs N(\ads)} \varphi(vg) \, dv + \sum_{i=1}^{n-1} \sum_{\gamma \in \Gamma_i} \intl_{N(F)\bs N(\ads)} \varphi(v \iota(\gamma) g) \psi^{-1}_{\alpha_i}(v) \, dv +{} \\
                + \sum_{j=1}^{n-3}\sum_{i=j+2}^{n-1} \sum_{\substack{\gamma_i \in \Gamma_i(\psi_{\alpha_j}\!) \\ \gamma_j \in \Gamma_j}} \intl_{N(F)\bs N(\ads)} \varphi(v \iota(\gamma_i) \iota(\gamma_j) g) \psi^{-1}_{\alpha_j, \alpha_i} (v) \, dv \,.
            \end{multline}
    \end{enumerate}
\end{mainthm}

Note that the Whittaker coefficients in the last sum of case \ref{itm:varphi-ntm} have characters supported on two strongly orthogonal (or commuting) simple roots. 
As mentioned in section~\ref{sec:background} and further described in \cite{FGKP18}, the Whittaker coefficients are sums of Eulerian Whittaker coefficients on smaller subgroups $\SL_n$, whose non-archimedean parts can be computed by the Casselman--Shalika formula \cite{S74, CS80}. 
The more degenerate a Whittaker coefficient is the smaller the subgroup we need to consider (and on which character becomes generic). Thus, maximally degenerate Whittaker coefficients, and the ones with characters supported on two commuting simple roots become particularly simple and are, in principle, one, or a product of two, known $\SL_2$ Whittaker coefficients respectively.

Next, we consider Fourier coefficients on maximal parabolic subgroups. 
Let $P_m$ the maximal parabolic subgroup of $SL_n$ with respect to the simple root $\alpha_m$ and let $U = U_m$ be the unipotent radical and $L_m$ be the corresponding Levi subgroup which stabilizes $U_m$ under conjugation. 
For an element $l\in L_m(F)$ and a character $\psi_U$ on $U_m$ we obtain another character $\psi_U^l$ by conjugation as
\begin{equation}
    \psi_U^l(u) = \psi_U(lul^{-1}) \, .
\end{equation} 

Fourier coefficients $\mathcal{F}_U$ with conjugated characters are related by $l$-translates of their arguments
\begin{align}
    \label{eq:Fourier-L-conjugation}
    \mathcal{F}_U(\varphi, \psi_U^l; g) &= \intl_{U(F) \bs U(\ads)} \varphi(ug) \psi_U^{-1}(lul^{-1}) \, du \\
    &= \intl_{U(F) \bs U(\ads)} \varphi(l^{-1}u'lg) \psi_U^{-1}(u') \, du' = \mathcal{F}_U(\varphi, \psi_U; lg) \, ,
\end{align}
where we have first made the variable substitution $u' = lul^{-1}$ and then used the automorphic invariance since $l \in L_m(F)$. 
This means that we only need to compute the Fourier coefficients of one character per $L_m(F)$-orbit.

We show in section~\ref{sec:fourier} that a character can be parametrized by an element $y \in \lie g$ by \eqref{eq:character} denoted by $\psi_y$, which, under conjugation, satisfies $\psi_y^l = \psi_{l^{-1} y l}$ according to \eqref{eq:character}.

In section~\ref{sec:thmB} and appendix~\ref{app:levi-orbits}, we describe these orbits  following~\cite{N11} and construct standard characters $\psi_{y(Y_r(d))}$ on $U_m$ based on anti-diagonal $(n-m) \times m$ rank $r$ matrices $Y_r(d)$, where $d \in F^\times/(F^\times)^2$ for $n = 2r = 2m$ and $d = 1$ otherwise (in which case we suppress the $d$), and $y(Y_r(d))$ is defined as
\begin{equation}
   y(Y_r(d)) =
   \begin{pmatrix}
        0_m & 0\\
        Y_r(d) & 0_{n-m}
    \end{pmatrix}\,.
\end{equation}

Let $\pi$ be a minimal or next-to-minimal automorphic representations, and
\begin{equation}
    r_\pi = 
    \begin{cases}
        1 & \text{if $\pi$ is a minimal automorphic representation} \\
        2 & \text{if $\pi$ is a next-to-minimal automorphic representation.} 
    \end{cases}
\end{equation}

We will show that only the characters with rank $r \leq r_\pi \leq 2$ give non-vanishing Fourier coefficients.
Let us briefly define the characters with rank $r \leq 2$ which will be used in the next theorem, postponing a more general definition to section~\ref{sec:thmB}.
The rank zero character is the trivial character $\psi_{y(Y_0)} = 1$ and the corresponding Fourier coefficient has been computed in \cite{MW95} as reviewed in \cite{FGKP18}.
The rank one character is $\psi_{y(Y_1)} = \psi_{\alpha_m}$ and the rank two character can be defined as follows
\begin{equation}
    \label{eq:psi-Y2}
    \psi_{y(Y_2)}(u) = \psi(u_{m,m+1} + u_{m-1, m+2}) \qquad u \in U_m(\ads) \, .
\end{equation}

The following theorem, together with the known constant term, then allows us to compute any Fourier coefficient with respect to the unipotent radical of a maximal parabolic subgroups for automorphic forms attached to minimal and next-to-minimal automorphic representations in terms of Whittaker coefficients.

\begin{mainthm}
    \label{thm:max-parabolic}
    Let $\pi$ be a minimal or next-to-minimal irreducible automorphic representation of $\SL_n(\ads)$, and let $r_\pi$ be $1$ or $2$ respectively (which denotes the maximal rank of the character matrix $Y_r$).
    Let also, $\varphi \in \pi$, $P_m$ be the maximal parabolic subgroup described above with its associated subgroups $U\equiv U_m$ and $L_m$, and let $\psi_U$ be a non-trivial character on $U_m$ with Fourier coefficient 
    $$\mathcal{F}_U(\varphi, \psi_U; g) = \int_{U_m(F)\bs U_m(\ads)} \varphi(ug) \psi_U^{-1}(u) \, du\, .$$ 
    Then, there exists an element $l \in L_m(F)$ such that 
    $$\mathcal{F}_U(\varphi, \psi_U; g) = \mathcal{F}_U(\varphi, \psi_{y(Y_r(d))}; lg)$$
     for some standard character $\psi_{y(Y_r(d))}$ described above and in the proof.

    Additionally, all $\mathcal{F}_U(\varphi, \psi_{y(Y_r(d))}; lg)$ for $r > r_\pi$ vanish identically.
    The remaining (non-constant) coefficients can be expressed in terms of Whittaker coefficients on $N$ as follows.
    \begin{enumerate}[label=\textnormal{(\roman*)}, leftmargin=0cm,itemindent=1.75\parindent,labelwidth=\itemindent,labelsep=0mm,align=left]
        \item \label{itm:max-parabolic-min} If $\pi = \pi_\text{min}$:
            \begin{flalign}
                \qquad \mathcal{F}_U(\varphi, \psi_{y(Y_1)}; g) &= \intl_{N(F)\bs N(\ads)} \varphi(ng) \psi_{\alpha_m}^{-1}(n) \, dn \, . &
            \end{flalign}
        \item \label{itm:max-parabolic-ntm-rank1} If $\pi = \pi_\text{ntm}$:
            \begin{flalign}
                \qquad \mathcal{F}_U(\varphi, \psi_{y(Y_1)}; g) &=
                \begin{multlined}[t]
                    \intl_{[N]} \varphi(ng) \psi_{\alpha_m}^{-1}(n) \, dn +{} \\ 
                + \sum_{j=1}^{m-2} \sum_{\gamma \in \Lambda_j(\psi_{\alpha_m}\!)}\, \intl_{[N]} \varphi(n\hat\iota(\gamma) g) \psi_{\alpha_j, \alpha_m}^{-1}(n) \, dn +{} \\ 
                + \sum_{i=m+2}^{n-1} \sum_{\gamma \in \Gamma_i(\psi_{\alpha_m}\!)} \, \intl_{[N]} \varphi(n \iota(\gamma) g) \psi_{\alpha_m, \alpha_i}^{-1}(n) \, dn \, . 
            \end{multlined}&
        \end{flalign}
        \item \label{itm:max-parabolic-ntm-rank2} If $\pi = \pi_\text{ntm}$:
            \begin{flalign}
                \qquad \mathcal{F}_U(\varphi, \psi_{y(Y_2)}; g) &= \intl_{C(\BA)}\intl_{N(F)\bs N(\ads)} \varphi (n\omega cg)\psi_{\alpha_1,\alpha_3}^{-1}(n) \,dn\,dc\, , &
            \end{flalign}
            where $\omega$ is the Weyl element mapping the torus elements 
            $$(t_1, t_2, \ldots, t_n) \mapsto (t_{m-1}, t_{m+2}, t_m, t_{m+1}, t_1, t_2, \ldots, t_{m-2}, t_{m+3}, t_{m+4}, \ldots, t_n)\,,$$
            and the subgroup $C$ of $U_m$ will be detailed in the proof in section~\ref{sec:thmB}.
    \end{enumerate}
\end{mainthm}

As described in detail in section~\ref{sec:fourier}, $F$-rational nilpotent orbits of $\SL_n$ are characterized by $(\ul p, d)$ where $\ul p$ is a partition of $n$ and $d \in F^\times/(F^\times)^k$ with $k = \gcd(\ul p)$. 
If $k = 1$ we will often suppress the extra $d = 1$ and only write out the partition.

There we will also see that, for each orbit, there are natural choices of unipotent subgroups and characters related by conjugations with elements $\gamma \in \SL_n(F)$ and the corresponding Fourier coefficients \eqref{eq:orbit-coefficient} are related by $\gamma$-translates of their arguments. 

The orbits may be partially ordered and the minimal and next-to-minimal orbits are described by the partitions $[21^{n-2}]$ and $[2^21^{n-4}]$, respectively. 
Besides the trivial partition, these are the only partitions whose associated Fourier coefficients are non-vanishing for $\varphi$ in a minimal or next-to-minimal irreducible automorphic representation. 
In section~\ref{sec:orbit-coefficients} we choose standard representatives for these orbits and specify the associated standard Fourier coefficients which we denote by $\mathcal{F}^{[211\ldots]}$ and $\mathcal{F}^{[221\ldots]}$.
For $n\geq 5$, we have that the trivial, minimal and next-to-minimal orbit all have $k=1$.

The following theorems express these standard Fourier coefficients associated with the two partitions above in terms of Fourier coefficients on maximal parabolic subgroups that, in turn, were written in terms of Whittaker coefficients in theorem~\ref{thm:max-parabolic}.

\begin{mainthm}
    \label{thm:min-coeff} 
    Let $\pi$ be an irreducible automorphic representation of $\SL_n(\ads)$, $\varphi \in \pi$ and $Y = \Pi_{i=3}^n X_{e_i - e_2}$. 
    Then,
    \begin{equation}
        \mathcal{F}^{[211\ldots]}(\varphi; g) = \sum_{y \in Y(F)} \, \intl_{U_1(F)\bs U_1(\ads)} \varphi(u y^{-1} g) \psi_{\alpha_1}^{-1}(u) \, du  \, , 
    \end{equation}
    where $U_1$ is the unipotent radical of $P_1$ consisting of the first row of $N$. The Fourier coefficient $\mathcal{F}^{[211\ldots]}$ is for a particular standard choice of orbit representative detailed in the proof; all other choices are related simply by $\SL_n(F)$ translation.
\end{mainthm}

\begin{mainthm}
    \label{thm:ntm-coeff}
    Let $\pi$ be an irreducible automorphic representation of $\SL_n(\ads)$, $\varphi \in \pi$, $Y' = \prod_{i=5}^{n} X_{e_i-e_4} \prod_{i=5}^{n} X_{e_i-e_3}$ and $\omega$ be the Weyl element mapping the torus elements 
    $$(t_1, t_2, \ldots, t_n) \mapsto (t_1, t_3, t_4, t_2, t_5, t_6, \ldots, t_n)\, .$$
    Then,
    \begin{equation}
        \mathcal{F}^{[221\ldots]}(\varphi; g) = \sum_{y \in Y'(F)} \, \intl_{U_2(F)\bs U_2(\ads)} \varphi(u y^{-1} \omega g) \psi_{y(Y_2)}^{-1}(u) \, du \, ,
    \end{equation}
    where $U_2$ is the unipotent radical of $P_2$ consisting of the first two rows of $N$ and $\psi_{y(Y_2)}$ is defined in \eqref{eq:psi-Y2} with $m = 2$. The Fourier coefficient $\mathcal{F}^{[2^21\ldots]}$ is for a particular standard choice of orbit representative detailed in the proof; all other choices are related simply by $\SL_n(F)$ translation.
\end{mainthm}

\subsection{Applications in string theory}
\label{sec:string}
String theory is a quantum theory of gravity describing maps $X: \Sigma \to M$, where $\Sigma$ is a Riemann surface (the string worldsheet) and $M$ is a ten-dimensional pseudo-Riemannian manifold (spacetime). Its low-energy limit is a supersymmetric extension of Einstein's theory of gravity in 10 dimensions coupled to additional matter in the form of scalar fields $\Phi : M\to \mathbb{C}$ and differential forms on spacetime $M$. Our main focus here will be the scalar fields. The scalar fields parametrize the space of string theory vacua, i.e. the moduli space $\mathcal{M}$. 

To make contact with a lower-dimensional world, one choice is to decompose spacetime into 
$$M=\mathbb{R}^{1,9-n}\times T^n\,,$$
 where $\mathbb{R}^{1,9-n}$ is the flat Minkowski space in $10-n$ dimensions and $T^n$  is an $n$-dimensional torus. In the limit when the size of the torus is small, the physics looks effectively $(10-n)$-dimensional and one says that the theory has been \emph{compactified}. As the size of the torus is  increased the moduli space $\mathcal{M}$ gets larger and larger due to an increased number of scalar fields $\Phi$. The moduli space for this toroidal compactification is always of the form 
\begin{equation}
\mathcal{M}=\mathrm{G}(\mathbb{Z})\backslash \mathrm{G}(\mathbb{R})/K\,,
\end{equation}
where $\mathrm{G}(\mathbb{R})$ is a semi-simple Lie group in its split real form, $K$  its maximal compact subgroup and $\mathrm{G}(\mathbb{Z})$ an arithmetic subgroup. The group $\mathrm{G}(\mathbb{Z})$ is known as the \emph{U-duality group} and is a symmetry of the full quantum string theory. The extreme case when $n=0$, i.e. for no compactification, the moduli space is given by 
$$\mathcal{M}=\SL_2(\mathbb{Z})\backslash \SL_2(\mathbb{R})/\mathrm{SO}_2\, .$$
 Another extreme case is $n=6$, corresponding to four space-time dimensions, for which the moduli space is given by \cite{HT95}
\begin{equation}
\mathcal{M}= \mathrm{E}_7(\mathbb{Z})\backslash \mathrm{E}_7(\mathbb{R})/(\mathrm{SU}_8/\mathbb{Z}_2)\,.
\end{equation}
Here $\mathrm{E}_7(\mathbb{R})$ is the split real form and $\mathrm{E}_7(\mathbb{Z})$ its Chevalley group of integer points. The sequence of groups in between are obtained by successively removing nodes from the $\mathrm{E}_7$ Dynkin diagram; see table \ref{tab:duality} for the complete list.

Constraints from $U$-duality and supersymmetry ensure that certain quantum corrections to Einstein's gravitational theory involve functions $f : \mathcal{M} \to \mathbb{C}$ that must be eigenfunctions of the ring of $\mathrm{G}(\mathbb{R})$-invariant differential operators. In particular they are eigenfunctions of the Laplacian on $\mathrm{G}(\mathbb{R})/K$ with  specific eigenvalues. In addition, they must have  well-behaved growth properties in certain limits corresponding to `cusps' of $\mathcal{M}$. Such quantum corrections are therefore controlled by automorphic forms on $\mathcal{M}$.

It turns out that the relevant automorphic forms are very special and are precisely those attached to a minimal and next-to-minimal automorphic representation of the groups $\mathrm{G}$ \cite{GMRV10,P10,GMV15}. The Fourier coefficients of such automorphic forms therefore have a direct physical interpretation: the constant terms encode perturbative quantum corrections, while the non-constant terms correspond to non-perturbative, instanton, effects \cite{FK12,FKP14,BV14,BV15a,BV15b,BCP17a,BP17,BCP17b}. For a recent book on automorphic representations and the connection with string theory, see \cite{FGKP18}.

Fourier coefficients with respect to different choices of parabolic subgroups $P\subset \mathrm{G}$ correspond to different limits in string theory, and reveal different types of effects. The ones of main interest are certain maximal parabolic subgroups. Let $P_\alpha= L_\alpha U_\alpha$ denote the maximal parabolic whose Levi subgroup is $L_\alpha = M_\alpha\times \GL_1$, where $M_\alpha$ is obtained by removing the node in the Dynkin diagram of $\mathrm{G}$ corresponding to the simple root $\alpha$. There are three types of maximal parabolics of main interest  in string theory (the numbering of nodes are according to the Bourbaki convention of the exceptional Lie algebras):
\begin{itemize}

\item  $P_{\alpha_1}$: this is the \emph{perturbative, or string theory, limit} where the Levi is of orthogonal type $M_{\alpha_1}=\mathrm{D}_{n}$;

\item $P_{\alpha_2}$: this is the \emph{M-theory limit} where the Levi is of type $M_{\alpha_2}=\mathrm{A}_{n}$;

\item $P_{\alpha_{n+1}}$: this is the \emph{decompactification limit} where the Levi is of exceptional type $M_{\alpha_{n+1}}=\mathrm{E}_n$ (for $n<6$ these are strictly speaking not exceptional, but given by table \ref{tab:duality}).
\end{itemize}

Theorem \ref{thm:max-parabolic}, together with its counterpart in \cite{GKP16}, then provides explicit results for the Fourier coefficients of automorphic forms in all these parabolics for the cases $n=2$ or $n=3$ when the symmetry groups are $\SL_2\times \SL_3$ or $\SL_5$, respectively. The case of $\SL_5$  will be treated in detail in section \ref{sec:sl5}.

\begin{table}[t!hb]
\centering
\caption{\label{tab:duality} List of U-duality groups in compactifications of (type IIB) string theory on $T^n$.}
\begin{tabular}{|c|c|c|c|}
\hline
$n$ & $\mathrm{G}(\mathbb{R})$ & $K(\mathbb{R})$ & $\mathrm{G}(\mathbb{Z})$ \\
\hline
$0$ & $\SL_2(\mathbb{R})$ & $\SO_2$ & $\SL_2(\mathbb{Z})$ \\
$1$ & $\GL_2(\mathbb{R})$ & $\SO_2$ & $\SL_2(\mathbb{Z})$ \\
$2$ & $\SL_2(\mathbb{R})\times \SL_3(\mathbb{R})$ & $\SO_2\times \SO_2$ & $\SL_3(\mathbb{Z})\times \SL_2(\mathbb{Z})$\\
$3$ & $\SL_5(\mathbb{R})$ & $\SO_5$ & $\SL_5(\mathbb{Z})$ \\
$4$ & $\Spin_{5,5}(\mathbb{R})$ & $(\Spin_5\times \Spin_5)/\mathbb{Z}_2$ & $\Spin_{5,5}(\mathbb{Z})$\\
$5$ & $\mathrm{E}_6(\mathbb{R})$ & $\mathrm{USp}_8/\mathbb{Z}_2$ & $\mathrm{E}_6(\mathbb{Z})$ \\
$6$ & $\mathrm{E}_7(\mathbb{R})$ & $\SU_8/\mathbb{Z}_2$ & $\mathrm{E}_7(\mathbb{Z})$\\
$7$ & $\mathrm{E}_8(\mathbb{R})$ & $\Spin_{16}/\mathbb{Z}_2$ & $\mathrm{E}_8(\mathbb{Z})$\\
\hline
\end{tabular}
\end{table}

\subsection*{Acknowledgements}

We have greatly benefitted from many discussions with Dmitry Gourevitch, Joseph Hundley, Stephen D. Miller  and Siddhartha Sahi. We gratefully acknowledge support from the Simons Center for Geometry and Physics, Stony Brook University during the program on ``Automorphic forms, mock modular forms and string theory'' in the fall of 2016 during which part of the research for this paper was performed. The fourth named author is partially supported by NSF grant DMS-1702218 and by a start-up fund from the Department of Mathematics at Purdue University.

\section{Nilpotent orbits and Fourier coefficients}
\label{sec:fourier}

In this section, first, we introduce Whittaker pairs and nilpotent orbits with their associated Fourier coefficients following~\cite{GGS15}, which is slightly more general and easier to use than the one given in~\cite{G06}. 
Then we recall the parametrization of $F$-rational nilpotent orbits of $\SL_n$ in terms of partitions of $n$ from \cite{N11} and a lemma for exchanging roots in Fourier integrals from \cite{GRS11}.

As before, let $F$ be a number field, $\ads$ be the adele ring of $F$ and fix a non-trivial additive character $\psi$ on $F \bs \BA$. 
Let also $\RG$ be a reductive group defined over $F$, or a central extension of finite degree, and let $\lie g$ be the Lie algebra of $\RG(F)$. For a semi-simple element $s \in \lie g$, let $\lie g^s_i$ be defined as the eigenspace of $s$ in $\lie g$ with eigenvalue $i$ under the adjoint action decomposing $\lie g$ to a direct sum of eigenspaces over different eigenvalues. For any $r \in \BQ$, we further define $\lie g^s_{\geq r} = \oplus_{r' \geq r} \lie g^s_{r'}$ and similarly for other inequality relations. For an element $X \in \lie g$, we will also denote the centralizer of $X$ in $\lie g$ as
\begin{equation}
\lie g_X = \left\{ x \in \lie g \,\middle|\, \left[x,X\right]=0\right\}\,.
\end{equation}

Furthermore, a semi-simple element $s$ is called \emph{rational semi-simple} if all of its eigenvalues under the adjoint action on $\lie g$ are in $\BQ$. For such a rational semi-simple element $s$ and a non-trivial nilpotent element $u \in \lie g^s_{-2}$ we call the ordered pair $(s, u)$ a \emph{Whittaker pair}. If, for such a pair, $s$ is also called a \emph{neutral element} for $u$ or $(s, u)$ a \emph{neutral pair} if the map $\lie g^s_0 \to \lie g^s_{-2} : X \mapsto [X,u]$ is surjective or, equivalently \cite[Lemma 2.2.1]{GGS15}, $s \in \Im(\operatorname{ad}(u))$.

An \emph{$\lie{sl}_2$-triple} is an ordered triple $(u, s, v)$ of elements in $\lie g$ that satisfy the standard commutation relations for $\lie{sl}_2$, 
\begin{equation}
    [s, v] = 2v \qquad [s, u] = - 2u \qquad [v, u] = s\,,
\end{equation}
where $u$ is called the nil-negative element, $v$ is called the nil-positive element and $s$ is a neutral element for $u$. 
We have, from \cite[Lemma 2.2.1]{GGS15}, that a Whittaker pair $(s,u)$ comes from an $\lie{sl}_2$-triple $(u,s,v)$ if and only if $s$ is a neutral element for $u$.

By the Jacobson--Morozov theorem, there exists an $\lie{sl}_2$ triple for any nilpotent element $u \in \lie g$. 
Moreover, the $\RG$-conjugacy classes of $\lie{sl}_2$-triples are one-to-one with the nilpotent orbits $\mathcal{O}_X = \{gXg^{-1} \mid g \in \RG(F)\}$ in $\lie g$ \cite{CollingwoodMcGovern}.  

We will now construct the Fourier coefficient that is associated to a Whittaker pair $(s, u)$. The pair defines a unipotent subgroup $N_s$ and a character $\psi_u$ on $N_s$ as follows. Following \cite[Lemma 3.2.6]{GGS15}, let
\begin{equation}
    \lie{n}_s = \lie{g}^s_{>1} \oplus \lie{g}^s_1 \cap \lie{g}_u \, ,
\end{equation}
which is a nilpotent subalgebra of $\lie g$, and define $N_s = \exp(\lie n_s)$ as the corresponding unipotent subgroup of $\RG$. Then $\psi_u$, defined by
\begin{equation}
    \psi_u(n) = \psi(\langle u, \log(n)\rangle) \qquad n \in N_s(\ads)\,,
\end{equation}
is a character on $N_s(\ads)$ where $\langle \cdot, \cdot \rangle$ is the Killing form.

Note that if the Whittaker pair $(s,u)$ comes from an $\mathfrak{sl}_2$-triple $(u,s,v)$, then, by $\lie{sl}_2$ representation theory, $\ad(s)$ has integer eigenvalues with a graded decomposition of the Lie algebra $\lie g = \bigoplus_{i \in \ints} \lie{g}^s_i$ and $\lie g_u \subset \bigoplus_{i \leq 0} \lie{g}^s_i$ \cite{CollingwoodMcGovern},
and thus, 
\begin{equation}
    \label{eq:neutral-ns}
    \mathfrak{n}_s=\mathfrak{g}^s_{\geq 2} \qquad (\text{for neutral } s).
\end{equation}

Let $\pi$ be an automorphic representation of $\RG(\BA)$ and $\varphi$ an automorphic form attached to $\pi$.
The Fourier coefficient associated with a Whittaker pair $(s, u)$ is
\begin{equation}
    \label{fc}
    \CF_{s,u}(\varphi)(g) = \intl_{N_{s}(F) \bs N_{s}(\BA)} \varphi(ng){\psi}^{-1}_u(n)dn, \quad g \in \RG(\BA) \, ,
\end{equation}
and let $\CF_{s,u}(\pi)=\{\CF_{s,u}(\varphi) \mid \varphi\in \pi\}$. 
For convenience, we introduce the following notation for a unipotent subgroup $U$
\begin{equation}
    [U] = U(F) \bs U(\ads) \, .    
\end{equation}

Consider the Fourier coefficient associated with a neutral Whittaker pair $(s, u)$, and let $(s', u') = (\gamma s\gamma^{-1}, \gamma u \gamma^{-1})$ which is also neutral for any $\gamma\in G(F)$. 
Because of the invariance of the Killing form we have that $\psi_{u'}(n') = \psi_u(\gamma^{-1} n' \gamma)$ where $n' \in [N_{s'}]$, and because of \eqref{eq:neutral-ns} we have that $N_{s'} = \gamma N_s \gamma^{-1}$.
Thus, with a variable substitution $n' = \gamma n \gamma^{-1}$,
\begin{equation}
    \label{eq:orbit-coefficient}
    \begin{split}
        \CF_{s',u'}(\varphi)(g) &= \intl_{[\gamma N_{s'} \gamma^{-1}]} \varphi(n'g) \psi^{-1}_u(\gamma^{-1} n' \gamma) \, dn \\
        &= \intl_{[N_s]} \varphi(\gamma n \gamma^{-1} g)  \psi_u^{-1}(n) \, dn = \CF_{s,u}(\varphi)(\gamma^{-1}g)\,,
    \end{split}
\end{equation}
using the automorphic invariance of $\varphi$. 
Note the resemblance with \eqref{eq:Fourier-L-conjugation} where we made a conjugation keeping $N_s$ invariant.
In particular, \eqref{eq:orbit-coefficient} means that if $\CF_{s,u}$ vanishes identically then so do all Fourier coefficients associated to neutral Whittaker pairs $(s',u')$ where $u' \in \Oh_u$.
For an $F$-rational nilpotent orbit $\Oh$, we say that the coefficients $\CF_{s,u}$ with neutral $s$ and $u \in \Oh$ are Fourier coefficients attached to the nilpotent orbit $\Oh$.

We define the \emph{(global) wave-front set} $\mathcal{WF}(\pi)$ of an automorphic representation $\pi$ of $\RG(\ads)$ as the set of nilpotent orbits $\mathcal{O}$ such that $\mathcal{F}_{s,u}(\pi)$ is non-zero, for some (and therefore all) neutral Whittaker pairs $(s,u)$ with $u \in \mathcal{O}$.
Note that nilpotent orbits can be partially ordered with respect to the inclusion of Zariski closures $\Oh' \leq \Oh$ if $\ol{\Oh'} \subseteq \ol{\Oh}$.

We recall \cite[Theorem C]{GGS15} as follows. 

\begin{thm}[Theorem C, \cite{GGS15}]\label{thm:ggsglobal}
Let $\pi$ be an automorphic representation of $\RG(\BA)$, let $(s,u)$ be a Whittaker pair, and $(h, u)$ a neutral Whittaker pair such that $\CF_{h,u}(\pi)$ is zero. 
Then, $\CF_{s,u}(\pi)$ is zero.
\end{thm}

This means that if $u \in \Oh$ where $\Oh \nin \mathcal{WF}(\pi)$ then, for any Whittaker pair $(s, u)$, not necessarily neutral, the associated Fourier coefficient $\mathcal{F}_{s, u}(\varphi)$ vanishes identically for $\varphi \in \pi$.

In this paper, we focus on the group $\SL_n$ where we parametrize a character on $N_s$ by $u \in \lie g^{s}_{-2}$ as 
\begin{equation}
    \label{eq:character}
    \psi_u(n)=\psi(\tr (u\log(n))) \qquad n \in N_s(\ads).
\end{equation}
Then, for any $l$ in the normalizer of $N_s(\ads)$ in $\RG(\ads)$
\begin{equation}
    \label{eq:character-conjugation}
    \begin{split}
        \psi_y^l(x) &= \psi_y(l x l^{-1}) = \psi\big(\tr( y \log(l x l^{-1}))\big) = \psi\big(\tr( y l \log(x) l^{-1})\big) \\
        &= \psi\big(\tr(l^{-1} y l \log(x))\big) = \psi_{l^{-1} y l}(x) \, .
    \end{split}
\end{equation}

The nilpotent orbits of $\SL_n$ can be described by partitions $\ul{p}$ of $n$.
Let us characterize the $F$-rational orbits of $\SL_n$ following \cite{N11}.

\begin{prop}[Proposition 4, \cite{N11}]\label{orbits}
Let $\ul{p}=[p_1 p_2 \cdots p_r]$ be an ordered partition of $n$, with $p_1 \geq p_2 \geq \ldots \geq p_r$ and let $m = \gcd(\ul{p})=\gcd(p_1, p_2, \ldots, p_r)$. 
For $d \in F^\times$, define $D(d) = \diag(1, 1, \ldots, 1, d)$ and let also $J_{\ul{p}}$ be the standard (lower triangular) Jordan matrix corresponding to $\ul{p}$: $J_{\ul{p}} = \diag(J_{[p_1]}, J_{[p_2]}, \ldots, J_{[p_3]})$, where $J_{[p]}$ is a $p\times p$ matrix with non-zero elements only on the subdiagonal which are one.

\begin{enumerate}
\item For each $d \in F^\times$, the matrix $D(d)J_{\ul{p}}$ is a representative of an $F$-rational nilpotent orbit of $\SL_n$ parametrized by $\ul{p}$, and conversely, every orbit parametrized by $\ul{p}$ has a representative of this form. 
We say that the $F$-rational orbit represented by $D(d)J_{\ul{p}}$ is parametrized by $(\ul{p},d)$.

\item The $\SL_n(F)$-orbits represented by $D(d)J_{\ul{p}}$ and $D(d')J_{\ul{p}'}$ coincide if and only if $\ul{p}=\ul{p}'$ and $d \equiv d'$ in $F^\times/(F^\times)^m$.
\end{enumerate}
\end{prop}

\begin{examp}
The $F$-rational orbit $([322], 1)$ of $\SL_7$ is represented by
\begin{equation}
    J_{[322]}=\diag(J_{[3]}, J_{[2]}, J_{[2]})=\begin{pmatrix}
0 & 0 & 0 & 0 & 0 & 0 & 0\\
1 & 0 & 0 & 0 & 0 & 0 & 0\\
0 & 1 & 0 & 0 & 0 & 0 & 0\\
0 & 0 & 0 & 0 & 0 & 0 & 0\\
0 & 0 & 0 & 1 & 0 & 0 & 0\\
0 & 0 & 0 & 0 & 0 & 0 & 0\\
0 & 0 & 0 & 0 & 0 & 1 & 0
\end{pmatrix}\,.
\end{equation}
\end{examp}

\begin{rmk}
    Over $\ol{F}$ the $F$-rational orbits for different $d$ become the same, meaning that they are completely characterized by partitions of $n$.
    There is partial ordering for partitions that agrees with the partial ordering of the $\ol{F}$-orbits, where $[p_1p_2\ldots p_r] \leq [q_1q_2\ldots q_r]$ (possibly padded by zeroes) if \cite{CollingwoodMcGovern}
    \begin{equation}
        \sum_{1\leq j \leq n} p_j \leq \sum_{1\leq j \leq n} q_j \quad \text{for } 1 \leq n \leq r \, .
    \end{equation}
   
    The Zarisky topology over $F$ is induced from that of $\overline F$ which means that we can use this partial ordering of partitions for the $F$-rational orbits as well. 
    Thus, when discussing the partial ordering of orbits or the closure of orbits we will sometimes not specify the $F$-rational orbit, but only the partition, that is, the $\SL_n(\ol{F})$-orbit.
\end{rmk}

An automorphic representation $\pi$ of $\SL_n(\ads)$ is called {\it minimal} if $\mathcal{WF}(\pi)$ is the set of orbits in the closure of the minimal (non-trivial) orbit which is represented by the partition $[21^{n-2}]$, and it is called \emph{next-to-minimal} if it is instead the set of orbits in the closure of the next-to-minimal orbit $[2^21^{n-4}]$.

We will now recall a general lemma for exchanging roots in Fourier coefficients from \cite{GRS11}. 
In \cite{GRS11}, the groups considered are quasi-split classical groups, but the lemma holds for any connected reductive group with exactly the same proof.

Let $\mathrm{G}$ be a connected reductive group defined over $F$ and let $C$ be an $F$-subgroup of a maximal unipotent subgroup of $\mathrm{G}$. 
Let also $\psi_C$ be a non-trivial character on $[C] = C(F) \bs C(\BA)$, and $X, Y$ two unipotent $F$-subgroups satisfying the following conditions:
\begin{enumerate}[label=(\arabic*)]
\item $X$ and $Y$ normalize $C$;
\item $X \cap C$ and $Y \cap C$ are normal in $X$ and $Y$, respectively, $(X \cap C) \bs X$ and $(Y \cap C) \bs Y$ are abelian;
\item $X(\BA)$ and $Y(\BA)$ preserve $\psi_C$ under conjugation;
\item $\psi_C$ is trivial on $(X \cap C)(\BA)$ and $(Y \cap C)(\BA)$;
\item $[X, Y] \subset C$;
\item there is a non-degenerate pairing 
    \begin{align}
        (X \cap C)(\BA) &\times (Y \cap C)(\BA) \rightarrow \BC^\times \\ 
        (x,y) &\mapsto \psi_C([x,y])
    \end{align}
    which is multiplicative in each coordinate, and identifies \linebreak[4]$(Y \cap C)(F) \bs Y(F)$ with the dual of
$
X(F)(X \cap C)(\BA) \bs X(\BA),
$
and
$(X \cap C)(F) \bs X(F)$ with the dual of
$
Y(F)(Y \cap C)(\BA) \bs Y(\BA).
$
\end{enumerate}

Let $B =CY$ and $D=CX$, and extend $\psi_C$ trivially to characters of $[B]=B(F)\bs B(\BA)$ and $[D]=D(F)\bs D(\BA)$,
which will be denoted by $\psi_B$ and $\psi_D$ respectively.

\begin{lem}[Lemma 7.1 of \cite{GRS11}]\label{exchangeroots}
Assume that $(C, \psi_C, X, Y)$ satisfies all the above conditions. Let $f$ be an automorphic form on $\mathrm{G}(\BA)$. Then for any $g \in \mathrm{G}(\BA)$, 
$$\int_{[B]} f(vg) \psi_B^{-1}(v) dv = \int_{(Y \cap C) (\BA) \bs Y(\BA)} \int_{[D]} f(vyg) \psi_D^{-1}(v) \,dv\,dy\,.$$
\end{lem}

For simplicity, we will use $\psi_C$ to denote its extensions $\psi_B$ and $\psi_D$ when using the lemma.

\section{Proof of theorem \ref{thm:varphi}}
\label{sec:thmA}
Before we prove Theorem \ref{thm:varphi} in this section, let us first introduce a few definitions and useful lemmas. 

Let $V_i$ be the unipotent radical of the parabolic subgroup of type $(1^i,n-i)$, that is, the parabolic subgroup with Levi subgroup $(\GL_1)^i \times GL_{n-i}$ together with a determinant one condition. Then, $N= V_n = V_{n-1}$ is the unipotent radical of the Borel subgroup and $V_i$ can be seen as the first $i$ rows of $N$.
For $1 \leq i \leq n-1$, let $\alpha_i = e_i-e_{i+1}$ be the $i$-th simple root of $\SL_n$, and let $\psi_{\alpha_i}$ be the character of $N$ defined by 
\begin{equation}
    \psi_{\alpha_i}(n)= \psi(n_{i,i+1}), \forall n \in N(\BA) \, .
\end{equation}
For a list of simple roots, we let $\psi_{\alpha_{i_1}, \ldots, \alpha_{i_m}} = \psi_{\alpha_{i_1}} \cdots {} \; \psi_{\alpha_{i_m}}$ and we also regard $\psi_{\alpha_j}$ for $j \leq i$ as a character of $V_i$ via restriction. 

Also, let $R_{i+1}$ be the subgroup of $V_{i+1}$, consisting of the elements $v$ with conditions that $v_{p,q}=0$, for all $1 \leq p \leq i$ and $p < q \leq n$, that is $R_{i+1}$ consists of the row $i+1$ in $V_{i+1}$. It is clear that $R_{i+1} \iso V_i \bs V_{i+1}$ is an abelian subgroup of $V_{i+1}$.

For a character $\psi_N$ on $N$, we say that $\psi_N$ is trivial along a simple root $\alpha_i$ if the restriction of $\psi_N$ to $R_i$ is identically zero.

\begin{examp}
    For $\SL_5$ we have that 
    \begin{equation*}
        V_3 = \Big\{
        \begin{psmallmatrix}
            1 & * & * & * & *  \\
            & 1 & * & * & *  \\
            &   & 1 & * & *  \\
            &   &   & 1 &    \\
            &   &   &   & 1  \\
        \end{psmallmatrix} \Big\} \qquad
        R_3 = \Big\{
        \begin{psmallmatrix}
            1 &   &   &   &   \\
              & 1 &   &   &   \\
              &   & 1 & * & * \\
              &   &   & 1 &   \\
              &   &   &   & 1
        \end{psmallmatrix}
        \Big\} \, .
    \end{equation*}
\end{examp}

Thus, we have that $[R_i] \cong (F \bs \BA)^{n-i}$ and the dual of $[R_i]$ is $F^{n-i}$, which can be identified with the nilpotent subalgebra ${}^t \lie r_i(F) = \log({}^t{R}_{i}(F))$, where ${}^t{R}_{i}(F)$ is the transpose of $R_{i}(F)$.
Given $y \in {}^t{\lie r}_{i}(F)$, the corresponding character $\psi_y$ on $[R_{i}]$ is given by \eqref{eq:character} as
\begin{equation}
    \psi_y(x) = \psi(\tr (y \log x)), \quad \forall x \in [R_{i}] \, .
\end{equation}

\begin{examp}
     For $SL_5$ with $R_3$ above, let
     \begin{equation*}
         y = 
         \begin{psmallmatrix}
             0 & & & & \\
             & 0 & & & \\
             & & 0 & & \\
             & & y_1 & 0 & \\
             & & y_2 & & 0
         \end{psmallmatrix} \in {}^t\lie r_3(F) \qquad
         x =
         \begin{psmallmatrix}
             1 & & & & \\
             & 1 & & & \\
             & & 1 & x_1 & x_2 \\
             & & & 1 & \\
             & & & & 1
        \end{psmallmatrix} \in [R_3] \, .
     \end{equation*}
     Then, $\psi_y(x) = \psi(\tr(y \log x)) = \psi(y_1 x_1 + y_2 x_2)$.
\end{examp}

Define
\begin{equation}
    \label{eq:trdiag}
    \trdiag(\cdot) = \diag(\cdot) -\frac1n \tr(\diag(\cdot))
\end{equation}
and let $s = s_{V_i}$
\begin{equation}
    \label{eq:s-Vi}
    s_{V_i} = \trdiag(2(i-1), 2(i-2), \ldots, 0, -2, \ldots, -2) 
\end{equation}
for which $\lie g^s_{1} = \emptyset$ and $\lie n_s = \lie g^s_{\geq 2}$ with the corresponding $N_s = V_{i}$. In particular, we have $s_N = s_{V_{n-1}}= \trdiag(2(n-2), \ldots, 0, -2)$

\begin{lem}
    \label{lem:gamma}
    Let $\varphi$ be an automorphic form on $\SL_n(\BA)$. Then, for $1 \leq i \leq n-2$,
    \begin{equation}
        \label{eq:row-expansion}
        \sum_{\substack{y \in {}^t \lie{r}_i(F) \\ y \neq 0}} \, \intl_{[R_i]} \varphi(xg) \psi^{-1}_y(x) \, dx =
            \sum_{\gamma \in \Gamma_i} \,  \intl_{[R_i]} \varphi(x \iota(\gamma) g) \psi^{-1}_{\alpha_i}(x) \, dx\,,
    \end{equation}
    where $\Gamma_i$ is defined in \eqref{eq:Gamma-i} and $\iota(\gamma)$ in \eqref{eq:iota}.
\end{lem}

We note that the left-hand side of the equation in this lemma equals $\varphi(g)$ up to constant terms corresponding to $y=0$.

\begin{proof}
    With $Y \in \Mat_{(n-i)\times 1}(F)$, we parametrize $y \in {}^t \lie{r}_i(F)$ as
    \begin{equation}
        y(Y) =
        \begin{psmallmatrix}
            0_{i-1} & 0 & 0 \\
            0 & 0 & 0 \\
            0 & Y & 0_{n-i}
        \end{psmallmatrix} \, .
    \end{equation}
    Let $\hat Y = {}^t(1, 0, \ldots, 0) \in \Mat_{(n-i)\times 1}(F)$. Then the surjective map $\SL_{n-i}(F) \to \Mat_{(n-i) \times 1}(F)^\times$ defined by $\gamma \mapsto \gamma^{-1} \hat Y$ gives that 
    \begin{equation}
        \Mat_{(n-i) \times 1}(F)^\times \iso (\SL_{n-i}(F))_{\hat Y}\bs\SL_{n-i}(F) = \Gamma_i
    \end{equation} from \eqref{eq:Gamma-i}.
    We then have that,
    \begin{equation}
        \label{eq:std-row-char-step}
        \sum_{y \neq 0} \, \intl_{[R_i]} \varphi(xg) \psi^{-1}_y(x) \, dx = \sum_{\gamma \in \Gamma_i} \intl_{[R_i]} \varphi(xg) \psi^{-1}_{y(\gamma^{-1} \hat Y)}(x) \, dx \, .
    \end{equation}
    
    We now rewrite the character using that for any $Y \in \Mat_{n-i}(F)$
    \begin{equation}
        y(\gamma^{-1}Y) = \begin{psmallmatrix}
            0_{i-1} & 0 & 0 \\
            0 & 0 & 0 \\
            0 & \gamma^{-1} Y & 0_{n-i}
        \end{psmallmatrix} =
        \begin{psmallmatrix}
            I_{i-1} & 0 & 0 \\
            0 & 1 & 0 \\
            0 & 0 & \gamma^{-1}
        \end{psmallmatrix}
        \begin{psmallmatrix}
            0_{i-1} & 0 & 0 \\
            0 & 0 & 0 \\
            0 & Y & 0_{n-i}
        \end{psmallmatrix}
        \begin{psmallmatrix}
            I_{i-1} & 0 & 0 \\
            0 & 1 & 0 \\
            0 & 0 & \gamma
        \end{psmallmatrix} = 
        l^{-1} y l\,,
    \end{equation}
    where we have introduced $l = \iota(\gamma)$ and denoted $y(Y)$ simply as $y$, which according to \eqref{eq:character-conjugation} gives, for any $x \in [R_i]$, that
    \begin{equation}
        \psi_{ y(\gamma^{-1} Y)}(x) = \psi_{l^{-1} y l}(x) = \psi_y(l x l^{-1}) \, .
    \end{equation}

    The element $l$ is in the Levi subgroup of the parabolic subgroup corresponding to $V_i$, meaning that it preserves $V_i$ under conjugation. In particular, it also normalizes $R_i$ since for $x \in R_i$ parametrized by $X \in \Mat_{1\times n-1}$
    \begin{equation}
        l x(X) l^{-1} = 
        \begin{psmallmatrix}
            I_{i-1} & 0 & 0 \\
            0 & 1 & 0 \\
            0 & 0 & \gamma
        \end{psmallmatrix}
        \begin{psmallmatrix}
            I_{i-1} & 0 & 0 \\
            0 & 1 & X \\
            0 & 0 & I_{n-i}
        \end{psmallmatrix}
        \begin{psmallmatrix}
            I_{i-1} & 0 & 0 \\
            0 & 1 & 0 \\
            0 & 0 & \gamma^{-1}
        \end{psmallmatrix} =
        \begin{psmallmatrix}
            I_{i-1} & 0 & 0 \\
            0 & 1 & X \gamma^{-1} \\
            0 & 0 & I_{n-i}
        \end{psmallmatrix} = x(X \gamma^{-1})\, .
    \end{equation}

    We can thus make the variable substitution $lxl^{-1} \to x$ in \eqref{eq:std-row-char-step} to obtain
    \begin{equation}
        \label{eq:Yhat}
        \sum_{\gamma \in \Gamma_i} \intl_{[R_i]} \varphi(x l g) \psi^{-1}_{y(\hat Y)}(x) \, dx \, ,
    \end{equation}
    where we have used the fact that $\varphi$ is left-invariant under $l^{-1}$. Noting that $\psi_{y(\hat Y)} = \psi_{\alpha_i}$ this proves the lemma.
\end{proof}

We will now state a similar lemma for the last row $R_{n-1}$, that needs to be treated separately. 
The freedom in choosing a character $\psi_0$ in this lemma will be of importance later.

\begin{lem}
    \label{lem:gamma-last-row}
    Let $\varphi$ be an automorphic form on $\SL_n(\BA)$. Then, for any character $\psi_0$ on $N$ trivial on 
    $R_{n-1}$ and along (at least) two adjacent simple roots not including $\alpha_{n-1}$,
    \begin{equation}
        \sum_{\substack{y \in {}^t \lie{r}_{n-1}(F) \\ y \neq 0}} \, \intl_{[R_{n-1}]} \varphi(xg) \psi^{-1}_y(x) \, dx =
            \sum_{\gamma \in \Gamma_{n-1}(\psi_0)} \,  \intl_{[R_{n-1}]} \varphi(x \iota(\gamma) g) \psi^{-1}_{\alpha_{n-1}}(x) \, dx\,,
    \end{equation}
    where $\Gamma_{n-1}(\psi_0)$ is defined in \eqref{eq:Gamma-i}.
\end{lem}
 
\begin{proof}
    With $Y \in F$, we parametrize $y \in {}^t \lie{r}_{n-1}(F)$ as
    \begin{equation}
        y(Y) =
        \begin{psmallmatrix}
            0_{n-2} & 0 & 0 \\
            0 & 0 & 0 \\
            0 & Y & 0
        \end{psmallmatrix} \, .
    \end{equation}
    
    We recall from page~\pageref{tpsi0} that $T_{\psi_0}$ is the subgroup of diagonal elements in $\SL_n(F)$ stabilizing $\psi_0$ under conjugation of its argument and that $y \in {}^t \lie r_{n-1}(F) \iso F$. 
    The map $T_{\psi_0} \to {}^t \lie r_{n-1}(F)^\times : h \mapsto h^{-1} y(1) h$ is surjective, which can be shown as follows.
    The character $\psi_0$ is, by assumption, trivial along at least two adjacent simple roots not including $\alpha_{n-1}$. 
    Pick such a pair $\alpha_{j-1}$ and $\alpha_j$ where $2 \leq j \leq n-2$ and for an arbitrary $m \in F^\times$ let $h = \diag(1, \ldots, 1, m, 1, \ldots 1, 1/m)$ where the first non-trivial element is at the $j$th position. 
    Then $h \in T_{\psi_0}$ since $y_0 \in {}^t \lie n$ corresponding to $\psi_0$ is zero at both rows and columns $j$ and $n$ and $h \mapsto y(m)$ 
    
    Because of \eqref{eq:character-conjugation} we have that the centralizer of $y(1)$ in $T$ is $T_{\psi_{\alpha_{n-1}}}$, and thus,
    \begin{equation}
        {}^t \lie r_{n-1}(F) \iso (T_{\psi_0} \cap T_{\psi_{\alpha_{n-1}}}) \bs T_{\psi_0} = \Gamma_{n-1}(\psi_0) \, .
    \end{equation}

    We then have that 
    \begin{align}
         &   \sum_{\substack{y \in {}^t \lie{r}_{n-1}(F) \\ y \neq 0}} \, \intl_{[R_{n-1}]} \varphi(xg) \psi^{-1}_y(x) \, dx = \sum_{\gamma \in \Gamma_{n-1}(\psi_0)} \, \intl_{[R_{n-1}]} \varphi(xg) \psi^{-1}_{\gamma^{-1} y(1) \gamma}(x) \, dx \\
            &\quad= \sum_{\gamma \in \Gamma_{n-1}(\psi_0)} \, \intl_{[R_{n-1}]} \varphi(xg) \psi^{-1}_{y(1)}(\gamma x \gamma^{-1}) = \sum_{\gamma \in \Gamma_{n-1}(\psi_0)} \, \intl_{[R_{n-1}]} \varphi(x \gamma g) \psi^{-1}_{\alpha_{n-1}}(x) \, ,
    \end{align}
    after making the variable change $\gamma x \gamma^{-1} \to x$, which concludes the proof.
\end{proof}

\begin{rmk}
    \label{rem:character-condition}
    For $n \geq 5$ any character $\psi_0$ on $N$ that is non-trivial along at most a single simple root which is not $\alpha_{n-1}$ satisfies the character condition in lemma~\ref{lem:gamma-last-row}.
\end{rmk}

The following lemma will be used to iteratively expand in rows. The lemma, which is valid for any automorphic representation, will be followed by two corollaries that specialize to the minimal and next-to-minimal representations respectively.

\begin{lem}
    \label{lem:Vi-to-Vi+1}
    Let $\varphi$ be an automorphic form on $\SL_n(\BA)$, $1 \leq i \leq n-2$, and $\psi_0$ be a character on $N$ trivial on the complement of $V_i$ in $N$. For $i = n-2$ we also require that $\psi_0$ is trivial along (at least) two adjacent simple roots not including $\alpha_{n-1}$. Then, 
    \begin{equation}
        \label{eq:Vi-expansion}
        \begin{multlined}
            \intl_{[V_i]} \varphi(vg) \psi^{-1}_0(v) dv = \intl_{[V_{i+1}]}  \varphi(vg) \psi^{-1}_0(v) \, dv +{} \\ 
            \quad + \sum_{\gamma \in \Gamma_{i+1}(\psi_0)} \, \intl_{[V_{i+1}]} \varphi(v \iota(\gamma) g) \psi^{-1}_0(v) \psi^{-1}_{\alpha_{i+1}}(v) \, dv \, .
        \end{multlined}
    \end{equation}
\end{lem}
\begin{proof} 
    For $x \in R_{i+1}(F)$ and $v \in V_i(\ads)$ we have that $\varphi(xvg) = \varphi(vg)$ and can thus Fourier expand along the abelian unipotent $R_{i+1}$ as
\begin{equation}
    \label{eq:further-expansion}
    \varphi(vg) = \sum_{y \in {}^t \lie r_{i+1}(F)} \intl_{[R_{i+1}]} \varphi(xvg) \psi^{-1}_y(x) \, dx \, .
\end{equation}
Then, using lemma \ref{lem:gamma} (for $i+1 \leq n-2$) or lemma~\ref{lem:gamma-last-row} (for $i+1 = n-1$)
\begin{equation}
    \varphi(vg) = \intl_{[R_{i+1}]} \varphi(xvg) \, dx + \sum_{\gamma \in \Gamma_{i+1}(\psi_0)} \, \intl_{[R_{i+1}]} \varphi(x\iota(\gamma)vg) \psi^{-1}_{\alpha_{i+1}}(x) \, dx\,
    .
\end{equation}
    
Let $v \in V_{i}$ be parametrized as
\begin{equation}
    v = 
    \begin{pmatrix}
        A & B \\
        0 & I_{n-i-1}
    \end{pmatrix}\,,
\end{equation}
where $A \in \Mat_{(i+1) \times (i+1)}$ is upper unitriangular and $B \in \Mat_{(i+1) \times (n-i-1)}$ with the elements in the last row being zero. 
Since $B$ does not intersect the abelianization $[N,N]\bs N$ (that is, the Lie algebra of $B$ does not contain any generator of a simple root), we have, by assumption, that $\psi_0$ only depends on $A$. 
Similarly, we parametrize $x \in R_{i+1}$ as
\begin{equation}
    x =
    \begin{pmatrix}
        I_{i+1} & B' \\
        0 & I_{n-i-1}
    \end{pmatrix}\,,
\end{equation}
where $B' \in \Mat{(i+1) \times (n-i-1)}$ with non-zero elements only in the last row. 
Then,
\begin{equation}
    xv =
    \begin{pmatrix}
        A & B + B' \\
        0 & I_{n-i-1}
    \end{pmatrix} \, ,
\end{equation}
which means that $\psi_0(v) = \psi_0(xv)$, and since $\psi_{\alpha_{i+1}}$ only depends on the first column in $B'$ which is the same as for $B + B'$, we also have that $\psi_{\alpha_{i+1}}(x) = \psi_{\alpha_{i+1}}(xv)$.

\begin{itemize-small}
    \item For $1 \leq i \leq n-3$ with $\gamma \in \Gamma_{i+1}$, $l = \iota(\gamma)$ is in the Levi subgroup corresponding to $V_{i}$ and we will now show that $\psi_0(l^{-1} v l) = \psi_0(v)$ for $v \in [V_{i}]$. 
We have that
\begin{equation}
    l^{-1}vl = 
    \begin{pmatrix}
        I_{i+1} & 0 \\
        0 & \gamma^{-1}
    \end{pmatrix}
    \begin{pmatrix}
        A & B \\
        0 & I_{n-i-1}
    \end{pmatrix}
    \begin{pmatrix}
        I_{i+1} & 0 \\
        0 & \gamma
    \end{pmatrix} = 
    \begin{pmatrix}
        A & B \gamma \\
        0 & I_{n-i-1}
    \end{pmatrix}
\end{equation}
and $\psi_0(v)$ only depends on $A$. 
\item For $i = n-2$ with $\gamma \in \Gamma_{n-1}(\psi_0)$, $l = \iota(\gamma) = \gamma$ is in the stabilizer $T_{\psi_0}$ which normalizes $V_i$ and, by definition, means that $\psi_0(v) = \psi_0(lvl^{-1})$.
\end{itemize-small}

Thus, for $1 \leq i \leq n-2$,
\begin{equation}
    \begin{multlined}
        \intl_{[V_i]} \varphi(vg) \psi^{-1}_0(v) dv = \intl_{[V_i]} \intl_{[R_{i+1}]}  \varphi(xvg) \psi^{-1}_0(v) \, dx \,dv +{} \\ + \sum_{\gamma \in \Gamma_{i+1}(\psi_0)} \, \intl_{[V_i]} \intl_{[R_{i+1}]} \varphi(x v l g) \psi^{-1}_{\alpha_{i+1}}(x) \psi^{-1}_0(v) \, dx  \,dv\,,
    \end{multlined}
\end{equation}
where we have made the variable change $l v l^{-1} \to v$.

Using that $R_{i+1} V_i = V_{i+1}$ the above expressions simplifies to
\begin{equation} 
    \intl_{[V_{i+1}]}  \varphi(vg) \psi^{-1}_0(v) \, dv + \sum_{\gamma \in \Gamma_{i+1}(\psi_0)} \, \intl_{[V_{i+1}]} \varphi(v \iota(\gamma) g) \psi^{-1}_0(v) \psi^{-1}_{\alpha_{i+1}}(v) \, dv \, .
\end{equation}
\end{proof}

\begin{cor}
    \label{cor:min-row}
    Let $\pi$ be an irreducible minimal automorphic representation of $\SL_n(\BA)$, $\varphi \in \pi$, and $\psi_0$ be a character on $N$ trivial on the complement of $V_i$ in $N$, $1 \leq i \leq n-2$. 
    Then, $\mathcal{F}_{\psi_0} \defeq \int_{[V_i]} \varphi(vg) \psi^{-1}_0 \, dv$ can be further expanded as follows.

    \begin{enumerate}[label=\textnormal{(\roman*)}, leftmargin=0cm,itemindent=1.75\parindent,labelwidth=\itemindent,labelsep=0mm,align=left]
            \item \label{itm:min-trivial} If $\psi_0 = 1$, then
            \begin{flalign}
                \qquad \mathcal{F}_{\psi_0} &= \intl_{[V_{i+1}]} \varphi(vg) dv + \sum_{\gamma \in \Gamma_{i+1}} \intl_{[V_{i+1}]} \varphi(v \iota(\gamma) g) \psi^{-1}_{\alpha_{i+1}}(v) \,dv, &
            \end{flalign}
             where $\Gamma_{i+1}(\psi_0)$ with $\Gamma_{i+1} = \Gamma_{i+1}(1)$ 
    is defined in \eqref{eq:Gamma-i}.
            \item \label{itm:min-single} If $\psi_0 = \psi_{\alpha_j}$ $(1 \leq j \leq i)$, then
            \begin{flalign}
                \qquad \mathcal{F}_{\psi_0} &= \intl_{[V_{i+1}]} \varphi(vg) \psi^{-1}_0(v) \,dv.&
            \end{flalign}
    \end{enumerate}
\end{cor}

\begin{proof}
    We will use lemma~\ref{lem:Vi-to-Vi+1} where all the considered $\psi_0$ satisfy the character condition for the last row according to remark~\ref{rem:character-condition}.
    
    For $\psi_0 = 1$, the expression is already in the form of lemma \ref{lem:Vi-to-Vi+1}. 
    This proves case \ref{itm:min-trivial}.
    
    For $\psi_0 = \psi_{\alpha_j}$ with $1 \leq j \leq i$ we have that $\psi_0(v) \psi_{\alpha_{i+1}}(v) = \psi_{\alpha_j, \alpha_{i+1}}(v) = \psi_u(v)$ for some $u \in \lie g$ which is in the next-to-minimal orbit. 
    Theorem \ref{thm:ggsglobal} with the Whittaker pair $(s_{V_{i+1}}, u)$ gives that $\mathcal{F}_{s_{V_{i+1}}, u}(\varphi)$ vanishes for $\varphi$ in the minimal representation which leaves only the constant (or trivial) mode in lemma~\ref{lem:Vi-to-Vi+1}. 
    This proves case \ref{itm:min-single}.
\end{proof}

\pagebreak[2]

\begin{cor}
    \label{cor:ntm-row}
    Let $\pi$ be an irreducible next-to-minimal automorphic representation of $\SL_n(\BA)$, $\varphi \in \pi$, and $\psi_0$ be a character on $N$ trivial on the complement of $V_i$ in $N$, $1 \leq i \leq n-2$.
    Then, $\mathcal{F}_{\psi_0} \defeq \int_{[V_i]} \varphi(vg) \psi^{-1}_0 \, dv$ can be further expanded as follows.
    
    \begin{enumerate}[label=\textnormal{(\roman*)}, leftmargin=0cm,itemindent=1.75\parindent,labelwidth=\itemindent,labelsep=0mm,align=left]
        \item \label{itm:ntm-trivial} If $\psi_0 = 1$, then
        \begin{flalign}
            \qquad \mathcal{F}_{\psi_0} &= \intl_{[V_{i+1}]} \varphi(vg) dv + \sum_{\gamma \in \Gamma_{i+1}} \intl_{[V_{i+1}]} \varphi(v\iota(\gamma) g) \psi^{-1}_{\alpha_{i+1}}(v) \,dv\,. &
        \end{flalign}
        \item \label{itm:ntm-single-j} If $\psi_0 = \psi_{\alpha_j}$ $(1 \leq j < i)$, then
        \begin{flalign}
            \qquad \mathcal{F}_{\psi_0} &= \intl_{[V_{i+1}]} \varphi(vg) \psi^{-1}_{\alpha_j}(v) dv + \hspace{-1.5em} \sum_{\gamma \in \Gamma_{i+1}(\psi_{\alpha_j}\!)} \, \intl_{[V_{i+1}]} \varphi(v\iota(\gamma) g) \psi^{-1}_{\alpha_j, \alpha_{i+1}}(v) \,dv\,. &
        \end{flalign}
        \item \label{itm:ntm-single-i} If $\psi_0 = \psi_{\alpha_i}$, then
        \begin{flalign}
            \qquad \mathcal{F}_{\psi_0} &= \intl_{[V_{i+1}]} \varphi(vg) \psi^{-1}_{\alpha_i}(v) \,dv\,. &
        \end{flalign}
        \item \label{itm:ntm-double} If $\psi_0 = \psi_{\alpha_j, \alpha_k}$ $(1 < j+1 < k \leq i)$, then
        \begin{flalign}
            \qquad \mathcal{F}_{\psi_0} &= \intl_{[V_{i+1}]} \varphi(vg) \psi^{-1}_{\alpha_j, \alpha_k}(v) \,dv\,. &
        \end{flalign}
    \end{enumerate}
    Where $\Gamma_{i+1}(\psi_0)$ with $\Gamma_{i+1} = \Gamma_{i+1}(1)$ is defined in \eqref{eq:Gamma-i}.
\end{cor}
\begin{proof}
    We will use lemma~\ref{lem:Vi-to-Vi+1} where the considered $\psi_0$ in cases \ref{itm:ntm-trivial}--\ref{itm:ntm-single-i} satisfy the character condition for the last row according to remark~\ref{rem:character-condition}.
    
    \begin{itemize-small}
    \item For $\psi_0 = 1$, the expression is already in the form of lemma~\ref{lem:Vi-to-Vi+1}.
    This proves case \ref{itm:ntm-trivial}.

    \item For $\psi_0 = \psi_{\alpha_j}$ with $1 \leq j < i$ we get that $\psi_0(v) \psi_{\alpha_{i+1}}(v) = \psi_{\alpha_j, \alpha_{i+1}}(v)$.
    This proves case \ref{itm:ntm-single-j}.

    \item For $\psi_0 = \psi_{\alpha_i}$ we get that $\psi_0(v) \psi_{\alpha_{i+1}}(v) = \psi_{\alpha_i, \alpha_{i+1}}(v) = \psi_u(v)$ for some $u \in \lie g$ belonging to an orbit higher than the next-to-minimal. 
    Theorem \ref{thm:ggsglobal} with the Whittaker pair $(s_{V_{i+1}}, u)$ gives that $\mathcal{F}_{s_{V_{i+1}}, u}(\varphi)$ vanishes both for $\varphi$ in the minimal and next-to-minimal representations which leaves only the constant mode in lemma~\ref{lem:Vi-to-Vi+1}. 
    This proves case \ref{itm:ntm-single-i}.

    \item Lastly, for $\psi_0 = \psi_{\alpha_j,\alpha_k}$ with $2 \leq j+1 < k \leq i$ we first consider $i \leq n-3$ with lemma~\ref{lem:Vi-to-Vi+1}. 
    We get that $\psi_0(v) \psi_{\alpha_{i+1}}(v) = \psi_{\alpha_j, \alpha_k, \alpha_{i+1}}(v) = \psi_u(v)$ for some $u \in \lie g$ belonging to an orbit higher than the next-to-minimal.
    Theorem \ref{thm:ggsglobal} with the Whittaker pair $(s_{V_{i+1}}, u)$ gives that $\mathcal{F}_{s_{V_{i+1}}, u}(\varphi)$ vanishes for $\varphi$ in next-to-minimal representation which leaves only the first term in \eqref{eq:Vi-expansion}. 
    \end{itemize-small}
    For $i = n-2$, we expand along the last row and obtain a sum over characters $\psi_u = \psi_0 \psi_y$ on $N$ for all $y \in {}^t \lie r_{n-1}(F)$ where only $y = 0$ gives a $u \in \lie g$ belonging to an orbrit in the closure of the next-to-minimal orbit.
    Again, using theorem~\ref{thm:ggsglobal} only the constant mode remains.
    This proves case \ref{itm:ntm-double} and completes the proof.
\end{proof}

\begin{proof}[\bf Proof of theorem \ref{thm:varphi}]
    Since $\varphi(x_1g) = \varphi(g)$ for $x_1 \in V_1(F)$ we can make a Fourier expansion on $V_1$ and then use lemma \ref{lem:gamma} to obtain
    \begin{equation}
        \label{eq:ThmA-first-row}
        \varphi(g) = \intl_{[V_1]} \varphi(v g) \, dv + \sum_{\gamma_1 \in \Gamma_1} \intl_{[V_1]} \varphi(v \iota(\gamma_1) g) \psi^{-1}_{\alpha_1}(v) \, dv \, .
    \end{equation}
    
    We will now make an iteration in the rows of the nilpotent, starting with the row $i = 1$ and continue until we reach the last row $i = n - 1$.
    
    \begin{itemize-small}
    \item For case \ref{itm:varphi-min}, that is, with $\varphi$ in the minimal representation, the first step, using corollary~\ref{cor:min-row}, is
    \begin{multline*}
        \varphi(g) = \intl{[V_2]} \varphi(vg) \, dv + \sum_{\gamma_2 \in \Gamma_2} \, \intl_{[V_2]} \varphi(v \iota(\gamma_2) g)\psi^{-1}_{\alpha_2}(v) \, dv +{} \\[-1em]
        + \sum_{\gamma_1 \in \Gamma_1} \, \intl_{[V_2]} \varphi(v \iota(\gamma_1) g) \psi^{-1}_{\alpha_1}(v) \, dv \, ,
    \end{multline*}
    where we note that the extra second term comes from the constant term on $V_1$. We will, after the iteration end up with
    \begin{equation}
        \varphi(g) = \intl_{[N]} \varphi(ng) \, dn + \sum_{i=1}^{n-1} \sum_{\gamma \in \Gamma_{i}} \, \intl_{[N]} \varphi(n \iota(\gamma) g) \psi^{-1}_{\alpha_i}(n) \, dn \, .
    \end{equation}
    This completes the proof for the minimal representation.

    \item For case \ref{itm:varphi-ntm}, where $\varphi$ is in the next-to-minimal-representation, we start again from \eqref{eq:ThmA-first-row} and expand using corollary~\ref{cor:ntm-row}.
    We get, for the first step, that
    \begin{multline}
        \varphi(g) = \Big( \intl_{[V_2]} \varphi(v g) \, dv + \sum_{\gamma_2\in\Gamma_2}\intl_{[V_2]} \varphi(v \iota(\gamma_2) g) \psi^{-1}_{\alpha_2}(vg) \, dv \Big) +{} \\[-1em]
        +\sum_{\gamma_1 \in \Gamma_1} \intl_{[V_2]} \varphi(v \iota(\gamma_1) g) \psi^{-1}_{\alpha_1}(v) \, dv \, ,
    \end{multline}
    where the parenthesis comes from the expansion of the constant term in \eqref{eq:ThmA-first-row}. Expanding in the next row as well, this becomes
    \begin{multline}
            \Big(\! \intl_{[V_3]} \!\!\! \varphi(vg) \, dv \, + \hspace{-0.4em} \sum_{\gamma_3\in\Gamma_3} \, \intl_{[V_3]} \!\!\! \varphi(v \iota(\gamma_3) g) \psi^{-1}_{\alpha_3}(v) \, dv \, + \hspace{-0.4em} \sum_{\gamma_2\in\Gamma_2} \, \intl_{[V_3]} \!\!\! \varphi(v \iota(\gamma_2) g) \psi^{-1}_{\alpha_2}(v) \, dv \Big) \,+{} \\
            + \hspace{-0.4em}\sum_{\gamma_1 \in \Gamma_1} \!\! \Big( \! \intl_{[V_3]} \!\!\! \varphi(v \iota(\gamma_1) g) \psi^{-1}_{\alpha_1}(v) \, dv + \hspace{-1.4em} \sum_{\gamma_3 \in \Gamma_3(\psi_{\alpha_1}\!)} \, \intl_{[V_3]} \!\!\! \varphi(v \iota(\gamma_3) \iota(\gamma_1) g) \psi^{-1}_{\alpha_1, \alpha_3}(v) \, dv \Big)\,.
    \end{multline}

    For each expansion adding a row $i$, the constant term gives an extra sum over $\Gamma_{i}$ of a Fourier integral with character $\psi_{\alpha_i}$, and from all terms with characters $\psi_{\alpha_j}$ with $j < i - 1$ we get an extra sum over $\Gamma_{i}(\psi_{\alpha_j})$ together with a character $\psi_{\alpha_j, \alpha_{i}}$. 
    Corollary~\ref{cor:ntm-row} \ref{itm:ntm-double} implies that these terms with characters non-trivial along two simple roots do not receive any further contributions. 
    Thus, after repeatedly using corollary~\ref{cor:ntm-row} to the last row, we get that
    \begin{multline}
        \label{eq:ThmA-ntm}
        \varphi(g) = \intl_{[N]} \varphi(ng) \, dn + \sum_{i=1}^{n-1} \sum_{\gamma \in \Gamma_i} \, \intl_{[N]} \varphi(n \iota(\gamma) g) \psi^{-1}_{\alpha_i}(n) \, dn +{} \\
        + \sum_{j=1}^{n-3} \sum_{i=j+2}^{n-1} \sum_{\substack{\gamma_i \in \Gamma_i(\psi_{\alpha_j}\!) \\ \gamma_j \in \Gamma_j}} \, \intl_{[N]} \varphi(n \iota(\gamma_i) \iota(\gamma_j) g) \psi^{-1}_{\alpha_j, \alpha_i} (n) \, dn \, ,
    \end{multline}
    which completes the proof of Theorem \ref{thm:varphi}.
    \end{itemize-small}
\end{proof}

\section{Proof of theorem \ref{thm:max-parabolic}}
\label{sec:thmB}
In this section, we prove Theorem \ref{thm:max-parabolic} which relates Fourier coefficients on a maximal parabolic subgroup with Whittaker coefficients on the Borel subgroup.
Recalling that the constant terms are known from \cite{MW95}, we only focus on non-trivial characters, but first we need to introduce some notation and lemmas.

For $1 \leq m \leq n-1$, let $U_m$ be the unipotent radical of the maximal parabolic subgroup $P_m$ with Levi subgroup $L_m$ isomorphic to the subgroup of $\GL_m \times \GL_{n-m}$ defined by $\{(g,g')\in \GL_m \times \GL_{n-m}: \det(g) \det(g')=1\}$. $U_m$ is abelian and is isomorphic to the set of all $m \times (n-m)$ matrices. Write $U_m$ as 
\begin{equation}
    U_m = \left\{ \begin{pmatrix}
        I_m & X\\
        0 & I_{n-m}
    \end{pmatrix} : X \in \Mat_{m \times (n-m)}\right\} \, .
\end{equation} 
Let $\ol{U}_m = {}^t U_m$ be the unipotent radical of the opposite parabolic $\ol{P}_m$. Then the Lie algebra of $\ol{U}_m$ can be written as
\begin{equation}
    \label{eq:um-param}
    \ol{\mathfrak{u}}_m = {}^t \lie u_m = 
    \left\{y(Y) = \begin{pmatrix}
        0_m & 0\\
        Y & 0_{n-m}
    \end{pmatrix} : Y \in \Mat_{(n-m) \times m}\right\}.
\end{equation}
It is clear that the character group of $U_m$ can be identified with ${}^t \lie u_m$. 
$L_m$ acts on ${}^t \lie u_m$ via conjugation and with \eqref{eq:character-conjugation} this becomes a conjugation of the corresponding character's argument. 
Because of \eqref{eq:Fourier-L-conjugation}, the Fourier coefficients for characters in the same $L_m(F)$-orbit are related by translates of their arguments, which means that we only need to compute one Fourier coefficient for each orbit.
We will therefore now describe the $L_m(F)$-orbits of elements $y(Y) \in {}^t \lie u_m$ but leave the details to be proven in appendix~\ref{app:levi-orbits}.

Starting first with $\ol F$ the number of $L_m(\ol F)$-orbits is $\min(m,n-m)+1$ and the orbits are classified by the rank of the $(n-m) \times m$ matrix $Y$. 
A representative of an $L_m(\ol F)$-orbit corresponding to rank $r$ can be chosen as $y(Y_r)$ where $Y_r$ is an $(n-m) \times m$ matrix, zero everywhere except for the upper right $r \times r$ submatrix which is anti-diagonal with all anti-diagonal elements equal to one.
For each rank $r$, $0 \leq r \leq \min(m,n-m)$, the corresponding $G(\ol{F})$-orbit is parametrized by the partition $[2^r 1^{n-2r}]$. 

As shown in appendix~\ref{app:levi-orbits}, the $L_m(F)$-orbits are characterized by the same data as the $G(F)$-orbits with $([2^r 1^{n-2r}], d)$, $0 \leq r \leq \min(m,n-m)$, $d \in F^\times/(F^\times)^k$ and $k \in \gcd([2^r 1^{n-2r}])$ with representatives $y(Y_r(d))$ where $Y_r(d)$ is of the same form as $Y_r$ above, but with the lower left element in the $r \times r$ matrix equal to $d$
\begin{equation}
    Y_r(d) = 
    \begin{pmatrix}
        \quad 0 &
        \begin{bsmallmatrix}
            & & & 1 \\
            & & \reflectbox{$\ddots$} & \\
            & 1 & & \\
            d & & & 
        \end{bsmallmatrix} \\[1.5em]
        \quad 0 & 0
    \end{pmatrix} \in \Mat_{(n-m) \times m}(F)\, .
\end{equation}
We will continue to write $Y_r(1) = Y_r$. 
Note that for $0 \leq r \leq 2$ and $n \geq 5$, $k$ is equal to $1$.
Each such $L_m(F)$-orbit is also part of the $G(F)$-orbit of the same data. 

From \eqref{eq:character} the corresponding character on $U_m$ is
\begin{equation}
    \psi_{y(Y_r)}(u)=\psi(\tr (y(Y_r) \log (u))), \quad u \in U_m(\ads).
\end{equation}
Let $s_m$ be the semisimple element $\trdiag(1,1,\ldots, 1, -1,-1, \ldots, -1)$ with $m$ copies of $1$'s and $(n-m)$ copies of $-1$'s. Then, for any automorphic form $\varphi$ on $\SL_n(\BA)$, the following Fourier coefficient
\begin{equation}
    \label{eq:Yr-coefficient}
    \int_{[U_m]} \varphi(ug) \psi_{y(Y_r(d))}^{-1}(u) \,du 
\end{equation}
is exactly the degenerate Fourier coefficient $\CF_{s_m,y(Y_r(d))}(\varphi)$. 

Note that in this paper, we focus on minimal and next-to-minimal representations, hence we only need to consider the cases of $0 \leq r \leq 2$. Indeed, for $3 \leq r \leq \min(m,n-m)$, by definition, the generalized Fourier coefficient attached to the partition $[2^r1^{n-2r}]$ is identically zero for minimal and next-to-minimal representations. By Theorem \ref{thm:ggsglobal} and since $y(Y_r(d))$ is in the $G(\ol{F})$-orbit $[2^r1^{n-2r}]$, all the Fourier coefficients $\CF_{s_m,y(Y_r)}(\varphi)$ are also identically zero.

This leaves $r \in \{1, 2\}$ and with our assumption that $n \geq 5$, we thus only need to consider the representatives $y(Y_1)$ and $y(Y_2)$ with $d = 1$ since $\gcd([2^r1^{n-2r}])=1$. 

The above arguments proves the first part of theorem~\ref{thm:max-parabolic}, that there exists an element $l \in L_m(F)$ such that $\mathcal{F}_U(\varphi, \psi_U; g) = \mathcal{F}_U(\varphi, \psi_{y(Y_r)}; lg)$ (note the slight difference in notation $\psi_{y(Y_r)}$ instead of $\psi_{Y_r}$), and that all $\mathcal{F}_U(\varphi, \psi_{y(Y_r)}; lg)$ for $r > r_\pi$ vanish identically where $r_{\pi_\text{min}} = 1$ and $r_{\pi_\text{ntm}} = 2$.

We will now determine the remaining Fourier coefficients $\mathcal{F}_U(\varphi, \psi_{y(Y_r)}; g)$ in terms of Whittaker coefficients.
For $1 \leq m \leq n-1$, $0 \leq i \leq m-1$, let $U_m^i$ be the unipotent radical of the parabolic of type $(m-i,1^i,n-m)$. Note that $U_m^0=U_m$. 
Note that the character $\psi_{y(Y_1)}$ can be extended to a character of any
subgroup of $N$ containing $U_m$, still denoted by $\psi_{y(Y_1)}$. 
Let $C_{m-i}$ be the subgroup of $U_m^{i+1}$ consisting of elements with $u_{p,q}=0$ except when $q = m-i$ and the diagonal elements. Note that $C_{m-i}$ is an abelian subgroup and its character group can be identified with ${}^t \mathfrak{c}_{m-i}$, the Lie algebra of ${}^t C_{m-i}$. 
Write $C_{m-i}$ as 
\begin{equation}
    C_{m-i} = \left\{ c(X) = 
    \begin{psmallmatrix}
        I_{m-i-1} & X & 0\\
        0 & 1 & 0\\
        0 & 0 & I_{n-m+i}
    \end{psmallmatrix}\right\}
\end{equation}
and ${}^t \mathfrak{c}_{m-i}$ as
\begin{equation}
    {}^t \mathfrak{c}_{m-i} = \left\{ y(Y)=
    \begin{psmallmatrix}
        0_{m-i-1} & 0 & 0\\
        Y & 0 & 0\\
        0 & 0 & 0_{n-m+i}
    \end{psmallmatrix}\right\} \, .
\end{equation} 
For each $y \in {}^t \mathfrak{c}_{m-i}$, the corresponding character $\psi_{y}$ of $C_{m-i}$ is defined by $\psi_{y}(c)=\psi(\tr(y \log(c))$. 
For any $g \in \GL_{m-i-1}$, let 
\begin{equation}
    \iota(g)=
    \begin{psmallmatrix}
        g & 0 & 0\\
        0 & I_{n-m+i} & 0\\
        0 & 0 & \det(g)^{-1}
    \end{psmallmatrix} \in \SL_n \, .
\end{equation}

\begin{examp}
    For $\SL_5$ we have that
    \begin{equation}
        U_3 =
        \left\{ \begin{psmallmatrix}
            1 &   &   & * & * \\
              & 1 &   & * & * \\
              &   & 1 & * & * \\
              &   &   & 1 &   \\
              &   &   &   & 1
        \end{psmallmatrix} \right\} \qquad
        U_3^1 =
        \left\{ \begin{psmallmatrix}
            1 &   & * & * & * \\
              & 1 & * & * & * \\
              &   & 1 & * & * \\
              &   &   & 1 &   \\
              &   &   &   & 1
        \end{psmallmatrix} \right\} \qquad
        C_3 =
        \left\{ \begin{psmallmatrix}
            1 &   & * &   &   \\
              & 1 & * &   &   \\
              &   & 1 &   &   \\
              &   &   & 1 &   \\
              &   &   &   & 1
        \end{psmallmatrix} \right\} \, .
    \end{equation}
\end{examp}

Note that $U_m^{m-1} = V_m$ and $U_m^{i+1} = C_{m-i} U_m^i$. We will sometimes use $j = m-i$ instead to denote column as follows $U_m^{m-j+1} = C_j U_m^{m-j}$.

We will now construct a semi-simple element $s = s_{U_m^i}$ for which $\lie g^s_1 = \emptyset$ and such that $\lie n_s = \lie g^s_{\geq 2}$ corresponds to $N_s = U_m^i$. These conditions are satisfied by
\begin{equation}
    \label{eq:s-Umi}
    \trdiag(2i, \ldots, 2i, 2(i-1) \ldots, 2, 0, -2, \ldots, -2)
\end{equation}
with $m-i$ copies of $2i$ and $n-m$ copies of $-2$.

Note that any character $\psi$ on $N$ trivial on the complement of $U_m^i$ in $N$ is also a character on $U_m^i$ by restriction and can be expressed as $\psi_y$ with $y \in \lie g^s_{-2}$ where $s = s_{U_m^i}$ such that $(s,y)$ forms a Whittaker pair. 
Indeed, we have that $y \in \lie g^{s_N}_{-2}$ where $s_N = \trdiag(2(n-1), 2(n-2), \ldots, 0,-2)$ from \eqref{eq:s-Vi} and the complement of $U_m^i$ is described by $s-s_N$ meaning that $[y, s-s_N] = 0$ for $\psi$ to be trivial on the complement and thus $[y, s] = [y, s_N] = -2y$. 

\begin{lem}
    \label{lem:col-conjugation}
    Let $\varphi$ be an automorphic form on $\SL_n(\ads)$ and $2 \leq j \leq n$.
    Let also $\psi_0$ be a character on $N$ which, if $j = 2$, should be trivial along $\alpha_1$ and (at least) two adjacent other simple roots.
    Then, 
    \begin{equation}
        \sum_{\substack{y \in {}^t \lie c_j(F) \\ y \neq 0}} \intl_{[C_j]} \varphi(xg) \psi_y^{-1}(x) \, dx = \sum_{\gamma \in \Lambda_{j-1}(\psi_0)} \intl_{[C_j]} \varphi(x \hat\iota(\gamma) g) \psi^{-1}_{\alpha_{j-1}}(x) \, dx \, .
    \end{equation}
    where $\Lambda_j(\psi_0)$ is defined in \eqref{eq:Lambda-j} and only depends on $\psi_0$ for $j = 2$.
\end{lem}

\begin{proof}
    The proof is similar to those of lemmas~\ref{lem:gamma} and \ref{lem:gamma-last-row}.
    \begin{itemize-small}
    \item For $2 < j \leq n$, we parametrize $y \in {}^t \lie c_j(F)$ by row vectors $Y \in \Mat_{1\times(j-1)}(F)$ with representative $\hat X = (0, \ldots, 0, 1)$ such that $\psi_{y(\hat X)} = \psi_{\alpha_{j-1}}$. 
        
    The surjective map $\SL_{j-1}(F) \to {}^t \lie c_j(F)^\times : \gamma \mapsto \hat X \gamma$ gives that ${}^t \lie c_j(F)^\times \iso (\SL_{j-1}(F))_{\hat X} \bs \SL_{j-1}(F) = \Lambda_{j-1}$. 

    As in lemma~\ref{lem:gamma}, we can write the action as a conjugation $y(\hat X \gamma) = \hat\iota(\gamma)^{-1} y(\hat X) \hat\iota(\gamma)$ and, using \eqref{eq:character-conjugation}, $\psi_{y(\hat X \gamma)}(x) = \psi_{y(\hat X)}(\hat\iota(\gamma) x \hat\iota(\gamma)^{-1})$. Since $\hat \iota(\gamma)$ normalizes $C_j$ a variable change gives the wanted expression.

    \item For $j = 2$, with $y \in {}^t \lie c_2(F) \iso F$ we instead consider the map $T_{\psi_0} \to {}^t \lie c_2(F)^\times : h \mapsto h^{-1} y(1) h$ which is surjective by similar arguments as in lemma~\ref{lem:gamma-last-row} and thus gives ${}^t \lie c_2(F)^\times \iso (T_{\psi_0} \cap T_{\psi_{\alpha_1}}) \bs T_{\psi_0} = \Lambda_1(\psi_0)$. Writing the conjugation of $y$ as a conjugation of the character's argument and then substituing variables in the Fourier integral thus proves the lemma.
    \end{itemize-small}
\end{proof}

\begin{lem}
    \label{lem:col-expansion}
Let $\varphi$ be an automorphic form on $\SL_n(\ads)$, $1 \leq m \leq n-1$, $2 \leq j \leq m$ and $\psi_0$ a character on $N$ trivial on the complement of $U_m^{m-j}$ in $N$. 
For $j = 2$, $\psi_0$ should also be trivial along (at least) two adjacent simple roots other than $\alpha_1$.
\begin{multline}
    \intl_{[U_m^{m-j}]} \hspace{-0.6em} \varphi(ug) \psi_0^{-1}(u) \, du = \\[-0.5em]
    = \hspace{-1em} \intl_{[U_m^{m-j+1}]} \hspace{-1em} \varphi(ug) \psi_0^{-1}(u) \, du + \hspace{-1.4em} \sum_{\gamma \in \Lambda_{j-1}(\psi_0)} \, \intl_{[U_m^{m-j+1}]} \hspace{-1em} \varphi(u \hat\iota(\gamma) g) \psi_0^{-1}(u) \psi_{\alpha_{j-1}}^{-1}(u) \, du\,.
\end{multline}
\end{lem}

\begin{proof}
    For $2 \leq j \leq m$ we have that $\varphi(xug) = \varphi(ug)$ for $x \in C_i(F)$ and $u \in U_m^i(\ads)$ and since $C_j$ is abelian
    \begin{equation}
        \label{eq:integrand-col-exp}
        \varphi(ug) = \sum_{y \in {}^t \lie c_j(F)} \intl_{[C_j]} \varphi(xug) \psi_y^{-1}(x) \, dx \, .
    \end{equation}
    Using lemma~\ref{lem:col-conjugation}, we get that
    \begin{equation}
        \varphi(ug) = \intl_{[C_j]} \varphi(xug) \, dx + \sum_{\gamma \in \Lambda_{j-1}(\psi_0)} \intl_{[C_j]} \varphi(x \hat \iota(\gamma) ug) \psi_{\alpha_{j-1}}^{-1}(x) \, dx \, .
    \end{equation}
   
    Let $u \in U_m^{m-j}$ be parametrized as
    \begin{equation}
        u =
        \begin{pmatrix}
            I_{j-1} & B \\
            0 & A
        \end{pmatrix}
    \end{equation}
    where $A \in \Mat_{(n-j+1)\times(n-j+1)}$ is upper unitriangular (with several upper triangular elements being zero) and $B \in \Mat_{(j-1)\times(n-j+1)}$ with elements in the first column being zero. 
    Since $B$ does not intersect the abelianization $[N,N]\bs N$ (that is, the Lie algebra of $B$ does not contain any generator of a simple root), we have, by assumption, that $\psi_0$ only depends on $A$.
    We also have that $x \in C_j$ can be parametrized as
    \begin{equation}
        x =
        \begin{pmatrix}
            I_{j-1} & B' \\
            0 & I_{n-j+1}
        \end{pmatrix}
    \end{equation}
    where $B' \in \Mat_{(j-1)\times(n-j+1)}$ with only the first column non-zero. Thus,
    \begin{equation}
        xu =
        \begin{pmatrix}
            I_{j-1} & B + B' A \\
            0 & A
        \end{pmatrix}
    \end{equation}
    which means that $\psi_0(u) = \psi_0(xu)$.
    The first column of $B$ is zero and $A$ is upper unitriangular which means that the first column of $B+B'A$ is the same as the first column of $B'$ and since $\psi_{\alpha_{j-1}}$ only depends on the first column of $B'$ this implies that $\psi_{\alpha_{j-1}}(x) = \psi_{\alpha_{j-1}}(xu)$.
    
    \begin{itemize-small}
    \item For  $3 \leq j \leq m$ with $\gamma \in \Lambda_{j-1}$ and $l = \hat\iota(\gamma)$,
    \begin{equation}
        l u l^{-1} = 
        \begin{pmatrix}
            \gamma & 0 \\
            0 & I_{n-j+1}
        \end{pmatrix}
        \begin{pmatrix}
            I_{j-1} & B \\
            0 & A
        \end{pmatrix}
        \begin{pmatrix}
            \gamma^{-1} & 0 \\
            0 & I_{n-j+1}
        \end{pmatrix} = 
        \begin{pmatrix}
            I_{j-1} & \gamma B \\
            0 & A
        \end{pmatrix}
    \end{equation}
    and since $\psi_0$, by assumption, only depends on $A$ we have that $\psi_0(u) = \psi_0(l u l^{-1})$.
    \item For $j = 2$ with $\gamma \in \Lambda_1$ and $l = \hat\iota(\gamma) = \gamma$ is in the stabilizer $T_{\psi_0}$ which, by definition, means that $\psi_0(u) = \psi_0(lul^{-1})$.
    \end{itemize-small}
    
    Hence, for $2 \leq j \leq m$, and after making a variable change $lul^{-1} \to u$, we get that
    \begin{equation}
        \begin{split}
            \MoveEqLeft
            \intl_{[U_m^{m-j}]} \intl_{[C_j]} \varphi(x l ug) \psi_0^{-1}(u) \psi_{\alpha_{j-1}}^{-1}(x)\, dx \, du =  \\ 
            &= \intl_{[U_m^{m-j}]} \intl_{[C_j]} \varphi(xulg) \psi_0^{-1}(u) \psi_{\alpha_{j-1}}^{-1}(x)\, dx \, du \\
            &= \intl_{[U_m^{m-j}]} \intl_{[C_j]} \varphi(xulg) \psi_0^{-1}(xu) \psi_{\alpha_{j-1}}^{-1}(xu)\, dx \, du \\
            &= \intl_{[U_m^{m-j+1}]} \varphi(ulg) \psi_0^{-1}(u) \psi_{\alpha_{j-1}}^{-1}(u)\, du \, .
    \end{split}
    \end{equation} 
    After similar manipulations for the constant term we obtain
    \begin{multline}
        \intl_{[U_m^{m-j}]} \hspace{-0.6em} \varphi(ug) \psi_0^{-1}(u) \, du = \\[-0.5em]
        = \hspace{-1em} \intl_{[U_m^{m-j+1}]} \hspace{-1em} \varphi(ug) \psi_0^{-1}(u) \, du + \hspace{-1.4em} \sum_{\gamma \in \Lambda_{j-1}(\psi_0)} \, \intl_{[U_m^{m-j+1}]} \hspace{-1em} \varphi(ulg) \psi_0^{-1}(u) \psi_{\alpha_{j-1}}^{-1}(u)\, du \, .
    \end{multline}
\end{proof}

\begin{rmk}
    \label{rem:col-expansion}
    We note that if $\psi_0$ is trivial along $\alpha_1$ but not along at least two adjacent other simple roots we cannot use lemma~\ref{lem:col-conjugation}, but we could still make an expansion over $C_2$ and keep the sum over $y \in {}^t \lie c_2(F) \iso F$ in the proof above.
    Since the character $\psi_y$ has the same support as $\psi_{\alpha_1}$ on $N$ we still have that $\psi_y(x) = \psi_y(xu)$ for $x \in C_j(\ads)$ and $u \in U_m^{m-j}(\ads)$ and since $\psi_0$ is still a character on $N$ trivial on the complement of $U_m^{m-j}$ it is still true that $\psi_0(u) = \psi_0(xu)$. Thus, using~\eqref{eq:integrand-col-exp}
    \begin{equation}
        \begin{split}
            \intl_{[U_m^{m-2}]} \hspace{-0.6em} \varphi(ug) \psi_0^{-1}(u) \, du 
            &= \sum_{y \in {}^t \lie c_2(F)} \intl_{[U_m^{m-2}]} \intl_{[C_2]} \varphi(xug) \psi_0^{-1}(u) \psi_y^{-1}(x) \,dx \, du \\
            &= \sum_{y \in {}^t \lie c_2(F)} \intl_{[U_m^{m-2}]} \intl_{[C_2]} \varphi(xug) \psi_0^{-1}(xu) \psi_y^{-1}(xu) \,dx \, du \\
            &= \sum_{y \in {}^t \lie c_2(F)} \intl_{[V_m]} \intl_{[C_2]} \varphi(vg) \psi_0^{-1}(v) \psi_y^{-1}(v) \, dv \, .
        \end{split}
    \end{equation}
\end{rmk}

\begin{lem}\label{TheoremB:Lemma1}
Assume that $\pi$ is an irreducible minimal automorphic representation of $\SL_n(\BA)$, $\varphi \in \pi$. 
For $1 \leq m \leq n-1$, $0 \leq i \leq m-2$, and $g \in \SL_n(\BA)$, 
\begin{equation}
\int_{[U_m^i]} \varphi(ug) \psi^{-1}_{y(Y_1)}(u) \,du = \int_{[U_m^{i+1}]} \varphi(ug) \psi^{-1}_{y(Y_1)}(u) \,du\,.
\end{equation}
\end{lem}
\begin{proof}
    Using lemma~\ref{lem:col-expansion} with $\psi_0 = \psi_{y(Y_1)} = \psi_{\alpha_m}$ we get that
    \begin{equation}
        \label{eq:col-min-rep}
        \intl_{[U_m^i]} \varphi(ug) \psi_0^{-1}(u) \, du = \intl_{[U_m^{i+1}]} \varphi(ug) \psi_0^{-1}(u) \, du + \sum_{\gamma \in \Lambda_{m-i-1}(\psi_0)} \hspace{-1em} \mathcal{F}(\varphi; m,i,\gamma,g) \, ,
    \end{equation}
    where we have introduced
    \begin{equation}
        \mathcal{F}(\varphi; m,i,\gamma,g) = \intl_{[U_m^{i+1}]} \hspace{-0.7em} \varphi(u \hat\iota(\gamma) g) \psi_0^{-1}(u) \psi_{\alpha_{m-i-1}}(u)\,du \, .
    \end{equation}

    Let $s = s_{U_m^{i+1}}$ from \eqref{eq:s-Umi}, and let $u \in \lie{sl}_n$ with two non-zero entries, both being $1$, at positions $(m-i, m-i-1)$ and $(m+1, m)$.
    Then, $\mathcal{F}(\varphi; m,i,\gamma,g) = \mathcal{F}_{s,u}(\varphi)(\hat\iota(\gamma)g)$ and since $u$ is not in the closure of the minimal orbit, theorem~\ref{thm:ggsglobal} gives that $\mathcal{F}_{s,u}(\varphi)$ is identically zero leaving only the constant mode in \eqref{eq:col-min-rep}.
\end{proof}

\pagebreak[3]

\noindent\textbf{Proof of Theorem \ref{thm:max-parabolic}.}

\begin{itemize-small}
\item{\bf Minimal representation.}
Assume that $\pi$ be an irreducible minimal automorphic representation of $\SL_n(\BA)$, and $\varphi \in \pi$. Applying Lemma \ref{TheoremB:Lemma1} repeatedly, 
we get that for each $1 \leq m \leq n-1$, 
$$\int_{[U_m]}\varphi(ug)\psi_{y(Y_1)}^{-1}(u)\,du
= \int_{[U_m^{m-1}]}\varphi(ug)\psi_{y(Y_1)}^{-1}(u)\,du\,.$$
Note that $U_m^{m-1}=V_m$ and $\psi_{y(Y_1)} = \psi_{\alpha_m}$. 
Applying corollary~\ref{cor:min-row}
repeatedly, 
we get that for each $1 \leq m \leq n-1$, 
$$\int_{[U_m^{m-1}]}\varphi(ug)\psi_{y(Y_1)}^{-1}(u)\,du=
\int_{[N]}\varphi(ng)\psi_{y(Y_1)}^{-1}(n)\,dn\,,$$
which is exactly
$$\int_{[N]}\varphi(ng)\psi_{\alpha_m}^{-1}(n)\,dn\,.$$

\item{\bf Next-to-minimal representation - rank 1.} 
    Let $\pi$ be an irreducible next-to-minimal automorphic representation of $\SL_n(\ads)$ and let $\varphi \in \pi$.
    Recalling that $U_m = U_m^0$ and applying lemma~\ref{TheoremB:Lemma1} with $\psi_0 = \psi_{y(Y_1)} = \psi_{\alpha_m}$ we get
    \begin{equation}
        \intl_{[U_m]} \varphi(ug) \psi_{y(Y_1)}^{-1}(u) \, du =
        \intl_{[U_m^1]} \varphi(ug) \psi_0^{-1}(u) \, du 
    \end{equation}
    since $\psi_0 \psi_{\alpha_{m-1}} = \psi_{\alpha_m,\alpha_{m-1}} = \psi_u$ for some $u$ that is not in the closure of the next-to-minimal orbit and thus the non-constant modes in lemma~\ref{TheoremB:Lemma1} can be expressed as Fourier coefficients $\mathcal{F}_{s,u}$ with $s = s_{U_m^1}$ from \eqref{eq:s-Umi} which vanish according to theorem~\ref{thm:ggsglobal}.
    
    Let us make an iteration in $1 \leq i \leq m-2$.
    Using lemma~\ref{TheoremB:Lemma1} we have that
    \begin{multline}
        \label{eq:ntm-rank1-induction}
        \intl_{[U_m^i]} \varphi(ug) \psi_0^{-1}(u) \, du = \\
        = \intl_{[U_m^{i+1}]} \hspace{-0.7em} \varphi(ug) \psi_0^{-1}(u) \, du + \hspace{-1.6em} \sum_{\gamma \in \Lambda_{m-i-1}(\psi_{\alpha_m}\!)} \, \intl_{[U_m^{i+1}]} \hspace{-0.7em} \varphi(u\hat\iota(\gamma) g) \psi_{\alpha_{m-i-1}, \alpha_m}^{-1}(u) \, du \, .
    \end{multline}
    
    Since $\psi_{\alpha_m, \alpha_{m-i-1}}$ is a character on $N$ trivial on the complement of $U_m^{i+1}$ we can expand the second term further with lemma~\ref{lem:col-expansion} (or remark~\ref{rem:col-expansion} if $m-i-1 = 2$ and $\psi_{\alpha_m, \alpha_{m-i-1}}$ is not trivial along at least two adjacent roots other than $\alpha_1$). This would lead to characters $\psi_u = \psi_{\alpha_m, \alpha_{m-i-1}, \alpha_{m-i-2}}$ \linebreak[3](or $\psi_u = \psi_{\alpha_m, \alpha_{m-i-1}}\psi_y$ with $y \in {}^t \lie c_2(F)$ respectively) where $u$ is not in the closure of the next-to-minimal orbit. 
    Then, $\mathcal{F}_{s,u}$ with $s = s_{U_m^{i+2}}$ from \eqref{eq:s-Umi} vanishes according to theorem~\ref{thm:ggsglobal} and the second term only receives the constant mode contribution.
    Repeating these arguments for the second term in \eqref{eq:ntm-rank1-induction}, it becomes
    \begin{equation}
        \sum_{\gamma \in \Lambda_{m-i-1}(\psi_{\alpha_m}\!)} \, \intl_{[V_m]} \varphi(u\hat\iota(\gamma) g) \psi_{\alpha_{m-i-1}, \alpha_m}^{-1}(u) \, du\,.             
    \end{equation}

    Iterating over $i$, starting from $i = 1$ above, we get that
    \begin{multline}
        \label{eq:ntm-rank-1-cols-done}
        \intl_{[U_m]} \varphi(ug) \psi_{y(Y_1)}^{-1}(u) \, du = \\
        \intl_{[V_m]} \varphi(ug) \psi_{\alpha_m}^{-1}(u) \, du + \sum_{j=1}^{m-2} \sum_{\gamma \in \Lambda_j(\psi_{\alpha_m}\!)}\, \intl_{[V_m]} \varphi(u\hat\iota(\gamma) g) \psi_{\alpha_j, \alpha_m}^{-1}(u) \, du \, .
    \end{multline}
    
    For $m = 1$, $U_1 = V_1$ and for $m = 2$ we only get the first term in \eqref{eq:ntm-rank-1-cols-done}.

    We will now use the methods of section~\ref{sec:thmA} to expand along rows.
    Using corollary~\ref{cor:ntm-row} case~\ref{itm:ntm-double}, we see that the second term in \eqref{eq:ntm-rank-1-cols-done} does not get any further contributions when expanding to $N$.
    Starting with the first term in \eqref{eq:ntm-rank-1-cols-done} and using corollary~\ref{cor:ntm-row} first with case~\ref{itm:ntm-single-i} to $V_{m+1}$ and then repeatedly with cases~\ref{itm:ntm-single-j} and \ref{itm:ntm-double} it becomes
    \begin{multline}
        \intl_{[V_{m+1}]} \varphi(ug) \psi_{\alpha_m}^{-1}(u) \, du = \\
        = \intl_{[N]} \varphi(ng) \psi_{\alpha_m}^{-1}(n) \, dn + \sum_{i=m+2}^{n-1} \sum_{\gamma \in \Gamma_i(\psi_{\alpha_m}\!)} \, \intl_{[N]} \varphi(n \iota(\gamma) g) \psi_{\alpha_m, \alpha_i}^{-1}(n) \, dn\,.
    \end{multline}

    Lastly,
    \begin{multline}
        \intl_{[U_m]} \varphi(ug) \psi_{y(Y_1)}^{-1}(u) \, du = \intl_{[N]} \varphi(ng) \psi_{\alpha_m}^{-1}(n) \, dn +{} \\ 
        + \sum_{j=1}^{m-2} \sum_{\gamma \in \Lambda_j(\psi_{\alpha_m}\!)}\, \intl_{[N]} \varphi(n\hat\iota(\gamma) g) \psi_{\alpha_j, \alpha_m}^{-1}(n) \, dn +{} \\ 
        + \sum_{i=m+2}^{n-1} \sum_{\gamma \in \Gamma_i(\psi_{\alpha_m}\!)} \, \intl_{[N]} \varphi(n \iota(\gamma) g) \psi_{\alpha_m, \alpha_i}^{-1}(n) \, dn \, . 
    \end{multline}

\item{\bf Next-to-minimal representation - rank 2.}
Let $\pi$ be an irreducible next-to-minimal automorphic representation of $\SL_n(\BA)$ and let $\varphi \in \pi$. 
We start from the integral
\begin{equation*}
\int_{[U_m]} \varphi (ug) \psi_{y(Y_2)}^{-1}(u) \,du\,.
\end{equation*}

For each root $\alpha$, let $X_{\alpha}$ be the corresponding one-dimensional root subgroup in $\SL_n$. Let 
$$C_1 = X_{e_m-e_{m+2}} \prod_{i=1}^{m-2} X_{e_i - e_{m+2}}\,,$$
and 
$$R_1 =  X_{e_{m-1}-e_{m}} \prod_{i=1}^{m-2} X_{e_{m-1} - e_{i}}\,.$$
Then $C_1$ is a subgroup of $U_m$. Let $U_m'$ be the subgroup of $U_m$ 
with $C_1$-part identically zero. Then one can see that the quadruple 
$$(U_m', C_1, R_1, \psi_{y(Y_2)})$$
satisfies all the conditions of Lemma \ref{exchangeroots}. By this lemma, 
\begin{align*}
\begin{split}
& \ \int_{[U_m]} \varphi (ug) \psi_{y(Y_2)}^{-1}(u) du\\
= & \ \int_{C_1(\BA)}\int_{[R_1U_m']} \varphi (ucg) \psi_{y(Y_2)}^{-1}(u)\, du\,dc\,.
\end{split}
\end{align*}

Let 
$$C_2 = \prod_{i=1}^{m-2} X_{e_i - e_{m+1}}\,,$$
and 
$$R_2 =  \prod_{i=1}^{m-2} X_{e_{m} - e_{i}}\,.$$
Then $C_2$ is a subgroup of $R_1U_m'$. Let $U_m''$ be the subgroup of $R_1U_m'$ 
with $C_2$-part identically zero. Then one can see that the quadruple 
$$(U_m'', C_2, R_2, \psi_{y(Y_2)})$$
satisfies all the conditions of Lemma \ref{exchangeroots}. Applying this lemma and by changing of variables, 
\begin{align}\label{theoremB:part2-equ1}
\begin{split}
& \ \int_{C_1(\BA)}\int_{[R_1U_m']} \varphi (ucg) \psi_{y(Y_2)}^{-1}(u) \,du\,dc\\
= & \ \int_{C_1(\BA)}\int_{C_2(\BA)}\int_{[R_2U_m'']} \varphi (uc_2c_1g) \psi_{y(Y_2)}^{-1}(u) \,du\,dc_2\,dc_1\\
= & \ \int_{(C_1C_2)(\BA)}\int_{[R_2U_m'']} \varphi (ucg) \psi_{y(Y_2)}^{-1}(u) \,du\,dc\,.
\end{split}
\end{align}

Let $\omega$ be the Weyl element sending torus elements 
$$(t_1, t_2, \ldots, t_n)$$
to torus elements 
$$(t_{m-1}, t_{m+2}, t_m, t_{m+1}, t_1, t_2, \ldots, t_{m-2}, t_{m+3}, t_{m+4}, \ldots, t_n)\,.$$
Conjugating $\omega$ cross from left, the integral
in \eqref{theoremB:part2-equ1} becomes
\begin{equation}\label{theoremB:part2-equ2}
\int_{C(\BA)}\int_{[U_m^{\omega}]} \varphi (u\omega cg) \psi_{y(Y_2)}^{\omega, -1}(u) \,du\,dc\,,
\end{equation}
where $U_m^{\omega} = \omega R_2U_m'' \omega^{-1}$, $C=C_1C_2$,
for $u \in U_m^{\omega}$, 
$\psi_{y(Y_2)}^{\omega}(u) = \psi_{y(Y_2)}(\omega^{-1} u \omega)$.  
$$U_m^{\omega} = U_m^{\omega,1}V_1\,,$$
where elements $u \in U_m^{\omega,1}$ have the following form
$$\begin{pmatrix}
I_2 & 0 \\
0 & u'
\end{pmatrix}\,,$$
and $U_m^{\omega,1}$ normalizes $V_1$. 
Recall that  $V_i$ be unipotent radical of parabolic subgroup of type $(1^i,n-i)$.
Note that $\psi_{y(Y_2)}^{\omega}|_{V_1}=\psi_{\alpha_1}$, 
$\psi_{y(Y_2)}^{\omega}|_{U_m^{\omega,1}}=\psi_{\alpha_3}$. 
Recall that $\alpha_1=e_1-e_2$, $\alpha_3=e_3-e_4$. 
Hence, the integral
in \eqref{theoremB:part2-equ2} becomes
\begin{equation}\label{theoremB:part2-equ3}
\int_{C(\BA)}\int_{[U_m^{\omega,1}]} \int_{[V_1]} \varphi (vu\omega cg)\psi_{\alpha_1}^{-1}(v) \psi_{\alpha_3}^{-1}(u) \,dv\,du\,dc\,.
\end{equation}

Since $\pi$ is an irreducible next-to-minimal automorphic representation of $\SL_n(\BA)$, 
by corollary~\ref{cor:ntm-row}, case~\ref{itm:ntm-single-i}, 
the integral
in \eqref{theoremB:part2-equ3} becomes
\begin{equation}\label{theoremB:part2-equ4}
\int_{C(\BA)}\int_{[U_m^{\omega,1}]} \int_{[V_2]} \varphi (vu\omega cg)\psi_{\alpha_1}^{-1}(v) \psi_{\alpha_3}^{-1}(u) \,dv\,du\,dc\,.
\end{equation}
$U_m^{\omega,1}$ still normalizes $V_2$, and 
$$U_m^{\omega,1}V_2 = U_m^{\omega,2}V_3\,,$$
where elements $u \in U_m^{\omega,2}$ have the following form
$$\begin{pmatrix}
I_4 & 0 \\
0 & u''
\end{pmatrix}\,,$$
$u''$ is in the radical of the parabolic subgroup of type $(m-2,n-m-2)$ in $\SL_{n-4}$, 
and $U_m^{\omega,2}$ normalizes $V_3$. 
Note that $\psi_{y(Y_2)}^{\omega}|_{V_3}=\psi_{\alpha_1,\alpha_3}$ and 
$\psi_{y(Y_2)}^{\omega}|_{U_m^{\omega,2}}$ is the trivial character. 
By corollary~\ref{cor:ntm-row}, case~\ref{itm:ntm-double}, 
the integral
in \eqref{theoremB:part2-equ4} becomes
\begin{equation}\label{theoremB:part2-equ5}
\int_{C(\BA)}\int_{[U_m^{\omega,2}]} \int_{[V_4]} \varphi (vu\omega cg)\psi_{\alpha_1,\alpha_3}^{-1}(v) \,dv\,du\,dc\,.
\end{equation}

Applying corollary~\ref{cor:ntm-row}, case~\ref{itm:ntm-double}, repeatedly,
the integral
in \eqref{theoremB:part2-equ5} becomes
\begin{equation*}
\int_{C(\BA)}\int_{[U_m^{\omega,2}]} \int_{[N]} \varphi (nu\omega cg)\psi_{\alpha_1,\alpha_3}^{-1}(n) \,dn\,du\,dc\,,
\end{equation*}
which becomes 
\begin{equation}\label{theoremB:part2-equ6}
\int_{C(\BA)}\int_{[U_m^{\omega,2}]} \int_{[N]} \varphi (n\omega cg)\psi_{\alpha_1,\alpha_3}^{-1}(n) \,dn\,du\,dc\,,
\end{equation}
by changing of variables. 
Since $\int_{[U_m^{\omega,2}]}du=1$, we have obtained that 
\begin{equation*}
\int_{[U_m]} \varphi (ug) \psi_{y(Y_2)}^{-1}(u) \,du
= \int_{C(\BA)}\int_{[N]} \varphi (n\omega cg)\psi_{\alpha_1,\alpha_3}^{-1}(n) \,dn\,dc\,.
\end{equation*}
\end{itemize-small}
This completes the proof of Theorem \ref{thm:max-parabolic}. \qed

\section{Proof of theorems \ref{thm:min-coeff} and \ref{thm:ntm-coeff}}
\label{sec:orbit-coefficients}
    
\textbf{Proof of Theorem \ref{thm:min-coeff}.}
Let $\pi$ be any irreducible automorphic representation of $\SL_n(\BA)$ and let $\varphi \in \pi$. The generalized Fourier coefficient of $\varphi$ attached to the partition $[21^{n-2}]$ has been defined in Section~\ref{sec:fourier}. We recall it as follows. 

Let $s=(1, -1, 0, \ldots, 0)$, and let $u = J_{[21^{n-2}]}$ which is a matrix zero everywhere except the (2,1) entry being $1$.  Then the generalized Fourier coefficient of $\varphi$ attached to the partition $[21^{n-2}]$ is as follows:
\begin{equation*}
\mathcal{F}^{[211\ldots]} (\varphi;g)  = \CF_{s,u} (\varphi;g)= \int_{[N_s]} \varphi(ng)\psi_u^{-1}(n) \,dn\,,
\end{equation*}
where elements in the one-dimensional unipotent $N_s$ have the form 
$$\begin{pmatrix}
1&* & 0\\
0&1&0\\
0 &0& I_{n-2}
\end{pmatrix}\,.$$

Let $X=\prod_{i=3}^{n} X_{e_1-e_i}$ and $Y=\prod_{i=3}^{n} X_{e_i-e_2}$. 
Then one can see that $Y(F)$ can be identified with the character space of $[X]$ as follows:
given $y \in Y(F)$, $\psi_y(x)=\psi_u([x,y])$, for any $x \in [X]$. 
Note that both $X$ and $Y$ normalize $N_s$. Taking the Fourier expansion of $\CF_{s,u} (\varphi)(g)$ along $[X]$, we obtain that 
$$\CF_{s,u} (\varphi;g)=\sum_{y\in Y(F)} \int_{[X]}\int_{[N_s]} \varphi(xng)\psi_u^{-1}(n) \psi_y^{-1}(x)\,dn\,dx\,.$$
Since $y^{-1}\in Y(F)$ and $\varphi$ is automorphic, 
the above integral becomes
\begin{align*}
& \ \sum_{y\in Y(F)} \int_{[X]}\int_{[N_s]} \varphi(xng)\psi_u^{-1}(n) \psi_y^{-1}(x)\,dn\,dx\\
= & \ \sum_{y\in Y(F)} \int_{[X]}\int_{[N_s]} \varphi(y^{-1}xng)\psi_u^{-1}(n) \psi_y^{-1}(x)\,dn\,dx\\
= & \ \sum_{y\in Y(F)} \int_{[X]}\int_{[N_s]} \varphi(y^{-1}xnyy^{-1}g)\psi_u^{-1}(n) \psi_y^{-1}(x)\,dn\,dx\\
= & \ \sum_{y\in Y(F)} \int_{[X]}\int_{[N_s]} \varphi(xn' y^{-1}g)\psi_u^{-1}(n) \psi_y^{-1}(x)\,dn\,dx\,,
\end{align*}
where $n'=n+[x,y]$. 
By changing variables, we obtain that 
\begin{align*}
 & \ \sum_{y\in Y(F)} \int_{[X]}\int_{[N_s]} \varphi(xn' y^{-1}g)\psi_u^{-1}(n) \psi_y^{-1}(x)dndx\\
= & \ \sum_{y\in Y(F)} \int_{[X]}\int_{[N_s]} \varphi(xn y^{-1}g)\psi_u^{-1}(n) 
\psi_u^{-1}(-[x,y]) 
\psi_y^{-1}(x)\,dn\,dx\,.
\end{align*}

Note that 
\begin{align*}
& \ \psi_u^{-1}(-[x,y]) 
\psi_y^{-1}(x)\\
= & \ \psi_u([x,y]) 
\psi_u(-[x,y]) \\
= & \ 1\,.
\end{align*}
Hence, we have that 
$$\CF_{s,u} (\varphi;g) = \sum_{y\in Y(F)} \int_{[X]}\int_{[N_s]} \varphi(xn y^{-1}g)\psi_u^{-1}(n) 
\,dn\,dx\,.$$
Note that $XN_s=U_1$ and $\psi_u=\psi_{\alpha_1}$. 
Therefore, we have that 
$$\CF_{s,u} (\varphi;g) = \sum_{y\in Y(F)} \int_{[U_1]} \varphi(u y^{-1}g)\psi_{\alpha_1}^{-1}(u) 
\,du\,.$$

This completes the proof of Theorem \ref{thm:min-coeff}. \qed

\textbf{Proof of Theorem \ref{thm:ntm-coeff}.}
Let $\pi$ be any irreducible automorphic representation of $\SL_n(\BA)$ and let $\varphi \in \pi$. The generalized Fourier coefficient of $\varphi$ attached to the partition $[2^21^{n-4}]$ has also been defined in Section~\ref{sec:fourier}. We recall it as follows. 

Let $s=(1, -1, 1, -1, 0, \ldots, 0)$, and let $u = J_{[2^21^{n-4}]}$ which is a matrix zero everywhere except the (2,1) and (4,3) entries being $1$.  Then the generalized Fourier coefficient of $\varphi$ attached to the partition $[2^21^{n-4}]$ is as follows:
\begin{equation*}
\mathcal{F}^{[221\ldots]}(\varphi;g) = 
\CF_{s,u} (\varphi;g)= \int_{[N_s]} \varphi(ng)\psi_u^{-1}(n) \,dn\,,
\end{equation*}
where elements in $N_s$ have the form 
$$\begin{pmatrix}
1&* & 0 & * & 0\\
0&1&0&0&0\\
0&*&1&*&0\\
0&0&0&1&0\\
0 &0& 0&0&I_{n-4}
\end{pmatrix}\,.$$
Let $\omega$ be the Weyl element sending the torus element
$$(t_1, t_2, \ldots, t_n)$$
to the torus element
$$(t_1, t_3, t_4, t_2, t_5, t_6, \ldots, t_n)\,.$$
Conjugating $\omega$ across from left, we obtain that 
\begin{equation*}
\CF_{s,u} (\varphi;g)= \int_{[N_s^{\omega}]} \varphi(n\omega g)\psi_u^{\omega,-1}(n) \,dn\,,
\end{equation*}
where $N_s^{\omega} = \omega N_s \omega^{-1}$, and for $n \in N_s^{\omega}$,
$\psi_u^{\omega}(n) =\psi_u(\omega^{-1} n \omega)$. 
Elements in $n \in N_s^{\omega}$ have the following form 
$$n=n(z)=\begin{pmatrix}
I_2&z & 0\\
0&I_2&0\\
0 &0& I_{n-4}
\end{pmatrix}\,,$$
and $\psi_u^{\omega}(n)=\psi(z_{1,2}+z_{2,1})$.

Let 
$$X'=\prod_{i=5}^{n} X_{e_1-e_i}\prod_{i=5}^{n} X_{e_2-e_i}$$
 and 
$$Y'=\prod_{i=5}^{n} X_{e_i-e_4}\prod_{i=5}^{n} X_{e_i-e_3}\,.$$
Then one can see that $Y'(F)$ can be identified with the character space of $[X']$ as follows:
given $y \in Y'(F)$, $\psi_y(x)=\psi_u^{\omega}([x,y])$, for any $x \in [X']$. 
Note that both $X'$ and $Y'$ normalize $N_s$. Taking the Fourier expansion of $\CF_{s,u} (\varphi)(g)$ along $[X']$, we obtain that 
$$\CF_{s,u} (\varphi;g)=\sum_{y\in Y'(F)} \int_{[X']}\int_{[N_s^{\omega}]} \varphi(xn\omega g)\psi_u^{\omega,-1}(n) \psi_y^{-1}(x)\,dn\,dx\,.$$
Since $y^{-1}\in Y'(F)$ and $\varphi$ is automorphic, 
the above integral becomes
\begin{align*}
& \ \sum_{y\in Y'(F)} \int_{[X']}\int_{[N_s^{\omega}]} \varphi(xn\omega g)\psi_u^{\omega,-1}(n) \psi_y^{-1}(x)\,dn\,dx\\
= & \ \sum_{y\in Y'(F)} \int_{[X']}\int_{[N_s^{\omega}]} \varphi(y^{-1}xn\omega g)\psi_u^{\omega,-1}(n) \psi_y^{-1}(x)\,dn\,dx\\
= & \ \sum_{y\in Y'(F)} \int_{[X']}\int_{[N_s^{\omega}]} \varphi(y^{-1}xnyy^{-1}\omega g)\psi_u^{\omega,-1}(n) \psi_y^{-1}(x)\,dn\,dx\\
= & \ \sum_{y\in Y'(F)} \int_{[X']}\int_{[N_s^{\omega}]} \varphi(xn' y^{-1}\omega g)\psi_u^{\omega,-1}(n) \psi_y^{-1}(x)\,dn\,dx\,,
\end{align*}
where $n'=n+[x,y]$. 
By changing variables, we obtain that 
\begin{align*}
 & \ \sum_{y\in Y'(F)} \int_{[X']}\int_{[N_s^{\omega}]} \varphi(xn' y^{-1}\omega g)\psi_u^{\omega,-1}(n) \psi_y^{-1}(x)\,dn\,dx\\
= & \ \sum_{y\in Y'(F)} \int_{[X']}\int_{[N_s^{\omega}]} \varphi(xn y^{-1}\omega g)\psi_u^{\omega,-1}(n) 
\psi_u^{\omega,-1}(-[x,y]) 
\psi_y^{-1}(x)\,dn\,dx\,.
\end{align*}

Note that 
\begin{align*}
& \ \psi_u^{\omega,-1}(-[x,y]) 
\psi_y^{-1}(x)\\
= & \ \psi_u^{\omega}([x,y]) 
\psi_u^{\omega}(-[x,y]) \\
= & \ 1\,.
\end{align*}
Hence, we have that 
$$\CF_{s,u} (\varphi;g) = \sum_{y\in Y'(F)} \int_{[X']}\int_{[N_s^{\omega}]} \varphi(xn y^{-1} \omega g)\psi_u^{-1}(n) 
\,dn\,dx\,.$$
Note that $X'N_s^{\omega}=U_2$ and $\psi_u=\psi_{y(Y_2)}$, using the notation from section~\ref{sec:thmB}.
Therefore, we have that 
$$\CF_{s,u} (\varphi;g) = \sum_{y\in Y'(F)} \int_{[U_2]} \varphi(u y^{-1}\omega g)\psi_{y(Y_2)}^{-1}(u) 
\,du\,dx\,.$$

This completes the proof of Theorem \ref{thm:ntm-coeff}. \qed

\section{Applications}
\label{sec:sl5}

As is evident from table~\ref{tab:duality}, the case $\SL_5$ appears in the list of symmetry and duality groups in string theory. It is related to compactification of type II string theory on a three-torus $T^3$ from ten to seven spacetime dimensions. Fourier coefficients of automorphic forms on $\SL_5$ are related to non-perturbative effects as discussed in the introduction. Therefore we analyse here in some detail the structure of Fourier coefficients for automorphic forms attached to a minimal or next-to-minimal automorphic representation of $\SL_5$ that are relevant to the first two higher-derivative corrections in four-graviton scattering amplitudes.

We will give a detailed description of how the formalism developed above can be used to calculate explicit expressions for Fourier coefficients on maximal parabolic subgroups for automorphic forms attached to a minimal or next-to-minimal automorphic representation. Following a general discussion, we will treat two explicit examples for $n = 5$.

\subsection{Generalities}
With applications to string theory in mind, throughout this section we are restricting to $F = \BQ$ and let $\BA \equiv \BA_{\BQ}$. The types of expressions that are of interest are of the form:
\begin{equation}
    \mathcal{F}^{\BR}(\varphi, \psi; g) = \intl_{U(\BZ) \bs U(\BR)}\varphi(ug) \psi^{-1}(u) \, du\,,
    \label{eqn:realunipotentcoeff}
\end{equation}
where $U(\BR)$ is a parabolic subgroup of $\mathrm{G}(\BR)$, $\psi$ is some rank-1 or rank-2 character on $U(\BR)$ and $\varphi$ is an automorphic form in the minimal- or next-to-minimal automorphic representations of $\mathrm{G}(\BR)$. Any such coefficient can be brought to a standard form using the action of the arithmetic Levi subgroup $L(\BZ)$. For rank-1 this form is $\psi = \psi_{y(kY_1)}$ for some integer $k\neq 0$ and for rank-2 one has $\psi(k_1 u_{m,m+1}+k_2 u_{m-1,m+2})$ for integers $k_1$ and $k_2$, cf.~\eqref{eq:psi-Y2}. For simplicity, we will restrict ourselves to the case $Y_{y(Y_2)}$ corresponding to $k_1=k_2=1$ and demonstrate how to apply theorem \ref{thm:max-parabolic}. The techniques demonstrated here allow for the calculation of all such Fourier coefficients for automorphic forms in the minimal and next-to-minimal representations on $\SL_n$.

In order to apply theorem \ref{thm:max-parabolic}, we first perform an adelic lift~\cite{FGKP18}
\begin{equation}
    \mathcal{F}^{\BR}(\varphi, \psi; g_\infty) = \mathcal{F}^{\BA}(\varphi, \psi; (g_\infty, I_n, I_n, \cdots)) = \intl_{[U]}\varphi(u(g_\infty, I_n, I_n, \dots)) \psi^{-1}(u)  \,du\,
    .
\end{equation}
The theorem now gives $\mathcal{F}^{\BA}$ in terms of adelic Whittaker functions. These Whittaker functions will then be evaluated using the adelic reduction formula
\begin{equation}
    W_\psi(\lambda, a) = \sum_{w_c w_0' \in \mathcal{C}_\psi}a^{(w_c w_0')^{-1}\lambda + \rho} M(w_c^{-1}, \lambda) W_{\psi^a}'(w_c^{-1}\lambda, 1)
    \label{eqn:reduction}
\end{equation}
of \cite{FKP14}. The power of this formula lies in that it expresses a degenerate Whittaker function evaluated on the Cartan torus of a group $\mathrm{G}(\BA)$ as a sum of generic Whittaker functions on a subgroup $\mathrm{G}'(\BA)$. This subgroup $\mathrm{G}'(\BA)$ is determined by deleting all nodes in the Dynkin diagram of $\mathrm{G}(\BA)$ on which $\psi$ is not supported. $\lambda$ denotes the weight of the Eisenstein series, $w_0'$ denotes the longest Weyl word on $\mathrm{G}'$, $\mathcal{C}_\psi$ denotes the set
\begin{equation}
    \mathcal{C}_\psi = \{ w \in \mathcal{W} \;| \;w \Pi' < 0 \}
\end{equation}
where $\Pi'$ is the set of simple roots of $\mathrm{G}'$ and $w_c$ is hence the summation variable and corresponds to a specific representative of the quotient Weyl group $\mathcal{W}/\mathcal{W}'$ described in~\cite{FKP14}. $\rho$ denotes the Weyl vector, $M$ denotes the intertwiner
\begin{equation}
    M(w, \lambda) = \prod_{\substack{\alpha > 0 \\ w\alpha < 0}} \frac{\xi(\langle \lambda | \alpha \rangle)}{\xi(\langle \lambda | \alpha \rangle + 1)}
\end{equation}
as featured in the Langlands constant term formula, where $\xi$ is the completed Riemann zeta function and $\psi^a$ denotes the ``twisted character'' both defined in appendix \ref{app:euler}.

The evaluation of a real Fourier coefficient $\mathcal{F}^{\BR}$ over a unipotent schematically looks like
\begin{equation}
    \begin{aligned}
        & \mathcal{F}^{\BR}(\varphi, \psi; g_\infty) = \mathcal{F}^{\BA}(\varphi, \psi; (g_\infty, I_n, I_n, \cdots)) && \text{Adelic lift} \\
        ={}& \sum_\psi \sum_{l\in \Lambda \text{ or } l\in\Gamma} W_{\psi}(l(g_\infty, I_n, I_n, \cdots)) && \text{Theorem \ref{thm:max-parabolic}} \\
        ={}& \sum_\psi \sum_{l\in \Lambda \text{ or } l\in\Gamma} W_{\psi}((n_\infty a_\infty k_\infty, n_2 a_2 k_2, n_3 a_3 k_3, \cdots)) && \text{Iwasawa-decomposition} \\
        ={}& \sum_\psi \left( \prod_{p\leq \infty} \psi_p(n_p) \right) \sum_{l\in \Lambda \text{ or } l\in\Gamma} W_{\psi}((a_\infty, a_2,a_3, \cdots)) && W_\psi(nak) = \psi(n) W_\psi(a) \\
        ={}& \sum_\psi \psi_{\infty}(n_{\infty}) \sum_{l\in \Lambda \text{ or } l\in\Gamma} \sum_{w} a^{\ldots} M(\cdots) W_{\psi^a}'(\cdots, 1) && \text{Reduction formula \eqref{eqn:reduction}}
    .
    \end{aligned}
\end{equation}
The fourth line extracts the unipotent $n_p$-dependence at each of the local places $p \leq \infty$. In the fifth line we have used that only the archimedean unipotent $n_\infty$ contributes. The reason that the $p$-adic unipotent matrices $n_p$ of the $p$-adic Iwasawa-decomposition of $l \in \mathrm{G}(F) \subset \mathrm{G}(\BQ_p)$ above drop out is as follows. In using theorem \ref{thm:max-parabolic}, we will be faced with evaluating Whittaker functions such as
\begin{equation}
    \begin{aligned}
        W_{\alpha_j, \alpha_m}(\hat{\iota}(\lambda_{j}) g) &\quad{}\text{for}\quad{} j \leq m-2 \quad{}\text{where}\quad{} \lambda_{j} \in \Lambda_{j} \quad{}\text{and}\quad{} \\
        W_{\alpha_m, \alpha_i}(\iota(\gamma_i) g) &\quad{}\text{for}\quad{} i \geq m+2 \quad{}\text{where}\quad{} \gamma_i \in \Gamma_i\,
        .
    \end{aligned}
\end{equation}
We have that $\gamma_i$ and $\lambda_{j}$ are embedded in $\SL_n$ as (cf.~\eqref{eq:iota})
\begin{equation}
    \hat{\iota}(\lambda_{j}) = 
    \left(
    \begin{smallmatrix}
        \lambda_{j} \\
        & I_{n-j}
    \end{smallmatrix}
    \right)
    \quad{}\text{and}\quad{}
    \iota(\gamma_i) = 
    \left(
    \begin{smallmatrix}
        I_i \\
        & \gamma_i
    \end{smallmatrix}
    \right)\,
    .
\end{equation}
It is clear from their block-diagonal form that the unipotent $n_p$ in the $p$-adic Iwasawa-decomposition of $\hat{\iota}(\lambda_j)$ (and $\iota(\gamma_i)$) will feature the same block-diagonal form. Since $W_{\alpha_j, \alpha_m}$ (and $W_{\alpha_m, \alpha_i}$) is only sensitive to the unipotent on rows $j$ and $m \geq j+2 > j$ (on rows $i$ and $m \leq i-2 \leq i$), the block diagonal structure of $n_p$ implies $\psi_{\alpha_j, \alpha_m; p}(n_p) = 1$ (and $\psi_{\alpha_m, \alpha_i; p}(n_p) = 1$).

For a real matrix $g \in \SL_n(\BR)$, we will denote its Iwasawa-decomposition
\begin{equation}
\label{eq:realIwa}
    g = n_\infty a_\infty k_\infty = 
    \left( 
    \begin{smallmatrix}
        1 & x_{12} & \cdots & \cdots & x_{1n} \\
          & 1 & \ddots & \ddots & \vdots \\
          & & \ddots & \ddots & \vdots \\
          & & & 1 & x_{n-1, n} \\
          & & & & 1
    \end{smallmatrix}
    \right)
    \left( 
    \begin{smallmatrix}
        y_1 & & & & \\
          & y_2/y_1 & & & \\
          & & \ddots & & \\
          & & & y_{n-1}/y_{n-2} & \\
          & & & & 1/y_{n-1}
    \end{smallmatrix}
    \right)
    k_{\infty}\,
    .
\end{equation}
Similarly, for a $p$-adic matrix $g \in \SL_n(\BQ_p)$ we denote it as
\begin{equation}
    g = n_p a_p k_p = 
    n_p
    \left( 
    \begin{smallmatrix}
        \eta_{1, p} & & & & \\
        & \eta_{2, p}/\eta_{1, p} & & & \\
          & & \ddots & & \\
          & & & \eta_{n-1, p}/\eta_{n-2, p} & \\
          & & & & 1/\eta_{n-1, p}
    \end{smallmatrix}
    \right)
    k_p\,
    .
\end{equation}
Appendix \ref{sec:iwasawa} contains closed formulae for the $x$'s and the $y$'s, as well as a closed formula for the $p$-adic norm $|\eta_{i, p}|_p$ of the $\eta$'s.

In what follows, we will make use of all formulae that are derived or stated in appendices \ref{app:euler}, \ref{sec:iwasawa} and \ref{sec:cosets} along with the following notation
\begin{itemize}\setlength{\itemsep}{1mm}
    \item A prime on a variable, eg.\ $x'$, generally denotes $x' \neq 0$.
    \item For sums we write $\displaystyle \sum_x \equiv \sum_{x \in \BQ}$.
    \item We write $\displaystyle \sum_{x'} f(x) \equiv \sum_{x \in \BQ \bs \{0\}} f(x)$ and $\displaystyle \sum_{x' \in \BZ} f(x) \equiv \sum_{x \in \BZ \bs \{0\}} f(x)$. Note that the prime is used to indicate whether or not zero is included in the sum but the prime is omitted in the summand.
    \item For products we write $\displaystyle \prod_p \equiv \prod_{p \text{ prime}}$. Writing $\displaystyle \prod_{p \leq \infty}$ denotes the product over all primes $p$ (the non-archimedean places) as well as the element $p = \infty$ (the archimedean place).
    \item For $x \in \BR$ we denote $\e{x} \equiv e^{2\pi i x}$.
\end{itemize}

\subsection{Example: Rank-1 coefficient of \texorpdfstring{$\pi_{\text{min}}$}{pimin} on \texorpdfstring{$P_{\alpha_4}\subset \SL_5$}{Palpha4 in SL5}}

Here, we will calculate the real rank-1 Fourier coefficient \eqref{eqn:realunipotentcoeff} for the minimal Eisenstein series $E(\lambda; g)$ with $\lambda = 2s\Lambda_1 - \rho$ in the maximal parabolic
\begin{equation}
    P_{\alpha_4} = \GL(4) \times \GL(1) \times U_{\alpha_4} \subset \SL(5) \quad{}\text{subject to } \quad \det(\GL(4) \times \GL(1)) = 1\,,
\end{equation}
associated with removing the ``last'' node in the Dynkin diagram of $\SL(5)$. The unipotent radical is
\begin{equation}
    U(\BR) = U_{\alpha_4}(\BR) = 
    \left\{
    \left(
    \begin{smallmatrix}
        1 & & & & * \\
          &1& & & * \\
          & &1& & * \\
          & & &1& * \\
          & & & & 1
    \end{smallmatrix}
    \right)
    \right\}\,
    .
\end{equation}

Theorem \ref{thm:max-parabolic} gives for the unramified character $\psi_{y(Y_1)}$ that
\begin{align}
    \mathcal{F}^{\BA}(E(2s\Lambda_1 - \rho), \psi_{y(Y_1)}; g) ={}& W_{\alpha_4}(g) \,
    .
\end{align}

\begin{table}[t]
    \begin{align*}
        \begin{tabu}{cl|c|c|c}
            &w_c & \left\langle w_c^{-1} \lambda + \rho | \alpha_4 \right\rangle & M\left( w_c^{-1}, \lambda \right) & \left( w_c w_0' \right)^{-1} \lambda + \rho \\
            \hline
            &\operatorname{Id} & 0 & 1 & \dots \\
            &w_{1} & 0 & \dots & \dots \\
            &w_{12} & 0 & \dots & \dots \\
            *&w_{123} & 2\left( s - \frac{3}{2} \right) & \frac{\xi(2s-3)}{\xi(2s)} & [0, 0, 0, 5 - 2s] \\
        \end{tabu}
    \end{align*}
    \caption{\label{tab:2}Data for the reduction formula \eqref{eqn:reduction} to evaluate $W_{\alpha_4}(a)$ on $\SL_5$ with $\lambda = 2s\Lambda_1 - \rho$. The star indicates the one and only row that contributes in the sum over Weyl words.}
\end{table}
The Whittaker function is found by the reduction formula with data given in table~\ref{tab:2}. In this case, there is no diagonally embedded rational matrix $l$, or equivalently $l = I_5$, in the general procedure and hence we have $|\eta_{1, p}|_p = |\eta_{2, p}|_p = |\eta_{3, 4}|_p = |\eta_{4, p}|_p = 1$. We get
\begin{align}
  \begin{aligned}
    & W_{\alpha_4}(\lambda; (g_\infty, I_5, I_5, \cdots)) = \\
    ={}& \e{x_{45}} \left( y_4^{5-2s} \frac{\xi\left( 2s-3 \right)}{\xi\left( 2s \right)} \prod_{p < \infty} |\eta_{4, p}|_p^{5-2s} \right) B_{s-3/2}\left( \frac{y_4^2}{y_3}, 1 \right) \\
    & \times \prod_{p < \infty} \gamma_{p}\left( \frac{\eta_{4, p}^2}{\eta_{3, p}} \right) \left( 1 - p^{-2(s-3/2)} \right) \frac{1-p^{-2(s-3/2)+1} \left|\frac{\eta_{4, p}^2}{\eta_{3, p}} \right|_{p}^{2(s-3/2)-1}}{1-p^{-2(s-3/2)+1}} \\
    ={}& \e{x_{45}} y_4^{5-2s} \frac{1}{\xi\left( 2s \right)} 2 \left| \frac{y_4^2}{y_3}\right|_\infty^{s-2} K_{s-2}\left( 2\pi \left| \frac{y_4^2}{y_3} \right|_\infty \right) \\
    ={}& 2 \e{x_{45}} y_3^{2-s} y_4 \frac{1}{\xi\left( 2s \right)}  K_{s-2}\left( 2\pi \left| \frac{y_4^2}{y_3} \right|_\infty \right) = \mathcal{F}^{\BR}(E(2s\Lambda_1 - \rho), \psi_{y(Y_1)}; g_\infty)\,.
    \label{eqn:unramifiedminimal}
  \end{aligned}
\end{align}
The $x$'s and $y$'s are the Iwasawa coordinates for the matrix $g_\infty$ as in~\eqref{eq:realIwa}. The function $B_s$ that appears is a more compact way of writing the $\SL_2$ Whittaker vector defined explicitly in~\eqref{eq:SL2Whitt}.

Parameterizing $g_\infty$ as
\begin{equation}
    g_\infty = ue = 
    \begin{psmallmatrix}
        I_4 & Q \\
        0   & 1
    \end{psmallmatrix}
    \begin{psmallmatrix}
        r^{-1/4}e_4 & 0 \\
        0 & r
    \end{psmallmatrix}
    \quad{}\text{where}\quad{} e_4 \in \SL_4(\BR)\,,
\end{equation}
we get in particular that
\begin{equation}
    y_3 = r^{-3/4}||N e_4|| \quad{}\text{and}\quad{} y_4 = r^{-1}\,,
\end{equation}
where $N = 
\displaystyle
\begin{psmallmatrix}
    0 & 0 & 0 & 1
\end{psmallmatrix}
$ so that $N e_4$ is equal to the last row in $e_4$. This is obtained using the formula \eqref{eqn:realdilatons}. We get in particular that
\begin{equation}
    y_3^{2-s} y_4 = r^{2s-5}\left( r^{-5/4} ||N e_4 || \right)^{s-2}  \quad{}\text{and}\quad{} \frac{y_4^2}{y_3} = r^{-5/4} ||N e_4 ||
   \, .
\end{equation}
The more general (real) ramified Fourier coefficient has the expression
\begin{equation}
    \begin{aligned}
        &\int E\left( 2s\Lambda_1 - \rho; 
        \begin{psmallmatrix}
            1 &   &   &   & u_1 \\
            & 1 &   &   & u_2 \\
            &   & 1 &   & u_3 \\
            &   &   & 1 & u_4 \\
            &   &   &   & 1
        \end{psmallmatrix}
        g_\infty
        \right)
        \overline{\e{m_1 u_1 + m_2 u_2 + m_3 u_3 + m'_4 u_4}} d^4u
        = \\
        ={}& 
        \e{x_{45}} r^{2s-5} \frac{2}{\xi(2s)} \sigma_{4-2s}(k) \left( r^{-5/4} ||N e_4 || \right)^{s-2} K_{s-2}\left( 2\pi r^{-5/4}||N e_4|| \right)\\
        =&\e{x_{45}} r^{\frac{3s}{4}-\frac{5}{2}} \frac{2}{\xi(2s)} \frac{\sigma_{2s-4}(k)}{|k|^{s-2}} ||\tilde{N} e_4 ||^{s-2} K_{s-2}\left( 2\pi|k|  r^{-5/4}||\tilde{N} e_4|| \right)\label{eq:minFC}
    \end{aligned}
\end{equation}
for integer $m$'s while for non-integer rational $m$'s it vanishes. Here $g_\infty$ has been parametrized as above, $N = 
\displaystyle
\begin{psmallmatrix}
    m_1 & m_2 & m_3 & m'_4
\end{psmallmatrix}=k \tilde{N}
$, $k = \gcd(N)$ and $m'_4 \neq 0$. This expression can also be found by starting from $\psi_{y(kY_1)}$ for the standard Fourier coefficient instead. This corresponds to $N = 
\displaystyle
\begin{psmallmatrix}
    0 & 0 & 0 &k
\end{psmallmatrix}$  and its $L(\BZ)$ orbit gives the general expression~\eqref{eq:minFC}. 
Formula~\eqref{eq:minFC} agrees with~\cite[Eq.~(H.37)]{GMV15} where the Fourier coefficients were computed by Poisson resummation technique after a translation of conventions.

\subsection{Example: Rank-1 coefficient of \texorpdfstring{$\pi_{\text{ntm}}$}{pintm} on \texorpdfstring{$P_{\alpha_4} \subset \SL_5$}{Palpha4 in SL5}}

Here, we will calculate the real rank-1\footnote{There is no rank-2 character for this parabolic.} Fourier coefficient \eqref{eqn:realunipotentcoeff} for the next-to-minimal Eisenstein series $E(\lambda; g)$ with $\lambda = 2s\Lambda_2 - \rho$ in the maximal parabolic
\begin{equation}
    P_{\alpha_4} = \GL(4) \times \GL(1) \times U_{\alpha_4} \subset \SL(5) \quad{}\text{subject to } \quad \det(\GL(4) \times \GL(1)) = 1
\end{equation}
associated with removing the ``last'' node in the Dynkin diagram of $\SL(5)$. The unipotent radical is
\begin{equation}
    U(\BR) = U_{\beta_4}(\BR) = 
    \left\{
    \left(
    \begin{smallmatrix}
        1 & & & & * \\
          &1& & & * \\
          & &1& & * \\
          & & &1& * \\
          & & & & 1
    \end{smallmatrix}
    \right)
    \right\}\,
    .
\end{equation}

Theorem \ref{thm:max-parabolic} gives
\begin{align}
  \begin{aligned}
      & \mathcal{F}^{\BA}(E(2s\Lambda_2 - \rho), \psi_{y(Y_1)}; g) = \\
      ={}& W_{\alpha_4}(g) + \sum_{\lambda_1 \in \Lambda_1} W_{\alpha_1, \alpha_4}(\lambda_1 g) + \sum_{\lambda_2 \in \Lambda_2} W_{\alpha_2, \alpha_4}(\lambda_2 g) \\
      ={}& W_{\alpha_4}(g) \\
      &+{} 
      \sum_{z'} W_{\alpha_1, \alpha_4}
      \Bigg(
      \underbrace{
          \left(
          \begin{smallmatrix}
              z \\
              & 1 \\
              & & 1 \\
              & & & 1 \\
              & & & & 1/z
          \end{smallmatrix}
          \right)
      }_{l_z}
      g
      \Bigg)
      \\
      &+{} 
      \sum_{x', y} W_{\alpha_2, \alpha_4}
      \bigg( 
      \underbrace{
      \left(
      \begin{smallmatrix}
          x^{-1} \\
          y & x \\
          & & I_3 \\
      \end{smallmatrix}
      \right)
      }_{l_{xy}}
      g
      \bigg)
      \\
      &+{} 
      \sum_{x'} W_{\alpha_2, \alpha_4}
      \bigg(
      \underbrace{
      \left(
      \begin{smallmatrix}
          0 & -x^{-1} \\
          x & 0 \\
          & & I_3 \\
      \end{smallmatrix}
      \right)
      }_{l_x}
      g
      \bigg)\,,
  \end{aligned}
\end{align}
using the representatives derived in appendix \ref{sec:cosets}.

\begin{table}[t]
    \begin{align*}
        \begin{tabu}{cl|c|c|c}
            &w_c & \left\langle w_c^{-1} \lambda + \rho | \alpha_4 \right\rangle & M\left( w_c^{-1}, \lambda \right) & \left( w_c w_0' \right)^{-1} \lambda + \rho \\
            \hline
            &\operatorname{Id} & 0 & 1 & \dots \\
            &w_{2} & 0 & \dots & \dots \\
            &w_{21} & 0 & \dots & \dots \\
            *&w_{23} & 2\left( s - 1 \right) & \frac{\xi(2s-2)}{\xi(2s)} & [2s-1, 0, 0, 4 - 2s] \\
            *&w_{213} & 2\left( s - 1 \right) & \frac{\xi(2s-2)^2}{\xi(2s) \xi(2s-1)} & [3-2s, 2s-2, 0, 4 - 2s] \\
            *&w_{2132} & 2\left( s - 1 \right) & \frac{\xi(2s-3)\xi(2s-2)}{\xi(2s)\xi(2s-1)} & [0, 4-2s, 2s-3, 4 - 2s] \\
            &w_{213243} & 0 & \dots & \dots \\
        \end{tabu}
    \end{align*}
    \caption{\label{tab:3}Data for the reduction formula \eqref{eqn:reduction} to evaluate $W_{\alpha_4}(a)$ on $\SL_5$ with $\lambda = 2s\Lambda_2 - \rho$. The stars indicate which rows contribute in the sum over Weyl words.}
\end{table}
The first Whittaker function is found by the reduction formula with the data of table~\ref{tab:3}. In this case, there is no diagonally embedded rational matrix $l$, or equivalently $l = I_5$, and hence we have $|\eta_{1, p}|_p = |\eta_{2, p}|_p = |\eta_{3, 4}|_p = |\eta_{4, p}|_p = 1$. We get{\allowdisplaybreaks
\begin{align}
\label{ex1part1}\\
    & W_{\alpha_4}(\lambda; (g_\infty, I_n, I_n, \cdots)) = \\
    ={}& \e{x_{45}} B_{s-1}\left( \frac{y_4^2}{y_3}, 1 \right) \left( y_1^{2s-1} y_4^{4-2s} \frac{\xi\left( 2s-2 \right)}{\xi\left( 2s \right)} \prod_{p < \infty} |\eta_{1, p}|_p^{2s-1}|\eta_{4, p}|_p^{4-2s} \right. \\
    &{}+ \left. y_1^{3-2s} y_2^{2s-2} y_4^{4-2s} \frac{\xi(2s-2)^2}{\xi(2s) \xi(2s-1)} \prod_{p < \infty} |\eta_{1, p}|_p^{3-2s} |\eta_{2, p}|_p^{2s-2} |\eta_{4, p}|_p^{4-2s} \right. + \\
    &{}+ \left. y_2^{4-2s} y_3^{2s-3} y_4^{4-2s} \frac{\xi(2s-3)\xi(2s-2)}{\xi(2s)\xi(2s-1)} \prod_{p < \infty} |\eta_{2, p}|_p^{4-2s} |\eta_{3, p}|_p^{2s-3} |\eta_{4, p}|_p^{4-2s} \right)  \\
    & \prod_{p < \infty} \gamma_{p}\left( \frac{\eta_{4, p}^2}{\eta_{3, p}} \right) \left( 1 - p^{-2(s-1)} \right) \frac{1-p^{-2(s-1)+1} \left|\frac{\eta_{4, p}^2}{\eta_{3, p}} \right|_{p}^{2(s-1)-1}}{1-p^{-2(s-1)+1}} \\
    ={}& \e{x_{45}} \left( y_1^{2s-1} y_4^{4-2s} \frac{1}{\xi\left( 2s \right)} + y_1^{3-2s} y_2^{2s-2} y_4^{4-2s} \frac{\xi(2s-2)}{\xi(2s) \xi(2s-1)}\right. \\
    &{}+ \left. y_2^{4-2s} y_3^{2s-3} y_4^{4-2s} \frac{\xi(2s-3)}{\xi(2s)\xi(2s-1)} \right) 2 \left| \frac{y_4^2}{y_3}\right|_\infty^{s-3/2} K_{s-3/2}\left( 2\pi \left| \frac{y_4^2}{y_3} \right|_\infty \right) \\
    ={}& 2 \e{x_{45}} \left( y_1^{2s-1} y_3^{3/2-s} y_4 \frac{1}{\xi\left( 2s \right)} + y_1^{3-2s} y_2^{2s-2} y_3^{3/2-s} y_4 \frac{\xi(2s-2)}{\xi(2s) \xi(2s-1)}\right. \\
    &{}+ \left. y_2^{4-2s} y_3^{s-3/2} y_4 \frac{\xi(2s-3)}{\xi(2s)\xi(2s-1)} \right) K_{s-3/2}\left( 2\pi \left| \frac{y_4^2}{y_3} \right|_\infty \right)\,. 
\end{align}
The $x$'s and $y$'s are the Iwasawa coordinates for the matrix $g_\infty$ as in~\eqref{eq:realIwa}.}

\begin{table}[t]
    \begin{align*}
        \begin{tabu}{cl|c|c|c|c}
            &w_c & \left\langle w_c^{-1} \lambda + \rho | \alpha_1 \right\rangle & \left\langle w_c^{-1} \lambda + \rho | \alpha_4 \right\rangle & M\left( w_c^{-1}, \lambda \right) & \left( w_c w_0' \right)^{-1} \lambda + \rho \\
            \hline
            &\operatorname{Id} & 0 & 0 & 1 & \dots \\
            &w_{2} & 2\left( s - \frac{1}{2} \right) & 0 & \dots & \dots \\
            *&w_{23} & 2\left( s-\frac{1}{2} \right) & 2(s-1) & \frac{\xi(2s-2)}{\xi(2s)} & v \\
            &w_{2132} & 0 & 2(s-1) & \dots & \dots \\
            &w_{213243} & 0 & 0 & \dots & \dots \\
        \end{tabu}
    \end{align*}
    \caption{\label{tab:4}Data for the reduction formula \eqref{eqn:reduction} to evaluate $W_{\alpha_1, \alpha_4}(a)$ on $\SL_5$ with $\lambda = 2s\Lambda_2 - \rho$. The stars indicate which Weyl words contribute to the reduction formula. We wrote $v = [3-2s, 2s-2, 0, 4-2s]$ here to conserve space.}
\end{table}

The second Whittaker function is found by the reduction formula with the data given in table~\ref{tab:4}. The $p$-adic Iwasawa-decomposition of $l_z$ has
\begin{equation}
    |\eta_{1, p}|_p = |\eta_{2, p}|_p = |\eta_{3, p}|_p = |\eta_{4, p}|_p = |z|_p\,
    .
\end{equation}
{\allowdisplaybreaks We get
\begin{align}
\label{ex1part2}\\
      \sum_{z'} {}& W_{\alpha_1, \alpha_4}\left( \lambda; l_z (g_\infty, I_n, I_n, \cdots) \right) = \\
    = \sum_{z'} {}& \e{x_{12} + x_{45}} y_1^{3-2s} y_2^{2s-2} y_4^{4-2s} \frac{\xi(2s-2)}{\xi(2s)}  \\
    & B_{s-1/2}\left( \frac{y_1^2}{y_2}, 1 \right) B_{s-1}\left( \frac{y_4^2}{y_3}, 1 \right) \prod_{p < \infty} |\eta_{1, p}|_p^{3-2s} |\eta_{2, p}|_p^{2s-2} |\eta_{4, p}|_p^{4-2s}\\
    & \prod_{p < \infty} \gamma_{p}\left( \frac{\eta_{1, p}^2}{\eta_{2, p}} \right) \left( 1 - p^{-2(s-1/2)} \right) \frac{1-p^{-2(s-1/2)+1} \left| \frac{\eta_{1, p}^2}{\eta_{2, p}} \right|_{p}^{2(s-1/2)-1}}{1-p^{-2(s-1/2)+1}} \times \\
    & \times \prod_{p < \infty} \gamma_{p}\left( \frac{\eta_{4, p}^2}{\eta_{3, p}} \right) \left( 1 - p^{-2(s-1)} \right) \frac{1-p^{-2(s-1)+1} \left| \frac{\eta_{4, p}^2}{\eta_{3, p}} \right|_{p}^{2(s-1)-1}}{1-p^{-2(s-1)+1}} \\
    =\sum_{z' \in \BZ} {}& \e{x_{12} + x_{45}} y_1^{3-2s} y_2^{2s-2}y_4^{4-2s} \frac{\xi(2s-1)}{\xi(2s)} \prod_{p < \infty} |z|_p^{5-2s} \\
    & 4\left| \frac{y_1^2}{y_2} \right|_{\infty}^{s-3/2}\left| \frac{y_4^2}{y_3} \right|_{\infty}^{s-2} K_{s-3/2}\left( 2\pi \left| \frac{y_1^2}{y_2} \right|_{\infty} \right) K_{s-2}\left( 2\pi \left| \frac{y_4^2}{y_3} \right|_{\infty} \right) \\
    & \sigma_{-2(s-1/2)+1}(|z|_\infty) \sigma_{-2(s-1)+1}(|z|_\infty) \\
    =\sum_{z' \in \BZ} {}& 4 \e{x_{12} + x_{45}} y_2^{s-1/2}y_3^{2-s} \frac{\xi(2s-1)}{\xi(2s)} |z|_\infty^{2s-5} \\
    & K_{s-3/2}\left( 2\pi \left| \frac{y_1^2}{y_2} \right|_{\infty} \right) K_{s-2}\left( 2\pi \left| \frac{y_4^2}{y_3} \right|_{\infty} \right) \sigma_{2-2s}(|z|_\infty) \sigma_{3-2s}(|z|_\infty)\,.
\end{align}
The $x$'s and $y$'s are the Iwasawa coordinates for the matrix $l_z g_\infty$.}

\begin{table}[t]
    \begin{align*}
        \begin{tabu}{cl|c|c|c|c}
            &w_c & \left\langle w_c^{-1} \lambda + \rho | \alpha_2 \right\rangle & \left\langle w_c^{-1} \lambda + \rho | \alpha_4 \right\rangle & M\left( w_c^{-1}, \lambda \right) & \left( w_c w_0' \right)^{-1} \lambda + \rho \\
            \hline
            &\operatorname{Id} & 2s & 0 & 1 & \dots \\
            &w_{21} & 0 & 0 & \dots & \dots \\
            &w_{23} & 0 & 2(s-1) & \dots & \dots \\
            *&w_{213} & 2(s-1) & 2(s-1) & \frac{\xi(2s-2)^2}{\xi(2s)\xi(2s-1)} & v \\
            &w_{213243} & 0 & 0 & \dots & \dots \\
        \end{tabu}
    \end{align*}
    \caption{\label{tab:5}Data for the reduction formula \eqref{eqn:reduction} to evaluate $W_{\alpha_2, \alpha_4}(a)$ on $\SL_5$ with $\lambda = 2s\Lambda_2 - \rho$. The star indicates the Weyl word that contributes to the reduction formula. We wrote  $v = [0, 4-2s, 2s-3, 4-2s]$  to save space.}

\end{table}
The third and fourth Whittaker functions are found by the reduction formula with the data from table~\ref{tab:5}. The $p$-adic Iwasawa-decomposition of $l_{xy}$ has
\begin{equation}
    |\eta_{1, p}|_p^{-1} = \max\{|y|_p, |x|_p\}
    \qtextq{and}
    |\eta_{2, p}|_p = |\eta_{3, p}|_p = |\eta_{4, p}|_p = 1\,
    .
\end{equation}

{\allowdisplaybreaks We get
\begin{align}
\label{ex1part3}\\
      \sum_{x', y} {}& W_{\alpha_2, \alpha_4}\left( \lambda; l_{xy} (g_\infty, I_n, I_n, \cdots) \right) = \\
      = \sum_{x', y} {}& \e{x_{23} + x_{45}} y_2^{4-2s} y_3^{2s-3} y_4^{4-2s} \frac{\xi(2s-2)^2}{\xi(2s)\xi(2s-1)}  \\
    & B_{s-1}\left( \frac{y_2^2}{y_1 y_3}, 1 \right) B_{s-1}\left( \frac{y_4^2}{y_3}, 1 \right)\prod_{p < \infty} |\eta_{2, p}|_p^{4-2s} |\eta_{3, p}|_p^{2s-3} |\eta_{4, p}|_p^{4-2s} \\
    & \prod_{p < \infty} \gamma_{p}\left( \frac{\eta_{2, p}^2}{\eta_{1, p} \eta_{3, p}} \right) \left( 1 - p^{-2(s-1)} \right) \frac{1-p^{-2(s-1)+1} \left| \frac{\eta_{2, p}^2}{\eta_{1, p} \eta_{3, p}} \right|_{p}^{2(s-1)-1}}{1-p^{-2(s-1)+1}} \times \\
    & \times \prod_{p < \infty} \gamma_{p}\left( \frac{\eta_{4, p}^2}{\eta_{3, p}} \right) \left( 1 - p^{-2(s-1)} \right) \frac{1-p^{-2(s-1)+1} \left| \frac{\eta_{4, p}^2}{\eta_{3, p}} \right|_{p}^{2(s-1)-1}}{1-p^{-2(s-1)+1}} \\
    = \sum_{x', y \in \BZ} {}& \e{x_{23} + x_{45}} y_2^{4-2s} y_3^{2s-3} y_4^{4-2s} \frac{1}{\xi(2s)\xi(2s-1)} 4 \left| \frac{y_2^2}{y_1 y_3} \right|_\infty^{s-3/2} \left| \frac{y_4^2}{y_3} \right|_\infty^{s-3/2} \\
    & K_{s-3/2}\left( 2\pi \left|\frac{y_2^2}{y_1 y_3} \right|_\infty \right) K_{s-3/2}\left( 2\pi \left| \frac{y_4^2}{y_3} \right|_\infty \right) \sigma_{-2(s-1)+1}(k) \\
    = \sum_{x', y \in \BZ} {}& 4 \e{x_{23} + x_{45}} y_1^{3/2-s} y_2^{1} y_4^{1} \frac{1}{\xi(2s)\xi(2s-1)}  \\
    & K_{s-3/2}\left( 2\pi \left|\frac{y_2^2}{y_1 y_3} \right|_\infty \right) K_{s-3/2}\left( 2\pi \left| \frac{y_4^2}{y_3} \right|_\infty \right) \sigma_{3-2s}(k)\,,
\end{align}
where $k = \gcd(|y|, |x|)$. Here, the $x$'s and $y$'s are the Iwasawa coordinates for the matrix $l_{xy} g_\infty$.}

The $p$-adic Iwasawa-decomposition of $l_{x}$ has
\begin{equation}
    |\eta_{1, p}|_p^{-1} = \max\{|0|_p, |x|_p\} = |x|_p
    \qtextq{and}
    |\eta_{2, p}|_p = |\eta_{3, p}|_p = |\eta_{4, p}|_p = 1\,
    .
\end{equation}
We get
\begin{equation}
    \begin{aligned}
      \sum_{x'} {}& W_{\alpha_2, \alpha_4}\left( \Lambda; l_x (g_\infty, I_n, I_n, \cdots) \right) = \\
      = \sum_{x'} {}& \e{x_{23} + x_{45}} y_2^{4-2s} y_3^{2s-3} y_4^{4-2s} \frac{\xi(2s-2)^2}{\xi(2s)\xi(2s-1)}  \\
    & B_{s-1}\left( \frac{y_2^2}{y_1 y_3}, 1 \right) B_{s-1}\left( \frac{y_4^2}{y_3}, 1 \right)\prod_{p < \infty} |\eta_{2, p}|_p^{4-2s} |\eta_{3, p}|_p^{2s-3} |\eta_{4, p}|_p^{4-2s}\\
    & \prod_{p < \infty} \gamma_{p}\left( \frac{\eta_{2, p}^2}{\eta_{1, p} \eta_{3, p}} \right) \left( 1 - p^{-2(s-1)} \right) \frac{1-p^{-2(s-1)+1} \left| \frac{\eta_{2, p}^2}{\eta_{1, p} \eta_{3, p}} \right|_{p}^{2(s-1)-1}}{1-p^{-2(s-1)+1}} \\
    & \prod_{p < \infty} \gamma_{p}\left( \frac{\eta_{4, p}^2}{\eta_{3, p}} \right) \left( 1 - p^{-2(s-1)} \right) \frac{1-p^{-2(s-1)+1} \left| \frac{\eta_{4, p}^2}{\eta_{3, p}} \right|_{p}^{2(s-1)-1}}{1-p^{-2(s-1)+1}} \\
    = \sum_{x' \in \BZ} {}& \e{x_{23} + x_{45}} y_2^{4-2s} y_3^{2s-3} y_4^{4-2s} \frac{1}{\xi(2s)\xi(2s-1)} 4 \left| \frac{y_2^2}{y_1 y_3} \right|_\infty^{s-3/2} \left| \frac{y_4^2}{y_3} \right|_\infty^{s-3/2} \\
    & K_{s-3/2}\left( 2\pi \left|\frac{y_2^2}{y_1 y_3} \right|_\infty \right) K_{s-3/2}\left( 2\pi \left| \frac{y_4^2}{y_3} \right|_\infty \right) \sigma_{-2(s-1)+1}(|x|_\infty) \\
    = \sum_{x' \in \BZ} {}& 4 \e{x_{23} + x_{45}} y_1^{3/2-s} y_2^{1} y_4^{1} \frac{1}{\xi(2s)\xi(2s-1)}  \\
    & K_{s-3/2}\left( 2\pi \left|\frac{y_2^2}{y_1 y_3} \right|_\infty \right) K_{s-3/2}\left( 2\pi \left| \frac{y_4^2}{y_3} \right|_\infty \right) \sigma_{3-2s}(|x|_\infty) \, . \label{ex1part4}
    \end{aligned}
\end{equation}
The $x$'s and $y$'s are the Iwasawa coordinates for the matrix $l_x g_\infty$.

The complete Fourier coefficient $\mathcal{F}^{\BR}(E(2s\Lambda_2 - \rho), \psi_{y(Y_1)}; g_\infty)$ is then given by the combination of~\eqref{ex1part1}, \eqref{ex1part2}, \eqref{ex1part3} and \eqref{ex1part4}. We note that the our final result differs formally from the one given in~\cite[Eq.~(H.52)]{GMV15} where the result is given as a convoluted integral over two Bessel functions whereas we do not have any remaining integral. The two results need not be in actual disagreement as there are many non-trivial relations involving infinite sums or integrals of Bessel functions.

The automorphic form
\begin{align}
\lim_{s\to 1/2} \frac{2\zeta(3)\xi(2s-3)}{\xi(2s)} E(2s\Lambda_2-\rho;g) = 2\zeta(3) E(3\Lambda_4-\rho)
\end{align}
lies in a minimal automorphic representation and controls the first non-trivial corrections that string theory predicts to the four-graviton scattering amplitude beyond standard general relativity~\cite{GMRV10,P10}. The Fourier coefficients that we computed above can then be used to to extract so-called 1/2 BPS instanton contributions in the string perturbation limit of the amplitude. More precisely, they represent non-perturbative corrections to the scattering amplitude that, albeit smooth, are not analytic in the string coupling constant around vanishing coupling. They are therefore not visible in standard perturbation theory for small coupling but represent important correction nonetheless. Their interpretation is in terms of specific D$p$-branes ($p\le 2$) that are extended $(p+1)$-dimensional objects that can extend on non-trivial cycles of the torus $T^3$ that is present when $\SL_5$ is the duality group. The detailed structure of the Fourier coefficient, in particular the arithmetic divisor sums appearing, can shed some light on the combinatorics of these D-branes similar to what is happening in the $\SL_2$ case~\cite{Yi:1997eg,Sethi:1997pa,Moore:1998et}.

For the next non-trivial correction to the four-graviton scattering amplitude one requires an automorphic form in the next-to-minimal automorphic representation~\cite{GMRV10,P10,GMV15}. This function is not a single Eisenstein series of the type we have analysed above but a very special combination of two formally divergent Eisenstein series with some Fourier coefficients computed using the Mellin transform of a theta lift in~\cite{GMV15}.

\appendix
\section{Euler products and twisted characters}
\label{app:euler}

This appendix contains details and explanations for section~\ref{sec:sl5}, which is why we restrict to the field $F = \BQ$ with the corresponding ring of adeles $\BA = \BA_\BQ$.

An Euler product is a product over the primes. The $p$-adic norm is denoted $|\cdot|_p$ and is defined for the $p$-adic numbers $\BQ_p$. The absolute value norm or ``infinity norm'' is denoted $|\cdot|_\infty$ and is defined for real numbers $\BR = \BQ_\infty$. The $p$-adic numbers as well as the real numbers (being completions of the rational numbers) all contain the rational numbers: $\BQ \subset \BQ_p$ for all $p$ prime. The norm of an adele $x = (x_\infty, x_2, x_3, x_5, \dots) \in \BA$ is denoted $|\cdot|$ (without ornaments) and is the product of norms at the local places
\begin{equation}
    |x| = \prod_{p \leq \infty} |x_p|_p\,
    .
\end{equation}
The rational numbers $\BQ$ are diagonally embedded into the adeles $\BA$
\begin{equation}
    \BQ \subset \BA \quad{}\text{in the sense that}\quad{} (q, q, q, q, \dots) \in \BA \quad{}\text{for}\quad{} q \in \BQ\,
    .
\end{equation}

\vspace{1.0em}
\paragraph{\bf Product of norms}
For a rational number $x \in \BQ$ with a decomposition into primes as 
\begin{align}
  x = \pm \prod_p p^{m^{(p)}}\,,
\end{align}
we get a particularly simple result for the adelic norm of $x$, namely
\begin{align}
  \prod_{p \leq \infty} |x|_p = |x|_\infty \prod_p p^{-m^{(p)}} = |x|_\infty |x|_\infty^{-1} = 1\,
  \label{eqn:unitproduct}
  .
\end{align}
This is most often used as
\begin{equation}
    x \in \BQ \quad{}\Rightarrow \quad{} \prod_p |x|_p = |x|_\infty^{-1}\,
    .
\end{equation}

\vspace{1.0em}
\paragraph{\bf Greatest common divisor}
For a set of natural numbers $\{ x_i \}$ where each $x_i$ has a decomposition into primes as
\begin{align}
  x_i = \prod_p p^{m^{(p)}_i}\,,
\end{align}
one can express the greatest common divisor $k$ as
\begin{align}
    k \equiv \gcd(\{x_i \}) = \prod_p p^{\min_i \left\{ m_i^{(p)} \right\} }
  .
\end{align}
Together with
\begin{align}
  \left| x_i \right|_p = p^{-m_i^{(p)}}\,,
\end{align}
we are led to the expression
\begin{align}
    k = \prod_p p^{\min_i \left\{ m_i^{(p)} \right\} } = \prod_p \min_i\left\{ p^{ m_i^{(p)}} \right\} = \prod_p \min_i\left\{ |x_i|_p^{-1} \right\} = \prod_p \left( \max_i \{\left| x_i \right|_p\} \right)^{-1}
  .
\end{align}
We also have the formula
\begin{align}
  \left| k \right|_p = \max_i \left\{ \left| x_i \right|_p \right\}
  \label{eqn:pnormk}
  .
\end{align}

Note that
\begin{equation}
    \gcd(x_1, \cdots, x_n, 0) = \gcd(x_1, \cdots, x_n)\,,
\end{equation}
since every (nonzero) integer divides 0. Additionally, we define
\begin{equation}
    \gcd(x) = x \quad\forall x \in \BZ\,,
\end{equation}
including $x = 0$.

\vspace{1.0em}
\paragraph{\bf Divisor sum}
We have the identity 
\begin{align}
  \prod_p \frac{1-p^{-s}\left| m \right|_p^s}{1 - p^{-s}} = \sum_{d|m}d^{-s} \equiv \sigma_{-s}(m)\,,
\end{align}
for $s \in \BC$ and $m \in \BZ$.

\vspace{1.0em}
\paragraph{\bf The completed Riemann zeta function}
The Riemann zeta function
\begin{align}
  \zeta(s) = \sum_{n=1}^\infty n^{-s} , \quad \Re(s) >1
\end{align}
can be written as an Euler product as
\begin{align}
  \zeta(s) = \prod_p \frac{1}{1 - p^{-s}}, \quad \Re(s)>1
\end{align}
and can be analytically continued to the whole complex plane except at $s = 0$ and $s = 1$ where it has simple poles. This is done by defining the completed Riemann zeta function
\begin{equation}
    \xi(s) \equiv \Gamma\left( \frac{s}{2} \right)\pi^{-s/2}\zeta(s)
\end{equation}
which obeys the functional relation
\begin{equation}
    \xi(s) = \xi(1-s)
\end{equation}
as shown by Riemann.

\vspace{1.0em}
\paragraph{\bf $p$-adic gaussian}
The $p$-adic gaussian $\gamma_p: \BQ_p \rightarrow \{0, 1\}$ is defined as
\begin{equation}
    \gamma_p(x) = 
    \begin{cases}
        1, & |x|_p \leq 1 \\
        0, & |x|_p > 1 \\
    \end{cases}
    =
    \begin{cases}
        1, & x \in \BZ_p \\
        0, & x \nin \BZ_p\,. \\
    \end{cases}
\end{equation}
For a rational number $x$ we then get
\begin{equation}
    \prod_p \gamma_p(x) = 
    \begin{cases}
        1, & x \in \BZ_p \forall p \\
        0, & \text{else}
    \end{cases}
    =
    \begin{cases}
        1, & x \in \BZ \\
        0, & \text{else\,.}
    \end{cases}
\end{equation}

Notice also that for rational numbers $x_1,\ldots, x_n\in \BQ$ and picking an $x \in \BQ$ such that for all primes $p$
\begin{equation}
    |x|_p = \max \{|x_1|_p, \cdots, |x_n|_p \} \,,
    \label{eqn:specialx}
\end{equation}
we have
\begin{equation}
    \gamma_p(x) = \prod_{i = 1}^n \gamma_p(x_i)
    .
\end{equation}
A consequence of this is that for an eulerian function depending only on the $p$-adic norms of its argument
\begin{equation}
    f(x) = \prod_p f_p(|x|_p)\,,
\end{equation}
then with $x$ as in \eqref{eqn:specialx}, we have
\begin{equation}
    f(x) \prod_p \gamma_p(x) = \prod_p f_p(|x|_p) \gamma_p(x) = \prod_p f_p(|k|_p) \gamma_p(x) = f(k) \prod_p \gamma_p(x)\,,
\end{equation}
where
\begin{equation}
    k = \gcd(|x_1|_\infty, \cdots, |x_n|_\infty)\,.
\end{equation}
This equation makes sense as $\prod_p \gamma_p(x)$ ensures that the left- and right hand sides are nonzero only when each $x_i$ is integer for which $k$ is well defined. We now see how a sum over rationals with $x$ as in \eqref{eqn:specialx} can collapse to a sum over integers due to the $p$-adic gaussian
\begin{equation}
    \sum_{x_1, \ldots, x_n} f(x) \prod_p \gamma_p(x) = \sum_{x_1, \ldots, x_n} f(k) \prod_p \gamma_p(x) = \sum_{x_1, \ldots, x_n \in \BZ} f(k)\, .
\end{equation}

\vspace{1.0em}
\paragraph{\bf $\SL(2)$ Whittaker function}
The ramified (meaning $m$ not necessarily unity) $\SL(2, \BA)$ Whittaker function evaluated at
\begin{align}
  g = \left( g_\infty, \I, \I, \dots \right) = 
  \left( 
  \left(
  \begin{array}{cc}
    1 & x \\
    0 & 1
  \end{array}
  \right)
  \left(
  \begin{array}{cc}
    y & 0 \\
    0 & \frac{1}{y}
  \end{array}
  \right)
  k, \I, \I, \dots
  \right)
\end{align}
written as an Euler product reads
\begin{align}
  \begin{aligned}
    &W_{\alpha}\left( 2s\Lambda - \rho, m; g \right) = W_{\alpha}\left( 2s\Lambda - \rho, m;
    \left( 
    \left(
    \begin{smallmatrix}
      1   &   x   \\
      &   1   \\
    \end{smallmatrix}
    \right)
    \left(
    \begin{smallmatrix}
      y   &       \\
      &\frac{1}{y}\\
    \end{smallmatrix}
    \right)
    k, \I, \I, \dots
   \right)
    \right) = \\
    ={}&
    \e{mx}
    B_{s}(m, y)
    \prod_p \gamma_p(m) \left( 1-p^{-2s} \right)\frac{1-p^{-2s+1}|m|_p^{2s-1}}{1 - p^{-2s+1}}\,,
  \end{aligned}
\end{align}
where $\alpha$ is the simple root and $\Lambda$ is the fundamental weight. Here
\begin{align}
\label{eq:SL2Whitt}
  B_s (m, y) \equiv \frac{2\pi^s}{\Gamma(s)} y^{1/2} |m|^{s-1/2} K_{s-1/2}\left( 2\pi |m| y \right)
\end{align}
should be seen as the archimedean $\SL(2)$-Whittaker function and each factor in the Euler product as the non-archimedean Whittaker functions. The product
\begin{align}
  \prod_p \gamma_p(m)
\end{align}
restricts to $m \in \BZ$ as explained above. The expression can then be written as
\begin{align}
  W_{\alpha}\left( 2s\Lambda - \rho, m; g \right) = \e{mx}\frac{2}{\xi(2s)}y^{1/2} |m|^{s-1/2} \sigma_{1-2s}(|m|) K_{s-1/2}\left( 2\pi |m| y \right) \,
  .
\end{align}
Notice how the factors of the Eulerian expression for the Riemann zeta function in the non-archimedean part combines with $\pi^s/\Gamma(s)$ in the archimedean part to form a completed Riemann zeta function.

\vspace{1.0em}
\paragraph{\bf Twisted character}
Let $m \in \BQ$ and $\psi_{p,m}$ be an additive character on $\BQ_p$ defined as
\begin{alignat}{3}
  \psi_{\infty, m}(x) &= e^{2\pi i m x}\,; \quad &  m, x &\in \BR &  &\qtext{for real numbers} \label{eqn:realcharacter} \\
  \psi_{p, m}(x) &= e^{-2\pi i \left[ m x \right]_p}\,; \quad & m, x &\in \BQ_p &  &\qtext{for $p$-adic numbers\,.} \label{eqn:padiccharacter}
\end{alignat}
An unitary multiplicative character on the unipotent radical $N(\BA)$ of the Borel subgroup of $\SL(n, \BA)$ can then be parametrized by $m_1, \dots, m_{n-1} \in \BQ$ as 
\begin{align}
  \psi( n ) = \psi( e^{\sum_{\alpha \in \Delta_+} u_{\alpha} E_{\alpha} } ) = \psi( e^{\sum_{\alpha \in \Pi} u_{\alpha} E_{\alpha} } ) = \prod_{p \leq \infty} \prod_{i=1}^n \psi_{p, m_i} \left( \left( u_{\alpha_i} \right)_p \right)\,, \label{eqn:adeliccharacter}
\end{align}
where $\Delta_+$ is the set of positive roots and $\Pi = \left\{ \alpha_1, \dots, \alpha_{n-1} \right\} \subset \Delta_+$ is the set of simple roots. The second equality is due to the fact that the additive character is only sensitive to the abelianization of $N(\BA)$. In the final equality, $\left( x_{\alpha_i} \right)_p$ denotes the $p$-adic (or real) component of the adelic coordinate $x_{\alpha_i}$. 

For an element $a \in A(\ads)$, we would like to evaluate the twisted character
\begin{align}
  \psi^a(n) \equiv \psi( a n a^{-1}) \, .
\end{align}

Let $x_\alpha(t) = \exp(t E_\alpha)$ where $t \in \ads$ and $E_\alpha$ is the positive Chevalley generator for the root $\alpha$. 
For $t \in \ads^\times$, define $w_\alpha(t) = x_\alpha(t) x_{-\alpha}(-t^{-1}) x_\alpha(t)$ and $h_\alpha(t) = w_\alpha(t) w_\alpha(1)^{-1}$. 
An element $a \in A(\ads)$ is then parametrized by \cite{Steinberg}
\begin{equation}
    a = \prod_{i=1}^{n-1} h_{\alpha_i}(y_i) \qquad y_i \in \ads^\times \, ,
\end{equation}
where the different generators $h_\alpha$ commute for all simple roots $\alpha \in \Pi$ and are multiplicative in $t_\alpha$.
For a simple root $\alpha$ and a root $\beta$, we have that \cite{Steinberg}
\begin{equation}
    h_{\alpha}(y) x_{\beta}(u) h_{\alpha}(y)^{-1} = x_{\beta}(y^{\beta(H_\alpha)} u)\,.
\end{equation}

Since the character $\psi$ on $N$ is sensitive only to the $x_\beta$ with $\beta$ a simple root, it is enough to consider 
\begin{equation}
    h_{\alpha_i}(y) x_{\alpha_j}(u) h_{\alpha_i}(y)^{-1} = x_{\alpha_j}(y^{A_{ij}} u)
\end{equation}
for the simple roots $\alpha_i$ and $\alpha_j$, where $A_{ij}$ is the Cartan matrix.

We then have that
\begin{multline}
        \psi(a n a^{-1}) = \psi\bigl(\exp\Bigl(\sum_{j = 1}^{n-1} \bigl(\prod_{i = 1}^{n-1} y_i^{A_{ij}} \bigr) x_{\alpha_j} E_{\alpha_j}\Bigr) \bigr) = \\
        = \psi\bigl(\exp\Bigl(\frac{y_1^2}{y_2} x_{\alpha_1} E_{\alpha_1} +  \sum_{j=2}^{n-2} \frac{y_j^2}{y_{j-1} y_{j+1}} x_{\alpha_j} E_{\alpha_j} + \frac{y_{n-1}^2}{y_{n-2}} x_{\alpha_{n-1}} E_{\alpha_{n-1}} + \dots \Bigr) \bigr) \, .
\end{multline}

We can interpret the transformation $\psi \rightarrow \psi^a$ as that the parameters $m_i$  transform according to
\begin{align}
  m_i &\rightarrow m_i' = \left( \frac{y_i^2}{y_{i-1}y_{i+1}} \right) m_i, \quad i = 1, \dots, n-1\,,
\end{align}
where we have defined $y_0 = y_n = 1$. Note that starting with rational parameters $m_i$, the transformed parameters $m_i'$ are no longer necessarily rational.

\section{Iwasawa-decomposition}
\label{sec:iwasawa}
Proof of the following results can be found in \cite{A16}. For a real matrix $g \in \SL_n(\BR)$ written in Iwasawa form
\begin{equation}
    g = n_\infty a_\infty k_\infty = 
    \left( 
    \begin{smallmatrix}
        1 & x_{12} & \cdots & \cdots & x_{1n} \\
          & 1 & \ddots & \ddots & \vdots \\
          & & \ddots & \ddots & \vdots \\
          & & & 1 & x_{n-1, n} \\
          & & & & 1
    \end{smallmatrix}
    \right)
    \left( 
    \begin{smallmatrix}
        y_1 & & & & \\
          & y_2/y_1 & & & \\
          & & \ddots & & \\
          & & & y_{n-1}/y_{n-2} & \\
          & & & & 1/y_{n-1}
    \end{smallmatrix}
    \right)
    k_{\infty}\,,
\end{equation}
we have
\begin{align}
x_{\mu \nu} ={} & y_{\nu-1}^2 \epsilon\left( V_\mu, V_{\nu+1}, \dots, V_n; V_{\nu}, V_{\nu+1}, \dots, V_n \right), \quad \mu < \nu, \label{eqn:axions} \qtext{and} \\
y_{\mu}^{-2} ={} & \epsilon\left( V_{\mu+1}, \dots, V_n; V_{\mu+1}, \dots, V_n \right)\,, \label{eqn:realdilatons}
\end{align}
where $V_\mu$ is the $\mu$\textsuperscript{th} row of $g$ (regarded as an $n$-vector). Furthermore, $\epsilon$ denotes the totally antisymmetric product
\begin{equation}
\begin{aligned}
  \epsilon\left( A_1, \dots, A_m; B_1, \dots, B_m \right) = \delta_{a_{1} \dash a_m}^{i_{1} \dash i_m} \left( A_{1} \right)^{a_{1}} \dots{} \left( A_m \right)^{a_m} \left( B_{1} \right)_{i_{1}} \dots{} \left( B_m \right)_{i_m}\,,
\end{aligned}
\end{equation}
where the $A$'s and $B$'s are $n$-vectors and 
\begin{align}
\delta_{a_{1} \dash a_m}^{i_{1} \dash i_m} = m! \delta_{\left[ a_{1} \right.}^{i_{1}} \dots \delta_{\left. a_{m} \right]}^{i_{m}} = \frac{1}{(n-m)!} \epsilon_{a_1 \dash a_m \alpha_{m+1} \dash \alpha_n} \epsilon^{i_1 \dash i_m \alpha_{m+1} \dash \alpha_n}
\end{align}
denotes the generalized Kronecker delta. Put in words, $\epsilon$ takes two sets of vectors and returns the sum of every possible product of scalar products between the two sets weighted by the signs of the given permutations. For example
\begin{equation}
    \epsilon(A_1, A_2; B_1, B_2) = (A_1\cdot B_1)(A_2\cdot B_2) - (A_1\cdot B_2)(A_2\cdot B_1)\,
    .
\end{equation}

For a $p$-adic matrix $g \in \SL_n(\BQ_p)$, the Iwasawa-decomposition
\begin{equation}
    g = n_p a_p k_p = 
    n_p
    \left( 
    \begin{smallmatrix}
        \eta_{1, p} & & & & \\
        & \eta_{2, p}/\eta_{1, p} & & & \\
          & & \ddots & & \\
          & & & \eta_{n-1, p}/\eta_{n-2, p} & \\
          & & & & 1/\eta_{n-1, p}
    \end{smallmatrix}
    \right)
    k_p
\end{equation}
is no longer unique. The $p$-adic norms of the $\eta$'s however are constant across the family of decompositions and are given by
\begin{equation}
\left| \eta_{n-k} \right|_p =
\left(
\max_{\sigma \in \Theta_{k}^{n}}
\left\{ \left|
  g\left(
  \begin{smallmatrix}
    n-k+1 & \dots & n\\
    \sigma(1) & \dots & \sigma(k) \\
  \end{smallmatrix}
  \right)
\right|_p \right\}
\right)^{-1}\,,
\label{eqn:min}
\qtextq{where}
k \in \{1, \dots, n-1\}\,,
\end{equation}
and $\Theta_k^{n}$ detones the set of all ordered subsets of $\{1, \dots, n\}$ of order $k$. Here, 
$
g\left(
\begin{smallmatrix}
  r_{1} & \dots & r_{k} \\
  c_{1} & \dots & c_{k} \\
\end{smallmatrix}
\right)
\label{eqn:minor}
$
denotes a minor of order $k$, given as the determinant of the submatrix of $g$ obtained by only picking the $k$ rows $\{r_{i}\}$ and $k$ columns $\{c_{i}\}$. For example, the matrix
\begin{equation}
    \begin{psmallmatrix}
        \frac{1}{1} & \frac{1}{2} & \frac{1}{3} & \frac{1}{4} \\
        \frac{1}{5} & \frac{1}{6} & \frac{1}{7} & \frac{1}{8} \\
        \frac{1}{9} & \frac{1}{10} & \frac{1}{11} & \frac{1}{12} \\
        \frac{1}{13} & \frac{1}{14} & \frac{1}{15} & \frac{1}{16} \\
    \end{psmallmatrix}\,
    .
\end{equation}
has
\begin{equation}
    \begin{aligned}
        |\eta_2|_3 ={}&
        \max\left\{
            \left| \left|
            \begin{smallmatrix}
                \frac{1}{9}  & \frac{1}{10} \\
                \frac{1}{13} & \frac{1}{14} \\
            \end{smallmatrix}
            \right| \right|_3, 
            \left| \left|
            \begin{smallmatrix}
                \frac{1}{9}  & \frac{1}{11} \\
                \frac{1}{13} & \frac{1}{15} \\
            \end{smallmatrix}
            \right| \right|_3, 
            \left| \left|
            \begin{smallmatrix}
                \frac{1}{9}  & \frac{1}{12} \\
                \frac{1}{13} & \frac{1}{16} \\
            \end{smallmatrix}
            \right| \right|_3, 
            \left| \left|
            \begin{smallmatrix}
                \frac{1}{10} & \frac{1}{11} \\
                \frac{1}{14} & \frac{1}{15} \\
            \end{smallmatrix}
            \right| \right|_3, 
            \left| \left|
            \begin{smallmatrix}
                \frac{1}{10} & \frac{1}{12} \\
                \frac{1}{14} & \frac{1}{16} \\
            \end{smallmatrix}
            \right| \right|_3, 
            \left| \left|
            \begin{smallmatrix}
                \frac{1}{11} & \frac{1}{12} \\
                \frac{1}{15} & \frac{1}{16} \\
            \end{smallmatrix}
            \right| \right|_3
        \right\} = \\
        ={}&
        \max \left\{ 
            \left| \frac{1}{4095} \right|_3, 
            \left| \frac{8}{19305} \right|_3, 
            \left| \frac{1}{1872} \right|_3, 
            \left| \frac{1}{5775} \right|_3, 
            \left| \frac{1}{3360} \right|_3, 
            \left| \frac{1}{7920} \right|_3, 
        \right\} \\
        ={}&
        \max \left\{ 3^2, 3^3, 3^2, 3^1, 3^1, 3^2 \right\} = 3^3 = 9\,
        .
    \end{aligned}
\end{equation}

\section{Parametrizing \texorpdfstring{$\Gamma_i$}{Gamma\_i} and \texorpdfstring{$\Lambda_j$}{Lambda\_j}}
\label{sec:cosets}
Recall the definitions
\begin{equation}
    \Gamma_i(\psi_0) \defeq
    \begin{cases}
        (\SL_{n-i}(F))_{\hat Y} \bs \SL_{n-i}(F) & 1 \leq i \leq n-2 \\
        (T_{\psi_0} \cap T_{\psi_{\alpha_{n-1}}}) \bs T_{\psi_0} & i = n-1\,,
    \end{cases}
\end{equation}
where
\begin{equation}
    (\SL_{n-i}(F))_{\hat Y} = \left\{ 
        \begin{psmallmatrix}
            1 & \xi^{\mathrm{T}} \\
            0 & h
        \end{psmallmatrix}
        : h \in \SL_{n-i-1}(F), \xi \in F^{n-i-1}
    \right\}\,;
\end{equation}
and
\begin{equation}
    \Lambda_j(\psi_0) \defeq
    \begin{cases}
        (\SL_j(F))_{\hat X} \bs \SL_j(F) & 2 < j \leq n \\
        (T_{\psi_0} \cap T_{\psi_{\alpha_1}}) \bs T_{\psi_0} & j = 2\,,
    \end{cases}
\end{equation}
where
\begin{equation}
    (\SL_{j}(F))_{\hat X} = \left\{ 
        \begin{psmallmatrix}
            h & \xi \\
            0 & 1
        \end{psmallmatrix}
        : h \in \SL_{j-1}(F), \xi \in F^{j-1}
    \right\}\,
    .
\end{equation}
In this appendix, we will find convenient representatives for these coset spaces. We begin with a lemma.

\begin{lem}
    \label{lem:GLcoset}
    Let $S_k(F)$ denote the set of all $k\times k$ matrices $m$ over the field $F$ satisfying $\dim \ker m = 1$. The coset space $\GL_{k}(F) \bs S_k(F)$ can then be parametrized as
    \begin{equation}
        \GL_{k}(F) \bs S_k =
        \bigcup_{a = 0}^{k-1}
        \left\{ 
            \left( 
            \begin{smallmatrix}
                0 & 0 & 0 \\
                I_a & 0 & 0 \\
                0 & v & I_{k-a-1}
            \end{smallmatrix}
            \right)
            : v \in F^{k-a-1}
        \right\}\,.
    \end{equation}
\end{lem}
\begin{proof}
    We will use induction to prove the lemma. Assume that the result holds up to and including matrices of size $k\times k$ and consider the coset space $\GL_{k+1}(F) \bs S_{k+1}(F)$. Start with a matrix $m_{k+1} \in S_{k+1}(F)$. Left action of the group $\GL_{k+1}(F) \ni h_{k+1}$ taking $m_{k+1} \rightarrow h_{k+1}m_{k+1}$ is equivalent to performing Gauss elimination among the rows of $m_{k+1}$. Since we have $\dim \ker m_{k+1} = 1$ we can bring $m_{k+1}$ to the form
    \begin{equation}
        m_{k+1} \rightarrow
        \begin{psmallmatrix}
            0 & 0 \\
            v & m
        \end{psmallmatrix}\,,
    \end{equation}
    where $v \in F^k$ and $m$ satisfies $\dim \ker m_k \leq 1$. We cannot have $\dim \ker m \geq 2$ as we could then perform additional row manipulations to produce two zero rows in $m$ and hence another zero row in $m_{k+1}$ which violates $m_{k+1} \in S_{k+1}(F)$.
    \subsubsection{Case 1: $\dim \ker m = 0$} 
    Here $m$ is invertible and we can bring $m_{k+1}$ to the form
    \begin{equation}
        \begin{psmallmatrix}
            0 & 0 \\
            v & m_k
        \end{psmallmatrix}
        \rightarrow
        \begin{psmallmatrix}
            0 & 0 \\
            v & I_k
        \end{psmallmatrix}
    \end{equation}
    having relabelled $v$. We get the contribution
    \begin{equation}
        \left.
        \left\{
            \begin{psmallmatrix}
                0 & 0 & 0\\
                I_a & 0 & 0 \\
                0 & v & I_{k+1-a-1}
            \end{psmallmatrix}
            : v \in F^{k+1-a-1}
        \right\}
        \right|_{a=0}\,
        .
    \end{equation}
    \subsubsection{Case 2: $\dim \ker m = 1$} 
    We now have $m = m_k$ for some $m_k \in S_k$ and we can apply the induction assumption which leads us to consider matrices of the form
    \begin{equation}
        \begin{psmallmatrix}
            0 & 0 & 0 & 0 \\
            v^{(1)} & 0 & 0 & 0 \\
            v^{(2)} & I_a & 0 & 0 \\
            v^{(3)} & 0 & u & I_{k-a-1}
        \end{psmallmatrix}\,,
        \quad\text{where}\quad a \in [0, k-2] \cap \BZ\,
        .
    \end{equation}
    We see that we must have $v^{(1)} \neq 0$ and with further row manipulations we can thus bring this to the form
    \begin{equation}
        \begin{psmallmatrix}
            0 & 0 & 0 & 0 \\
            v^{(1)} & 0 & 0 & 0 \\
            v^{(2)} & I_a & 0 & 0 \\
            v^{(3)} & 0 & u & I_{k-a-1}
        \end{psmallmatrix}
        \rightarrow
        \begin{psmallmatrix}
            0 & 0 & 0 & 0 \\
            1 & 0 & 0 & 0 \\
            0 & I_a & 0 & 0 \\
            0 & 0 & u & I_{k-a-1}
        \end{psmallmatrix}\,
        .
    \end{equation}
    We get the contributions
    \begin{equation}
        \bigcup_{a = 0}^{k-2}
        \left\{ 
            \begin{psmallmatrix}
                0 & 0 & 0 & 0 \\
                1 & 0 & 0 & 0 \\
                0 & I_a & 0 & 0 \\
                0 & 0 & v & I_{k-a-1}
            \end{psmallmatrix}
            : v \in F^{k-a-1}
        \right\}
        =
        \bigcup_{a = 1}^{k-1}
        \left\{ 
            \left( 
            \begin{smallmatrix}
                0 & 0 & 0 \\
                I_a & 0 & 0 \\
                0 & v & I_{k+1-a-1}
            \end{smallmatrix}
            \right)
            : v \in F^{k+1-a-1}
        \right\}\,
        .
    \end{equation}
    This combines with the contribution from case 1 to give the form stated in the lemma.
    
    That the base case $k = 1$ has the correct form is trivial. Peano's axiom of induction now establishes the lemma.
\end{proof}

\begin{lem}
    \label{lem:gammacoset}
    The coset space
    \begin{equation}
        (\SL_{n-i}(F))_{\hat Y} \bs \SL_{n-i}(F) \quad 1 \leq i \leq n-2
    \end{equation}
    can be parametrized as
\begin{equation}
    \begin{aligned}
        & (\SL_{n-i}(F))_{\hat Y} \bs \SL_{n-i}(F) = \\
        ={}& \left\{
            \begin{psmallmatrix}
                x'^{-1} & 0 & 0 \\
                y & x' & 0 \\
                v & 0 & I_{n-i-2}
            \end{psmallmatrix}
            : x' \in F^{\times}, y \in F, v \in F^{n-i-2}
        \right\} \\
        \cup{}&
        \bigcup_{a = 0}^{n-i-2}
        \left\{
            \begin{psmallmatrix}
                0 & 0 & (-1)^{a+1} x'^{-1} & 0 \\
                x' & 0 & 0 & 0 \\
                0 & I_a & 0 & 0 \\
                0 & 0 & v & I_{n-i-a-2}
            \end{psmallmatrix}
            : x' \in F^{\times}, v \in F^{n-i-a-2}
        \right\} \,.\\
    \end{aligned}
\end{equation}
\end{lem}
\begin{proof}
    Denote $k \equiv n-i$. Consider a matrix
    \begin{equation}
        G =  
        \begin{psmallmatrix}
            s & T^{\mathrm{T}} \\
            B & m
        \end{psmallmatrix}
        \in \SL_k(F)\,,
    \end{equation}
    where $s$ is a scalar, $m$ is a $(k-1)\times (k-1)$-matrix and $T$ and $B$ (for ``top'' and ``bottom'') are $(k-1)$-column vectors.
    The action of an element
    \begin{equation}
        M =  
        \begin{psmallmatrix}
            1 & \xi^{\mathrm{T}} \\
            0 & h
        \end{psmallmatrix}
        \in (\SL_{k}(F))_{\hat Y}
    \end{equation}
    on $G$ is
    \begin{equation}
        G \rightarrow MG = 
        \begin{psmallmatrix}
            s + \xi^{\mathrm{T}} B & T^{\mathrm{T}} + \xi^{\mathrm{T}} m \\
            h B & h m
        \end{psmallmatrix}\,
        .
    \end{equation}
    Parametrizing the coset space $(\SL_{k}(F))_{\hat Y} \bs \SL_{k}(F)$ amounts to choosing $\xi \in F^{k-1}$ and $h \in \SL_{k-1}(F)$ such that the product $MG$ takes a particularly nice form, manifestly with at most $k$ degrees of freedom which is the dimension of this coset space.

    Even though $h \in \SL_{k-1}(F)$ we will proceed with $h \in \GL_{k-1}(F)$ and restore the unit determinant of $h$ at the end by left multiplication of the matrix
    $
        \begin{psmallmatrix}
            1 & 0 & 0 \\
            0 & x' & 0 \\
            0 & 0 & I_{k-2}
        \end{psmallmatrix}
    $
    where $0 \neq x' = (\det h)^{-1}$. By having $h \in \GL_{k-1}(F)$ we are free to perform Gauss elimination among the bottom $k-1$ rows in $G$.

    We consider the two cases $\dim \ker m = 0$ and $\dim \ker m = 1$. Note that the cases $\dim \ker m \geq 2$ do not arise as with row elimination it would then be possible to produce two zero-rows in $m$ and hence a zero-row in $G$ which violates $G \in \SL_k(F)$.
    
    \subsubsection{Case 1: $\dim \ker m = 0$}
    We choose $h = m^{-1}$ and $\xi^{\mathrm{T}} = -T^{\mathrm{T}} m^{-1}$. Since $h$ has full rank, we can redefine $hB \rightarrow B$ without loss of generality and redefine $s + \xi^{\mathrm{T}} B \rightarrow s$. This leads to the representative
    \begin{equation}
        G \rightarrow
        \begin{psmallmatrix}
            s & 0 \\
            B & I_{k-1}
        \end{psmallmatrix}\,
        .
    \end{equation}
    We now restore the unit determinant to $h$
    \begin{equation}
        G \rightarrow 
        \begin{psmallmatrix}
            1 & 0 & 0 \\
            0 & x' & 0 \\
            0 & 0 & I_{k-2}
        \end{psmallmatrix}
        \begin{psmallmatrix}
            s & 0 \\
            B & I_{k-1}
        \end{psmallmatrix}
        =
        \begin{psmallmatrix}
            s & 0 & 0 \\
            y & x' & 0 \\
            v & 0 & I_{k-2}
        \end{psmallmatrix}\,,
    \end{equation}
    where we have split the $(k-1)$-vector into a scalar $y$ and a $(k-2)$-vector $v$. The condition $\det G = 1$ now sets $s = x'^{-1}$ leading to
    \begin{equation}
        G 
        =
        \begin{psmallmatrix}
            x'^{-1} & 0 & 0 \\
            y & x' & 0 \\
            v & 0 & I_{k-2}
        \end{psmallmatrix}\,
        .
    \end{equation}
    This is a nice form of the representative $G$ which manifestly has $k$ degrees of freedom.

    \subsubsection{Case 2: $\dim \ker m = 1$}
    We can no longer choose $h = m^{-1}$. Having promoted $h$ to be an element of $\GL_{k-1}(F)$, we can make use of lemma \ref{lem:GLcoset} which leads us to consider representatives of the form
    \begin{equation}
        G \rightarrow
        \begin{psmallmatrix}
            s & T^{(1)\mathrm{T}} & T^{(2)} & T^{(3)\mathrm{T}} \\
            B^{(1)} & 0 & 0 & 0 \\
            B^{(2)} & I_a & 0 & 0 \\
            B^{(3)} & 0 & v & I_{k-a-2}
        \end{psmallmatrix}
        \quad\text{for}\quad a \in [0, k-2] \cap \BN
        \quad\text{and}\quad v \in F^{k-a-2}\,
        .
    \end{equation}
    We see that $B^{(1)} \neq 0$ in order for $G$ to remain non-singular. With further row elimination we can therefore bring this to the form
    \begin{equation}
        \begin{psmallmatrix}
            s & T^{(1) \mathrm{T}} & T^{(2)} & T^{(3) \mathrm{T}} \\
            B^{(1)} & 0 & 0 & 0 \\
            B^{(2)} & I_a & 0 & 0 \\
            B^{(3)} & 0 & v & I_{k-a-2}
        \end{psmallmatrix}
        \rightarrow
        \begin{psmallmatrix}
            s & T^{(1) \mathrm{T}} & T^{(2)} & T^{(3) \mathrm{T}} \\
            1 & 0 & 0 & 0 \\
            0 & I_a & 0 & 0 \\
            0 & 0 & v & I_{k-a-2}
        \end{psmallmatrix}\,
        .
    \end{equation}
    Next, using the $\xi$-freedom we can bring this to the form
    \begin{equation}
        \begin{aligned}
            &
            \begin{psmallmatrix}
                s & T^{(1) \mathrm{T}} & T^{(2)} & T^{(3) \mathrm{T}} \\
                1 & 0 & 0 & 0 \\
                0 & I_a & 0 & 0 \\
                0 & 0 & v & I_{k-a-2}
            \end{psmallmatrix}
            \rightarrow
            \begin{psmallmatrix}
                1 & \xi^{(1)} & \xi^{(2) \mathrm{T}} & \xi^{(3) \mathrm{T}} \\
                0 & 1 & 0 & 0 \\
                0 & 0 & I_a & 0 \\
                0 & 0 & 0 & I_{k-a-2}
            \end{psmallmatrix}
            \begin{psmallmatrix}
                s & T^{(1) \mathrm{T}} & T^{(2)} & T^{(3) \mathrm{T}} \\
                1 & 0 & 0 & 0 \\
                0 & I_a & 0 & 0 \\
                0 & 0 & v & I_{k-a-2}
            \end{psmallmatrix} \\
            ={}&
            \begin{psmallmatrix}
                s + \xi^{(1)} & T^{(1) \mathrm{T}} + \xi^{(2) \mathrm{T}} & T^{(2)} + \xi^{(3) \mathrm{T}} v & T^{(3) \mathrm{T}} + \xi^{(3) \mathrm{T}} \\
                1 & 0 & 0 & 0 \\
                0 & I_a & 0 & 0 \\
                0 & 0 & v & I_{k-a-2}
            \end{psmallmatrix}
            \rightarrow
            \begin{psmallmatrix}
                0 & 0 & T^{(2)} & 0 \\
                1 & 0 & 0 & 0 \\
                0 & I_a & 0 & 0 \\
                0 & 0 & v & I_{k-a-2}
            \end{psmallmatrix}\,,
        \end{aligned}
    \end{equation}
    with a suitable choice of $\xi$ and having redefined $T^{(2)}$. We now restore the unit determinant to $h$
    \begin{equation}
        \begin{psmallmatrix}
            0 & 0 & T^{(2)} & 0 \\
            1 & 0 & 0 & 0 \\
            0 & I_a & 0 & 0 \\
            0 & 0 & v & I_{k-a-2}
        \end{psmallmatrix}
        \rightarrow 
        \begin{psmallmatrix}
            1 & 0 & 0 \\
            0 & x' & 0 \\
            0 & 0 & I_{k-2}
        \end{psmallmatrix}
        \begin{psmallmatrix}
            0 & 0 & T^{(2)} & 0 \\
            1 & 0 & 0 & 0 \\
            0 & I_a & 0 & 0 \\
            0 & 0 & v & I_{k-a-2}
        \end{psmallmatrix}
        \rightarrow 
        =
        \begin{psmallmatrix}
            0 & 0 & T^{(2)} & 0 \\
            x' & 0 & 0 & 0 \\
            0 & I_a & 0 & 0 \\
            0 & 0 & v & I_{k-a-2}
        \end{psmallmatrix}\,
        .
    \end{equation}
    The condition $\det G = 1$ now sets $T^{(2)} = (-1)^{a+1} x'^{-1}$, leading to the representative
    \begin{equation}
        \begin{psmallmatrix}
            0 & 0 & (-1)^{a+1} x'^{-1} & 0 \\
            x' & 0 & 0 & 0 \\
            0 & I_a & 0 & 0 \\
            0 & 0 & v & I_{k-a-2}
        \end{psmallmatrix}\,
        .
    \end{equation}
\end{proof}

\begin{lem}
    \label{lem:lambdacoset}
    The coset space
    \begin{equation}
        (\SL_{j}(F))_{\hat X} \bs \SL_{j}(F) \quad 2 < j \leq n
    \end{equation}
    can be parametrized as
\begin{equation}
    \begin{aligned}
        & (\SL_{j}(F))_{\hat X} \bs \SL_{j}(F) = \\
        ={}& \left\{
            \begin{psmallmatrix}
                I_{j-2} & 0 & 0 \\
                0 & x' & 0 \\
                v^{\mathrm{T}} & y & x'^{-1}
            \end{psmallmatrix}
            : x' \in F^{\times}, y \in F, v \in F^{j-2}
        \right\} \\
        \cup{}&
        \bigcup_{a = 0}^{j-2}
        \left\{
            \begin{psmallmatrix}
                I_a & 0 & 0 & 0 \\
                0 & 0 & I_{j-a-2} & 0 \\
                0 & 0 & 0 & x' \\
                0 & (-1)^{j+a+1} x'^{-1} & v^{\mathrm{T}} & 0
            \end{psmallmatrix}
            : x' \in F^{\times}, v \in F^{j-a-2}
        \right\} \,.\\
    \end{aligned}
\end{equation}
\end{lem}
\begin{proof}
    Denote $k \equiv j$. Consider a matrix
    \begin{equation}
        G =  
        \begin{psmallmatrix}
            m & T \\
            B^{\mathrm{T}} & s
        \end{psmallmatrix}
        \in \SL_k(F)\,,
    \end{equation}
    where $s$ is a scalar, $m$ is a $(k-1)\times (k-1)$-matrix and $T$ and $B$ (for ``top'' and ``bottom'') are $(k-1)$-column vectors.
    The action of an element
    \begin{equation}
        M =  
        \begin{psmallmatrix}
            h & h \xi \\
            0 & 1
        \end{psmallmatrix}
        \in (\SL_{k}(F))_{\hat X}
    \end{equation}
    on $G$ is
    \begin{equation}
        G \rightarrow MG = 
        \begin{psmallmatrix}
            h(m + \xi B^{\mathrm{T}}) & h(T + s \xi) \\
            B^{\mathrm{T}} & s
        \end{psmallmatrix}\,
        .
    \end{equation}
    Parametrizing the coset space $(\SL_{k}(F))_{\hat X} \bs \SL_{k}(F)$ amounts to choosing $\xi \in F^{k-1}$ and $h \in \SL_{k-1}(F)$ such that the product $MG$ takes a particularly nice form, manifestly with at most $k$ degrees of freedom which is the dimension of this coset space.

    Even though $h \in \SL_{k-1}(F)$ we will proceed with $h \in \GL_{k-1}(F)$ and restore the unit determinant of $h$ at the end by left multiplication of the matrix
    $
        \begin{psmallmatrix}
            I_{k-2} & 0 & 0 \\
            0 & x' & 0 \\
            0 & 0 & 1
        \end{psmallmatrix}\,,
    $
    where $0 \neq x' = (\det h)^{-1}$. By having $h \in \GL_{k-1}(F)$ we are free to perform Gauss elimination among the top $k-1$ rows in $G$.

    We consider the two cases $s \neq 0$ and $s = 0$.

    \subsubsection{Case 1: $s = s' \neq 0$}
    We choose $\xi = \frac{-1}{s'} T$. This leads to the representative
    \begin{equation}
        G \rightarrow
        \begin{psmallmatrix}
            m - \frac{1}{s'}T B^{\mathrm{T}} & 0 \\
            B^{\mathrm{T}} & s'
        \end{psmallmatrix}\,
        .
    \end{equation}
    From the condition
    \begin{equation}
        1 = \det G = \det\left( m - \frac{1}{s'}T B^{\mathrm{T}} \right) s'
    \end{equation}
            we get that
        $$\det\left( m - \frac{1}{s'}T B^{\mathrm{T}} \right) \neq 0\,,$$
    and hence the matrix $m - \frac{1}{s'}T B^{\mathrm{T}}$ can be inverted using our $h$-freedom which leads to the representative
    \begin{equation}
        \begin{psmallmatrix}
            m - \frac{1}{s'}T B^{\mathrm{T}} & 0 \\
            B^{\mathrm{T}} & s'
        \end{psmallmatrix}
        \rightarrow
        \begin{psmallmatrix}
            I_{k-1} & 0 \\
            B^{\mathrm{T}} & s'
        \end{psmallmatrix}\,
        .
    \end{equation}
    We now restore the unit determinant to $h$
    \begin{equation}
        G \rightarrow 
        \begin{psmallmatrix}
            I_{k-2} & 0 & 0 \\
            0 & x' & 0 \\
            0 & 0 & 1
        \end{psmallmatrix}
        \begin{psmallmatrix}
            I_{k-1} & 0 \\
            B^{\mathrm{T}} & s'
        \end{psmallmatrix}
        =
        \begin{psmallmatrix}
            I_{k-2} & 0 & 0 \\
            0 & x' & 0 \\
            v^{\mathrm{T}} & y & s'
        \end{psmallmatrix}\,,
    \end{equation}
    where we have split $B$ into a scalar $y$ and a $(k-2)$-vector $v$. The condition $\det G = 1$ now sets $s' = x'^{-1}$ leading to
    \begin{equation}
        G 
        =
        \begin{psmallmatrix}
            I_{k-2} & 0 & 0 \\
            0 & x' & 0 \\
            v^{\mathrm{T}} & y & x'^{-1}
        \end{psmallmatrix}\,
        .
    \end{equation}
    This is a nice form of the representative $G$ which manifestly has $k$ degrees of freedom.

    \subsubsection{Case 2: $s = 0$}
    We can no longer eliminate $T$ with our $\xi$-freedom. The group element $G$ takes the form
    \begin{equation}
        G 
        =
        \begin{psmallmatrix}
            m & T' \\
            B'^{\mathrm{T}} & 0 \\
        \end{psmallmatrix}\,,
    \end{equation}
    where the vectors $T'$ and $B'$ must be non-zero (as indicated by the primes) in order for $G$ to be non-singular. A $\xi$-transformation takes the form
    \begin{equation}
        \begin{psmallmatrix}
            m & T' \\
            B'^{\mathrm{T}} & 0 \\
        \end{psmallmatrix}
        \rightarrow
        \begin{psmallmatrix}
            m + \xi B'^{\mathrm{T}} & T' \\
            B'^{\mathrm{T}} & 0 \\
        \end{psmallmatrix}\,
        .
    \end{equation}
    We now consider the $k-1$ distinct cases labelled by $a \in [0, k-2] \cap \BN$ defined by that $B'^{\mathrm{T}}$ takes the form $
    B'^{\mathrm{T}} = 
    \begin{psmallmatrix}
        0_{1\times a} & b' & v
    \end{psmallmatrix},
    $ where $v$ is a $k-a-2$-vector and $0 \neq b' \in F$. The $\xi$-transformation then lets us eliminate the $(a+1)$\textsuperscript{th} column of $m$. This works since the $(a+1)$\textsuperscript{th} column of the matrix $\xi B'^{\mathrm{T}}$ is $b' \xi$ where $b' \neq 0$ by assumption. We are led to the representative
    \begin{equation}
        \begin{psmallmatrix}
            m & T' \\
            B'^{\mathrm{T}} & 0 \\
        \end{psmallmatrix}
        \rightarrow
        \begin{psmallmatrix}
            m_1 & 0 & m_2 & T' \\
            0 & b' & v^{\mathrm{T}} & 0 \\
        \end{psmallmatrix}\,,
    \end{equation}
    where $m_1$ is a $(k-1)\times a$-matrix and $m_2$ is a $(k-1)\times (k-a-2)$-matrix.

    The $(k-1)\times (k-1)$ matrix $
    \begin{psmallmatrix}
        m_1 & 0 & m_2
    \end{psmallmatrix}
    $ clearly has column-rank at most $k-2$. Since column-rank and row-rank for matrices are equal, we know that the row-rank is also at most $k-2$ and with row manipulations we can thus produce a zero row
    \begin{equation}
        \begin{psmallmatrix}
            m_1 & 0 & m_2 & T' \\
            0 & b' & v^{\mathrm{T}} & 0 \\
        \end{psmallmatrix}
        \rightarrow
        \begin{psmallmatrix}
            0 & 0 & 0 & t' \\
            m_{21} & 0 & m_{22} & T^{-} \\
            0 & b' & v^{\mathrm{T}} & 0 \\
        \end{psmallmatrix}\,,
    \end{equation}
    where $0 \neq t' \in F$ in order for $G$ to be non-singular. With further row manipulations we can then eliminate the vector $T^{-}$ and bring this to the form
    \begin{equation}
        \begin{psmallmatrix}
            0 & 0 & 0 & t' \\
            m_{21} & 0 & m_{22} & T^{-} \\
            0 & b' & v^{\mathrm{T}} & 0 \\
        \end{psmallmatrix}
        \rightarrow
        \begin{psmallmatrix}
            0 & 0 & 0 & 1 \\
            m_{21} & 0 & m_{22} & 0 \\
            0 & b' & v^{\mathrm{T}} & 0 \\
        \end{psmallmatrix}\,
        .
    \end{equation}
    The $(k-2)\times (k-2)$-matrix $
    \begin{psmallmatrix}
        m_{21} & m_{22}
    \end{psmallmatrix}
    $
    must have full rank in order for $G$ to be non-singular and can thus be inverted, leading to the representative
    \begin{equation}
        \begin{psmallmatrix}
            0 & 0 & 0 & 1 \\
            m_{21} & 0 & m_{22} & 0 \\
            0 & b' & v^{\mathrm{T}} & 0 \\
        \end{psmallmatrix}
        \rightarrow
        \begin{psmallmatrix}
            0 & 0 & 0 & 1 \\
            I_a & 0 & 0 & 0 \\
            0 & 0 & I_{k-a-2} & 0 \\
            0 & b' & v^{\mathrm{T}} & 0 \\
        \end{psmallmatrix}\,
        .
    \end{equation}
    We permute the first $k-1$ rows to get
    \begin{equation}
        \begin{psmallmatrix}
            0 & 0 & 0 & 1 \\
            I_a & 0 & 0 & 0 \\
            0 & 0 & I_{k-a-2} & 0 \\
            0 & b' & v^{\mathrm{T}} & 0 \\
        \end{psmallmatrix}
        \rightarrow
        \begin{psmallmatrix}
            I_a & 0 & 0 & 0 \\
            0 & 0 & I_{k-a-2} & 0 \\
            0 & 0 & 0 & 1 \\
            0 & b' & v^{\mathrm{T}} & 0 \\
        \end{psmallmatrix}\,
        .
    \end{equation}
    Lastly, we restore the unit determinant to $h$
    \begin{equation}
        \begin{psmallmatrix}
            I_a & 0 & 0 & 0 \\
            0 & 0 & I_{k-a-2} & 0 \\
            0 & 0 & 0 & 1 \\
            0 & b' & v^{\mathrm{T}} & 0 \\
        \end{psmallmatrix}
        \rightarrow 
        \begin{psmallmatrix}
            I_{k-2} & 0 & 0 \\
            0 & x' & 0 \\
            0 & 0 & 1
        \end{psmallmatrix}
        \begin{psmallmatrix}
            I_a & 0 & 0 & 0 \\
            0 & 0 & I_{k-a-2} & 0 \\
            0 & 0 & 0 & 1 \\
            0 & b' & v^{\mathrm{T}} & 0 \\
        \end{psmallmatrix}
        \rightarrow 
        =
        \begin{psmallmatrix}
            I_a & 0 & 0 & 0 \\
            0 & 0 & I_{k-a-2} & 0 \\
            0 & 0 & 0 & x' \\
            0 & b' & v^{\mathrm{T}} & 0 \\
        \end{psmallmatrix}\,
        .
    \end{equation}
    The condition $\det G = 1$ now sets $b' = (-1)^{k+a+1} x'^{-1}$, leading to the representative
    \begin{equation}
        \begin{psmallmatrix}
            I_a & 0 & 0 & 0 \\
            0 & 0 & I_{k-a-2} & 0 \\
            0 & 0 & 0 & x' \\
            0 & (-1)^{k+a+1} x'^{-1} & v^{\mathrm{T}} & 0 \\
        \end{psmallmatrix}\,
        .
    \end{equation}
\end{proof}

\begin{rmk}
    Another way of parametrizing the coset $(\SL_{k}(F))_{\hat X} \bs \SL_k(F)$ is to parametrize the coset $\SL_k(F) / (\SL_{k}(F))_{\hat X}$ which works completely analogously to the how the coset $(\SL_{k}(F))_{\hat Y} \bs \SL_k(F)$ was parametrized in lemma \ref{lem:gammacoset} and then invert the resulting matrices.
\end{rmk}

\section{Levi orbits}
\label{app:levi-orbits}
Let $P_m$ for $1 \leq m \leq n-1$ be the maximal parabolic subgroup of $\SL_n$ associated to the simple root $\alpha_m$ with Levi decomposition $L U_m$ (where we drop the subscript for $L$ for convenience) and let $y \in {}^t u_m(F)$ be parametrised by a matrix $Y \in \Mat_{(n-m) \times m}(F)$ as in \eqref{eq:um-param}.
We will now study the $L(F)$-orbits of elements $y$.

We parametrise an element $l \in L(F)$ by the two matrices $A \in \Mat_{m\times m}(F)$ and $B \in \Mat_{(n-m)\times(n-m)}(F)$ with $\det(A) \det(B) = 1$ as $l = \diag(A, B^{-1})$.
This element acts by conjugation on $y$ as $ly(Y)l^{-1} = y(AYB)$. 
Using unit determinant matrices $A$ and $B$ we may perform standard row and column additions to put $Y$ on a form which has zero elements everywhere except for an anti-diagonal $r \times r$ matrix with non-vanishing determinant in the upper right corner where $0 \leq r \leq \min(m, n-m)$ is the rank of $Y$. 
This can be seen as follows.

If the upper right element is zero, pick any non-zero element whose row-column position we denote $(i, j)$ and add multiples of row $i$ to the first row, and column $j$ to the last column to make the upper right element non-zero.
Then, use the non-zero upper right element to cancel all remaining non-zero elements on the first row and last column by further row and column additions.
Repeat the procedure for the matrix obtained by removing the first row and last column.
The induction terminates when we run out of rows or columns, or when the remaining elements are all zero.

We will now rescale the anti-diagonal elements by conjugating $y$ with diagonal matrices $l$, meaning that the $i$th diagonal element in $l$ rescales both row $i$ and column $i$ (inversely).
Each rescaling of an element in the anti-diagonal of the $r \times r$ matrix then leaves two less diagonal elements in $l$ for further rescalings.
Since the non-zero elements of $Y$ at this stage do not share any rows or columns (because of the anti-diagonal $r \times r$ submatrix), we may then perform any and all rescalings until we run out of free diagonal elements in $l$. 
The number of free diagonal elements in $l$ is $n-1$ because of the determinant condition which means we can make $\floor{\frac{n-1}{2}}$ rescalings. We have that $1 \leq m \leq n-1$ and $r \leq \min(m, n-m) \leq \floor{\frac{n}{2}}$.  Thus, it is possible to rescale all anti-diagonal elements unless $n = 2r = 2m$, for which there will be one remaining anti-diagonal element $d$.
Using conjugations with $l = \diag(a, 1, \ldots, 1, 1/a)$, $a \in F^\times$ this $d$ can be shown to be in $F^\times / (F^\times)^2$. 

We have now shown that the $L_m(F)$-orbits of elements $y(Y) \in {}^t u_m(F)$ are characterized by matrices $Y = Y_r(d) \in \Mat_{(n-m) \times m}(F)$, where $Y_r(d)$ which is non-zero only for an anti-diagonal $r\times r$ matrix in its upper right corner whose elements are all one except the lower left which is $d$
\begin{equation}
    Y_r(d) = 
    \begin{pmatrix}
        \quad 0 &
        \begin{bsmallmatrix}
            & & & 1 \\
            & & \reflectbox{$\ddots$} & \\
            & 1 & & \\
            d & & & 
        \end{bsmallmatrix} \\[1.5em]
        \quad 0 & 0
    \end{pmatrix} \, .
\end{equation}
For $n = 2r = 2m$, $d \in F^\times / (F^\times)^2$ and otherwise $d = 1$.
For convenience we will denote $Y_r(1)$ as $Y_r$. 
Thus, the $L_m(F)$-orbits on ${}^t u_m(F)$ are characterized by the same data as the $\SL_n(F)$-orbits $([2^r1^{n-2r}], d)$ with $d \in F^\times/(F^\times)^k$, $k = \gcd([2^r1^{n-2r}])$ and $0 \leq r \leq \min(m, n-m)$.

By conjugations with the Weyl element $w \in \SL_n(F)$ mapping torus elements $(t_1, t_2, \ldots, t_n) \mapsto (t_1, \ldots t_{m-r}, \underline{t_m, t_{m+1}, t_{m-1}, t_{m+2}, \ldots, t_{m-r+1}, t_{m+r}},$ $t_{m+r+1}, \ldots t_n)$, where we have underlined the changed elements, we see that $y(Y_r(d))$ is put on the form of the standard representative for the $\SL_n(F)$-orbit $([2^r1^{n-2r}], d)$ shown in proposition~\ref{orbits}. 

In this paper we will always be able to find a representative on the form $Y_r(1)$, that is, to rescale all elements, since we consider $n \geq 5$ and $r \leq 2$ where the latter restriction comes from the fact that higher rank elements have vanishing associated Fourier coefficients in a next-to-minimal or minimal automorphic representation according to theorem~\ref{thm:ggsglobal}.

Lastly, we note that if we instead consider $L(\overline F)$-orbits the last remaining rescaling in the maximal rank $n=2r$ case would be possible by conjugation with $l = \diag(\sqrt{d}, 1, \ldots, 1, 1/\sqrt{d})$.

\newpage

\bibliography{SLn}
\bibliographystyle{utphys-alpha}

\end{document}